\documentclass[]{amsart}
\usepackage{graphicx, amsmath, amsthm, amsfonts, todonotes} 
\usepackage[]{geometry}
\usepackage{hyperref}
\usepackage{float}
\usepackage{colortbl}
\newtheorem{theorem}{Theorem}[section]
\newtheorem{lemma}[theorem]{Lemma}

\theoremstyle{remark}
\begingroup
\newtheorem{remark}[theorem]{Remark}

\endgroup 

\usepackage[ruled,noend]{algorithm2e}
\RestyleAlgo{ruled}
\SetKwComment{Comment}{// }{ }
\SetKwInput{Input}{Input}
\SetKwInput{Return}{Return}

\DeclareMathOperator{\argmin}{argmin}
\DeclareMathOperator{\Per}{Per}

\title[Momentum-based minimization of the Ginzburg-Landau functional]{Momentum-based minimization of the Ginzburg-Landau functional on Euclidean spaces and graphs}



\author{Oluwatosin Akande}
\address{Oluwatosin Akande\\
Industrial and Systems Engineering,
Lehigh University,
200 West Packer Avenue,
Bethlehem, PA 18015, USA
}
\email{oaa323@lehigh.edu}

\author{Patrick Dondl}
\address{
Patrick Dondl\\
Abteilung für Angewandte Mathematik,
Albert-Ludwigs-Universit\"at Freiburg,
Hermann-Herder-Straße 10,
79104 Freiburg i.\ Br., 
Germany 
}
\email{patrick.dondl@mathematik.uni-freiburg.de}

\author{Kanan Gupta}
\address{Kanan Gupta\\
University of Pittsburgh,
Department of Mathematics,
Thackeray Hall,
Pittsburgh, PA 15213, USA
}
\email{kanan.g@pitt.edu}

\author{Akwum Onwunta}
\address{
Akwum Onwunta\\
Industrial and Systems Engineering,
Lehigh University,
200 West Packer Avenue,
Bethlehem, PA 18015, USA
}
\email{ako221@lehigh.edu}

\author{Stephan Wojtowytsch}
\address{Stephan Wojtowytsch\\
University of Pittsburgh,
Department of Mathematics,
Thackeray Hall,
Pittsburgh, PA 15213, USA
}
\email{s.woj@pitt.edu}

\let\eps = \varepsilon  
\let \wto = \rightharpoonup
\let \Ra = \Rightarrow

\newcommand{\dist}{\mathrm{dist}}
\newcommand{\sdist}{\mathrm{sdist}}

\newenvironment{pde}{\left\{\begin{array}{rll} } {\end{array}\right.}

\renewcommand{\d}{\,\mathrm{d}}
\newcommand{\dr} {\d r}

\newcommand{\dt} {\d t}
\newcommand{\dx} {\d x}

\newcommand{\dz} {\d z}

\newcommand{\R}{\mathbb{R}}

\begin{document}
\begin{abstract}
We study the momentum-based minimization of a diffuse perimeter functional on Euclidean spaces and on graphs with applications to semi-supervised classification tasks in machine learning.
While the gradient flow in the task at hand is a parabolic partial differential equation, the momentum method corresponds to a damped hyperbolic PDE, leading to qualitatively and quantitatively different trajectories.
Using a convex-concave splitting-based FISTA-type time discretization, we demonstrate empirically that momentum can lead to faster convergence if the time step size is large but not too large.
With large time steps, the PDE analysis offers only limited insight into the geometric behavior of solutions and typical hyperbolic phenomena like loss of regularity are not be observed in sample simulations.
We obtain the singular limit of the evolution equations as the length parameter of the phase fields tends to zero by formal expansions and numerically confirm its validity for circles in two dimensions.
Our analysis is complemented by numerical experiments for planar curves, surfaces in three-dimensional space, and semi-supervised learning tasks on graphs.
\end{abstract}

\maketitle
\markboth{\shorttitle}{\shorttitle}


\section{Introduction}

From its inception as Dido's problem \cite[Book I]{aeneid}, the task of separating two regions with as short a boundary as possible is one of the oldest problems in mathematics. Many problems across the sciences are driven by the energetic imperative to minimize a perimeter or weighted perimeter functional, from crystal grain growth to rocks being ground to smooth pebbles in the sea and from liquids with surface tension to soap films and bubbles. The heuristic of minimizing a transition region in a suitable sense has been applied to graphs in the classical setting (min-cut problem) and more recently on weighted graphs in the setting of semi-supervised learning in data science.

In semi-supervised learning tasks, we have a large number of data points but only a small subset of them are labeled. Our goal is to label the remaining points. Two heuristics are common:

\begin{enumerate}
    \item {\em Proximity-based clustering.} We base new labels on the closest labeled point or points. The simplest instance of this heuristic is the $k$-nearest neighbors algorithm, which does not exploit the knowledge of where other unlabeled points are. A more advanced PDE-based version which integrates this information can be based on the eikonal equation \cite{dunbar2022models}.

    \item {\em Perimeter minimization clustering.} We try to assign consistent labels to clusters of data and change labels only in between clusters in regions of low data density (small weighted perimeter).
\end{enumerate}

The second approach leads to similar mathematical structures in data science as in the sciences. PDE-based algorithms are popular in machine learning due to their interpretability: While intuition may not transfer one-to-one to the new setting, such analogy can offer insight into the relative strengths and weaknesses of algorithms and even inspire methods to remedy such pitfalls \cite{calder2020poisson}.

Finding minimizers of perimeter functionals is not easy, neither analytically nor numerically. It has stimulated active research across disciplines from minimal surfaces in differential geometry to mean curvature flow (the $L^2$-gradient flow of the perimeter functional) in partial differential equations and to rigorous computational approximations in numerical analysis.

Direct numerical discretizations of an evolving interface between two regions are possible, but challenging even for surfaces in dimension three \cite{deckelnick2005computation, dziuk2007finite}. Mimicking extrinsic perspectives in geometric measure theory \cite{brakke2015motion}, extrinsic approaches to mean curvature flows which operate on (potentially smooth approximations of) characteristic functions of sets separated by an evolving boundary rather than the hypersurface itself have been proposed. These include the `thresholding' or Merriman-Bence-Osher (MBO) scheme \cite{merriman1992diffusion, evans1993convergence} and the Allen-Cahn equation. Both have a gradient-flow/minimizing movements structure with respect to a smooth approximation of the perimeter functional \cite{esedoglu2015threshold, laux2020thresholding}. Heuristically, the sharp jump between the two domains is `smeared out' across a narrow region in a principled way, leading to computationally stable methods which can easily accommodate topological transitions. Due to their stability, diffuse interface models have been popular both in mathematical modeling and computational works. Both the MBO scheme and the Allen-Cahn equation have been studied extensively on graphs for applications in image segmentation and semi-supervised learning \cite{luo2017convergence, merkurjev2018semi,bosch2018generalizing, bertozzi2019graph, cristoferi2019clustering, mercado2020node, bosch2020graph, budd2021classification, budd2023joint}.

Classical energy-driven approaches (MCF, MBO, Allen-Cahn) have emphasized using a gradient flow to minimize a diffuse perimeter functional, both on Euclidean spaces and graphs. In this work, we consider a momentum-based method or `accelerated gradient flow' instead. 

While gradient flows choose a locally optimal `descent' direction based on first order information about an objective landscape, momentum methods (also referred to as `accelerated gradient flows') retain information on past descent directions along their trajectory. Both can be seen as applications of Newton's second law, but while the (inertial) mass vanishes in gradient flow models, it is normalized as 1 in momentum methods. Due to inertia, the velocity does not change instantaneously to adapt to a new gradient. Integrating such global information allows them to converge faster in `favorable' landscapes where past information is indicative of future geometry. 

Rigorous guarantees for momentum methods outperforming gradient descent schemes are available mostly in convex optimization or in landscapes with convex-like properties \cite{gupta2024nesterov}. Qualitative differences can arise between finite-dimensional and infinite-dimensional tasks \cite{attouch2018fast, siegel2023qualitative}. Still, for non-convex optimization tasks in machine learning such as the training of neural networks, there is a large corpus of empirical evidence that momentum-based methods converge significantly faster than pure gradient descent methods -- see e.g.\ \cite{gupta2024nesterov} and the sources cited therein.

We note that (diffuse) perimeter minimization geometrically differs from most benchmarks in convex (and non-convex) optimization. For instance, compact initial surfaces vanish in finite time under mean curvature flow (the gradient flow of the perimeter functional), while gradient flows typically require infinite time to find minimizers in convex and strongly convex optimization tasks. This article is a curiosity-driven exploration into momentum methods for perimeter minimization: Are they (provably and/or empirically) outperforming gradient flow discretizations?

Our main contributions are as follows.

\begin{enumerate}
\item We introduce the `accelerated Allen-Cahn' equation and establish its elementary properties (total energy dissipation, conditional convergence to a minimizer).

\item We show that the accelerated Allen-Cahn equation is hyperbolic and inherits its geometric properties from the wave equation, in particular: finite speed of propagation and a lack of implicit regularization for evolving interfaces, compared to the parabolic mean curvature flow equation.

\item We demonstrate that it also inherits features from momentum-methods compared to gradient flows (such as `overshooting' a global minimizer due to inertia).

\item We formally deduce a geometric evolution equation for the hypersurfaces in the singular limit as the width of the diffuse interfaces is taken to zero.

\item We introduce two time discretizations of the `accelerated Allen-Cahn equation', both based on a stablizing convex-concave splitting, i.e.\ a splitting where the gradient of the `convex part' of the objective function is evaluated implicitly and the gradient of the `concave part' of the objective function is evaluated explicitly.

The {\em Convex-Implicit Nonconvex-Explicit Momentum Algorithm} (CINEMA) is a discretization which provably decreases the sum of kinetic and potential energy in every time step even for large time step sizes, but empirically it does not achieve acceleration over a standard gradient descent scheme with convex-concave splitting.

The second discretization is an instance of the {\em Fast iterative shrinkage and thresholding algorithm} (FISTA), also with convex-concave splitting. We do not guarantee it to be energy decreasing, but for moderately large time step sizes it empirically achieves substantial acceleration over a gradient descent scheme with the same splitting and {\em any} step size.

\item We demonstrate that both algorithms can be implemented with negligible excess computational cost over a mere gradient descent scheme. A key ingredient is the choice of a double-well potential whose convex part is quadratic, leading to a linear problem for the implicit part of the time step.

\item We compare momentum methods and local-in-time gradient descent methods numerically in Eudlidean spaces and on semi-supervised learning tasks on graphs.
\end{enumerate}

Oversimplifying, we can summarize: The `accelerated gradient flow' of the diffuse perimeter functional has potentially undesirable geometric properties. However, in a large time step discretization, the algorithm may achieve significantly faster convergence than a convex-concave splitting gradient descent scheme. The choice of a time discretization is crucial.

We are not considering the `accelerated Allen-Cahn equation' as a physical model, but only as a computational tool to find a (local) minimizer of the diffuse perimeter functional. Momentum-based methods have been previously considered in numerical analysis for solving obstacle problems \cite{schaeffer2018penalty, calder2019pde} with an $L^1$-penalty (which enforces the obstacle constraint exactly). 
Unlike the Allen-Cahn equation, these tasks fall into the realm of convex optimization where FISTA is provably faster (at least in the sense of minmax rates).
Non-convex problems have been considered under the name of `second order flows' in \cite{chen2024second} for Allen-Cahn like evolutions, in the context of image processing \cite{dong2021class} and the bending of thin elastic structures \cite{bonito2023numerical, bonito2024finite, dong2024accelerated, dong2024bdf}.

Momentum methods have also recently gained traction in a different PDE setting. Namely, particles with non-vanishing inertia subject to a potential force, friction, and a stochastic diffusion force follow a phase-space SDE
\[
\begin{cases} \d X_t &= V_t \dt\\ \d V_t &= - \big(\nabla U(X_t) + \alpha V_t\big)\dt + \sqrt{2}\sigma\,\d B_t\end{cases}.
\]
Their law obeys associated hypo-coercive PDE
\[
\partial_t\pi = div \left(\pi \begin{pmatrix} v\\ - \nabla U(x) - \alpha v\end{pmatrix}\right) + \sigma^2\,\Delta_v \pi,
\]
which can be interpreted as a `Hamiltonian flow' in Wasserstein space \cite{ambrosio2008hamiltonian, park2024variational}. As time approaches infinity, $\pi$ converges to the invariant measure with density proportional to $\rho^*(x,v) = \exp\left(- U(x) + \frac{\|v\|^2}{2\,\sigma^2}\right)$. These dynamics be used in place of the more classical overdamped Langevin-dynamics
\[
\d X_t = -\nabla U(X_t)\dt + \sqrt 2 \d B_T, \qquad \partial_t\pi = div(\pi\,\nabla U) + \Delta \pi
\]
for sampling, and in fact $\pi(t)$ converges to $\rho^*$ faster (in $\chi^2$-divergence) if the invariant measure $\rho \sim \exp(- U)$ satisfies a Poincar\'e inequality \cite{cao2023explicit}. Functional inequalities and log-concavity are the analogue of convexity-type assumptions in the setting of optimization.

The article is organized as follows. In Section \ref{section background}, we give more context for the mathematical models underlying this work: The Ginzburg-Landau (diffuse perimeter) functional, momentum-based first order optimization, partial differential equations on graphs, and semi-supervised learning. In Section \ref{section allen-cahn and acceleration}, we analyze the accelerated gradient flow and gradient flow of the Ginzburg-Landau energy on Euclidean spaces and compare their geometric properties. In Section \ref{section numerical approach}, we derive discrete time algorithms for the numerical solution of the evolution equations, which we use for numerical experiments both in Euclidean spaces and on graphs in Section \ref{section numerical experiments}. We conclude with a brief reflection in Section \ref{section conclusion}.

\section{Background Material}\label{section background}

\subsection{Perimeter minimization and the Ginzburg-Landau functional}\label{section ginzburg-landau}

Many physical phenomena are driven by `perimeter minimization', i.e.\ by the imperative to minimize the length (or area) separating two different `phases'/open sets in $\R^2$ (or $\R^3$). Mathematically, this is generally described as minimizing
\[
\Per(E) = \int_\Omega \|\nabla 1_E\|\dx = \sup\left\{\int_E \mathrm{div}(\phi)\dx : \phi \in C^\infty(\Omega; \R^d), \:\:\|\phi\|_{L^\infty}\leq 1\right\}
\]
where $1_E$ is the indicator function of the set $E$ and $\Omega$ is a larger containing set. Since $1_E$ is non-smooth, the integral has to be understood in the sense of functions of bounded variation on the right \cite{giusti1984minimal}, which leads to analytic and numerical challenges. 

A computationally stable approximation is the Ginzburg-Landau functional (also sometimes referred to as Modica-Mortola energy)
\[
\Per_\eps(u) = \int_\Omega\,\frac\eps2\,\|\nabla u\|^2 + \frac{W(u)}\eps\dx
\]
where $W$ is a `double-well potential': A non-negative function which takes the value zero only if $u\in\{0,1\}$, for instance $W(u) = u^2(1-u)^2$. If $\eps$ is small, then $u$ takes values very close to $0$ or $1$ on most of the domain. However, the transition between the potential wells $0,1$ cannot happen arbitrary quickly due to the presence of the squared gradient. Both contributions to the energy balance when $u$ transitions between being close to $0$ and close to $1$ on a length scale $\sim \eps$. Indeed, $\Per_\eps$ converges to (a $W$-dependent multiple of) $\Per$ as $\eps\to 0^+$ (in the sense of $\Gamma$-convergence). Depending on the boundary conditions for $u$, the limit may also be a perimeter relative to the set $\Omega$ -- we refer to \cite{modica1977esempio, modica1987gradient, ambrosio2000variational} for details. 

Rather than the characteristic function of the set $E$, which jumps along the boundary $\partial E$, we encounter functions of the form
\[
u_\eps(x) = \phi\left(\frac{\sdist_E(x)}\eps\right)
\]
when studying $\Per_\eps$. Here $\sdist(x) = \dist(x, E^c) - \dist(x, E)$ is the {\em signed distance function} from $\partial E$, taken to be positive inside of $E$, and $\phi$ is the optimal transition between the potential wells at $0$ and $1$ in one dimension, i.e.\ the monotone increasing function which balances the contributions to the energy: $(\phi')^2 = W(\phi)$. Under mild assumptions, there is a unique solution to this ODE such that $\phi(0) = 1/2$ and $\lim_{x\to\infty}\phi(x) = 1, \:\lim_{x\to-\infty}\phi(x) = 0$.

By differentiation, we see that the `optimal profile' $\phi$ satisfies $\phi'' = W'(\phi)$. 
The transition between $0$ and $1$ is `smeared out' across an area of width $\sim \eps$ with the characteristic shape $\phi$. For further details, see also Appendix \ref{appendix corrector}. 

We recall that $\sdist$, like the regular distance function, has a unit gradient: $\|\nabla \sdist\|\equiv 1$ wherever the distance function is smooth. By standard results, $\sdist$ is always $C^2$-smooth close to a $C^2$-boundary \cite[Chapter 14.6]{gilbarg1977elliptic} and even if $\partial E$ is non-smooth, it is differentiable except on a set of measure zero by Rademacher's theorem.

\subsection{The Allen-Cahn Equation}
The Allen-Cahn equation
\[
\eps\,\partial_tu = \eps\,\Delta u - \frac{W'(u)}\eps
\]
is the (time-normalized) $L^2$-gradient flow of the Ginzburg-Landau functional. Different boundary conditions imposed on the Ginzburg-Landau energy -- which correspond to different limits as $\eps\to 0^+$ -- correspond to different boundary conditions for the PDE.

\subsubsection{Singular limit}\label{section singular limit allen-cahn}
Like $\Per_\eps$ converges to a perimeter functional, also solutions to the Allen-Cahn equation (i.e.\ the $L^2$-gradient flow of $\Per_\eps$) converge to solutions of Mean Curvature Flow (i.e.\ the $L^2$-gradient flow of $\Per$) in a suitable sense. Heuristically, this can be reasoned out as follows:

If $u_{\eps,0} : \Omega\to [0,1]$ is the initial condition for the Allen-Cahn equation, then $0 < u_\eps(t,x) < 1$ for all $t>0$ and $x\in \Omega$ by the maximum principle (unless the boundary conditions on $\partial\Omega$ force us to leave the interval). In particular, we can write
\[
u_\eps(t,x) = \phi\left(\frac{r_\eps(t,x)}\eps\right), \qquad r_\eps = \eps \,\phi^{-1}(u_\eps).
\]
In $r_\eps$, the Allen-Cahn operator can be expressed as
\[
(\partial_t-\Delta)u_\eps + \frac{W'(u_\eps)}{\eps^2} 
\:=\:
\frac{\phi''(r_\eps/\eps)}{\eps^2}\big(1- \|\nabla r_\eps\|^2\big) + \phi'\left(\frac{r_\eps}\eps\right) \,\big(\partial_t-\Delta\big)r_\eps.
\]
It is in general not possible to simultaneously make both terms zero, i.e.\ to simultaneously solve $\|\nabla r_\eps\|^2 =1$ and $(\partial_t-\Delta)r_\eps=0$. If $r_\eps\to r$ for some limiting $r$, then we expect that the coefficient of $1/\eps^2$ has to be zero, while there is a bit of leeway for the $O(1)$ term. Thus, we expect that $\|\nabla r\| \equiv 1$ everywhere, suggesting that $r$ should be a signed distance function in the spatial coordinates. This does not tell us anything about the time evolution of $r$ and the interface. To find it, we posit that the second PDE $(\partial_t -\Delta)r$ is solved on the interface $r=0$, i.e.\ in the place where $\phi'$ is largest. 

It is well known that $\mathrm{div}(\nabla v/ \|\nabla v\|)(x)$ is the mean curvature of the level set $\{z : v(z) = v(x)\}$ at $x$ for any smooth function $v$ -- see e.g.\ \cite[Section 2.1]{evans1991motion}. If we are correct and $\|\nabla r\| \equiv 1$, then $\Delta r$ is the mean curvature of the interface 
\[
I(t) = \{x : r(t,x) = 0\} = \left\{ x : \phi\left(\frac{r(t,x)}\eps\right) = 1/2 \right\}.
\]
A proof in the context of distance functions is also given in \cite[Chapter 14.6]{gilbarg1977elliptic}.

The meaning of $\partial_tr$ is easiest to glean in the setting of moving hyperplanes. Namely, if $H(t) = \{ x : x_n= x_n^0 + vt\}$, then the signed distance function is $\sdist = (x_n^0+vt)- x_n$ and the normal velocity is $v = \partial_tr$ (where $v>0$ corresponds to the area where $\sdist>0$ expanding in time). The same is true more generally, as can be derived easily in locally adapted coordinates.

We thus conjecture that $r_\eps$ is (close to) the signed distance function from an interface $I(t)$ which moves by mean curvature. Far away from the interface, we conjecture that `nothing happens', i.e.\ $u_\eps$ remains almost constant in space and time close to the potential wells. Overall, $u_\eps$ would then be close (e.g.\ in $L^2$) to the characteristic function of a set $E(t)$ whose boundary moves by mean curvature flow.

The heuristic explanation above in fact describes the limiting behavior of the Allen-Cahn equation. Rigorous proofs of this result in various forms are given in \cite{ilmanen1993convergence, mugnai2011convergence, fischer2020convergence}.

\subsubsection{Vector-valued extension}\label{section vector-valued}
In many applications, there may be more than two phases. If the phases are unordered (i.e.\ phase 1 can border phase 3 and does not have to pass through phase 2), we have to design Ginzburg-Landau type functionals like for example
\[
F_\eps^{GL}(u) = \int_\Omega \frac\eps2\,\|Du\|^2_F + \frac{W(u)}\eps \dx
\]
for functions $u:\Omega\to\R^k$ where $k\geq 3$ is the number of classes, $\|Du\|_F$ is the Frobenius norm of the derivative matrix and $W$ is a potential with $k$ wells, usually selected at the unit vectors $e_1, \dots, e_k$, unless prior information suggests that boundaries between different phases should have different `surface tension'. 

An easy way to create such a multi-well potential is to select a double-well potential $W_{1D} :\R\to\R$ which vanishes only at $0,1$ and set
\[
W:\R^k\to [0, \infty], \qquad W(u) = \sum_{i=1}^k W_{1D}(u_i) + \lambda \cdot \left(1- \sum_{i=1}^k u_i\right)^2
\]
for some $\lambda\in (0, \infty]$. The first contribution to the potential $W$ vanishes if and only if $W_{1D}(u_i) = 0$ for all $i$, i.e.\ if and only if $u_i \in \{0,1\}$ for all $i$. The second term ensures that exactly one of the $u_i$ is $1$ and the others are $0$.

Also solutions to the vector-valued Allen-Cahn equation approach (multi-phase) mean curvature flow as proved recently in \cite{laux2018convergence, fischer2024quantitative}.

\subsubsection{The hyperbolic Allen-Cahn equation}

Another modification of the classical Allen-Cahn equation is its hyperbolic version 
\[
\tau\,u_{tt} + \alpha \,u_t = \Delta u + \frac{W'(u)}{\eps^2}.
\]
which has been considered more recently by several authors with parameters scaling as $\tau/\alpha = O(\eps)$. Like in the parabolic case, the limiting dynamics are described by mean curvature flow. For details, see e.g.\ \cite{nizovtseva2016hyperbolic, folino2016metastable, folino2017slow}.

The `accelerated Allen-Cahn equation' considered in this work is given by the same equation, but in a parameter regime where $\tau$ and $\alpha$ are both of order $1$. Its behavior is therefore quite different and borrows more from the hyperbolic world of PDEs. It has previously been considered in \cite{chen2024second} from the perspective of numerical analysis. The authors construct two stable numerical schemes, one of which is second order accurate in time, and use the discrete time approximation to prove existence in continuous time. By contrast, our focus lies on the geometric properties of the equation and potential applications in semi-supervised learning.

\subsection{Momentum-based first order methods in optimization}

Gradient flows reduce an objective function by adjusting the function inputs towards a locally optimal `steepest descent' direction. If the objective function has favorable geometric properties -- for instance, in convex optimization -- it is possible to integrate more global geometric information gained throughout the optimization process to achieve faster convergence. This is the rationale behind both conjugate gradient methods in numerical linear algebra and momentum-based optimizers such as Nesterov's algorithm \cite{nesterov1983method} and FISTA \cite{beck2009fast} in convex optimization.

For instance, if $f$ is a convex objective function which has a minimizer and a Lipschitz-continuous gradient, then gradient descent achieves a decay of $f(x_t^{GD}) - \inf f = O(1/t)$ in $t$ steps while Nesterov's method achieves the much faster decay $f(x_t^{Nest}) - \inf f = O(1/t^2)$ with constants of comparable size hidden in the Landau notation. Non-asymptotic lower bounds demonstrate that Nesterov's method achieves optimal decay, at least up to a constant factor -- see e.g.\ \cite{nesterov2018lectures} for precise statements.

For information from previous time-steps to remain useful, the objective function $f$ must have favorable geometric properties such as convexity or strong convexity. Recently, multiple works have relaxed the geometric conditions under which faster convergence can be established in several directions -- see e.g.\ \cite{gupta2024nesterov, hermant2024study} and the references cited therein. While many realistic optimization tasks do not fall into classes where faster convergence for momentum methods can be guaranteed, the use of momentum often leads to a notable improved in practice (for instance in the training of neural networks).

Both momentum methods and gradient flows are derived from Newton's second law of mechanics 
\[
m\ddot x = - \alpha \,\dot x - \nabla f(x)
\]
for a particle of mass $m$ under the influence of a linear friction and a potential force $-\nabla f$. Gradient flows formally correspond to $m=0, \alpha =1$ while momentum methods correspond to positive mass $m=1$. The optimal coefficient of friction $\alpha$ depends on the geometry of $f$: For $\mu$-strongly convex functions, we select (an estimate of) $2\sqrt\mu$ while for general convex functions, $\alpha = \alpha(t) = 3/t$ is in fact a function of time.

\subsection{Momentum-based time stepping algorithms in convex optimization}\label{section convex algorithms}

Notably, the choice of time discretization for the heavy-ball dynamics is crucial. Polyak's \cite{polyak1964some} heavy-ball method
\[
x_{n+1} = x_n -\alpha\,\nabla f(x_n) + \beta(x_n-x_{n-1})
\]
corresponds to the symplectic Euler discretization
\[
\begin{cases}
x_{n+1} &= x_n + h\,v_n\\ v_{n+1} &= \rho v_n - h \,\nabla f(x_{n+1}),
\end{cases}
\qquad
h = \sqrt{\alpha}, \quad \rho = \beta, \quad v_n = \frac{x_n-x_{n-1}}{\sqrt\alpha}
\]
of the heavy ball ODE and generally does not achieve acceleration (or even fails to converge) in situations where Nesterov's scheme \cite{nesterov1983method} remains stable \cite{lessard2016analysis, goujaud2023provable}. In situations where a convex function can be decomposed into a smooth part and a `simple' part, the {\em fast iterative shrinkage and thresholding algorithm} (FISTA), which treats the smooth part explicitly and the simple part implicitly, is a stable algorithm with provable acceleration \cite{beck2009fast}.

Let $H$ be a Hilbert space. Consider the task of minimizing the sum of two functions $F+G: H\to \R$. If $F, G$ are both convex, then Nesterov's method and FISTA are two well-studied methods which are both based on numerical discretizations of the `heavy ball ODE' $\ddot x = -\alpha(t)\,\dot x - \nabla (F+G)(x)$. Both can be written in the form
\begin{equation}\label{eq our algorithm}
x_{n+1} = x_n + \tau v_n - \eta g_n, \qquad v_{n+1} = \rho_n\big(v_n - \tau g_n\big)
\end{equation}
where $g_n$ is an approximation of the gradient. The general scheme \eqref{eq our algorithm} encompasses
\begin{enumerate}
    \item Nesterov's method with
    \[
    g_n = \big(\nabla F + \nabla G\big)\big(x_n + \tau v_n\big)
    \]
    and parameters $\eta = \tau^2$ and $\rho_n = \frac n{n+3}$ (convex case, corresponding to $\alpha(t) = 3/t$) or $\rho_n = \frac{1-\sqrt\mu\,\tau}{1+\sqrt\mu\,\tau}$ ($\mu$-strongly convex case, corresponding to $\alpha = 2\sqrt\mu$). 

    The scheme converges if $F+G$ is convex and $\tau\leq 1/\sqrt L$ where $L$ is the Lipschitz-constant of $\nabla(F+G)$. The rate of decay $(F+G)(x_n)\to 0$ is faster than the GD rate in the sense of minimax optimal rates.
    
    \item FISTA with
    \[
    g_n = \nabla F(x_{n+1}) + \nabla G\big(x_n + \tau v_n\big)
    \]
    and parameters $\eta = \tau^2$ and $\rho_n = \frac n{n+3}$ (convex) or  $\rho_n = \frac{1-\sqrt\mu\,\tau}{1+\sqrt\mu\,\tau}$ ($\mu$-strongly convex). 
    
    The scheme converges (at a rate faster than GD) if $F,G$ are convex and $\tau\leq 1/\sqrt L$ where $L$ is the Lipschitz-constant of $\nabla G$ (which may be finite even if $\nabla F$ is discontinuous).
\end{enumerate}

Neither scheme is guaranteed to decrease the `total energy'
\[
e_n = (F+G)(x_n) + \frac{\lambda}2\,\|v_n\|^2
\]
monotonically in non-convex optimization for a suitable $\lambda$ depending on $\tau, \eta, \rho$, and even for very small time steps, we found that FISTA increased $e_n$ in numerical experiments in certain time steps when we split $F_\eps^{GL}$ into its convex part $F$ and concave part $G$ (details below).


\subsection{Convex-concave splitting}\label{section convex concave}

If $F$ is convex and $G$ is concave, there is an effective version of the (momentum-less) gradient descent scheme for the function $F+G$.
Namely, for given $x_n$ the next iterate $x_{n+1}$ is the {\em unique} minimizer of the strongly convex function 
\[
f_{x_n, h}^{aux}(z) := \frac12\,\|z-x_n\|^2 + h \big(F(z) + G(x_n) + \langle \nabla G(x_n), z-x_n\rangle\big).
\]
With this scheme, the sequence $(F+G)(x_n)$ is monotone decreasing independently of the step size -- see e.g.\ \cite{bartels2015numerical, related_splitting_paper}. The scheme coincides with the ISTA scheme (the momentum-less version of FISTA, see e.g.\ the sources in the introduction of \cite{beck2009fast}), but crucially, $G$ is assumed to be concave here, while both $F$ and $G$ are convex in FISTA. Convergence can be proved in both settings: In the convex-concave splitting, geometric conditions suffice. In the purely convex setting, we must assume that $\nabla G$ is Lipschitz-continuous and that $h\leq 2/ [\nabla G]_{Lip}$.

The assumptions on $F, G$ can be relaxed somewhat in terms of regularity, and the scheme may be defined on a Hilbert space rather than $\R^d$. These extensions are indeed necessary for the Allen-Cahn equation. We will pursue them below, where we study a momentum scheme of the type \eqref{eq our algorithm} with convex-concave splitting in the gradients.

The scheme is particularly convenient to implement if $F(x) = \frac12\,x^TAx$ is quadratic and amounts to solving the linear system
\[
0 = \nabla_z f^{aux}_{x_n, h}(z) = (z-x_n) + h\,\big(Az + \nabla G(x_n)\big)\qquad \Ra \quad
\big(I + hA\big)x_{n+1} = x_n - h\,\nabla G(x_n).
\]
We will pursue this simplification in \ref{section ws}.

While arbitrarily large gradient descent steps are formally possible with the convex-concave splitting scheme, we demonstrate in \cite{related_splitting_paper} that they generally do not result in large steps in practice. The scheme achieves its stability at the cost of `freezing' the transition regions of $u$ in space. Pairing the stability of convex-concave splitting with momentum-based acceleration is therefore a natural goal.

\subsection{Partial differential equations on graphs}

A graph $\Gamma = (V,E)$ is a tuple of two sets: The set $V$ of vertices and the set $E$ of edges, where an edge $e=\{v, v'\}$ is a set containing two distinct vertices $v, v'\in V$. We say that $e=\{v,v'\}$ {\em connects} $v$ and $v'$.

In the following, we will assume that each edge has a weight $w_e \in[0, \infty)$ which is large if its vertices are `close' (the edge is short) and small if the vertices are `far apart' (the edge is long). We can think of $w_e$ as a reciprocal length or a reciprocal length squared.

The graphs we consider have a finite set $V= \{v_1, \dots, v_n\}$ of vertices and two vertices can only be connected by at most one edge. No vertex is connected to itself by an edge. Note that our graphs are not directed, i.e.\ an edge is a set, not a tuple, and $v$ is connected to $v'$ by an edge if and only if $v'$ is connected to $v$ by the same edge. Also the weight $w_e$ depends only on the edge, not on its orientation.

If $e$ connects the vertices $v_i, v_j$ in $V = \{v_1, \dots, v_n\}$, we also denote $w_e = w_{ij} = w_{ji}$. A function $u:V\to\R^k$ can also be interpreted as a vector $u_i = u(x_i)$.

In analogy to the Euclidean Dirichlet energy $E_{DR}(u) = \frac12\int \|\nabla u\|^2\dx$, one can define the graph Dirichlet energy
\[
E_{DR}(u) = \frac12 \sum_{e\in E} w_e \,\big|u_i -u_j\big|^2 = \frac14 \sum_{i, j=1}^n w_{ij}\,\big|u_i-u_j\big|^2
\]
where $w_{ij}=0$ if $v_i, v_j$ are not connected by an edge. The $L^2$-gradient of the Dirichlet energy is the (negative) Laplacian, so we call the gradient
\[
(Lu)_k = \partial_{u_k} E_{DR}(u) = \frac12\sum_{i,j} w_{ij}(u_i-u_j) (\delta_{ik}-\delta_{jk}) =  \sum_j w_{kj} \,(u_k - u_j)
\]
of the Dirichlet energy the `graph Laplacian.' Like the classical (negative) Laplacian, also the graph-Laplacian is symmetric and positive semi-definite and, since
\[
\langle Lu, u\rangle = \sum_i u_i\sum_j w_{ij} \,(u_i - u_j) = \frac12\left(\sum_{i, j} w_{ij}(u_i-u_j)u_i - \sum_{i, j} w_{ij}(u_i-u_j)u_j\right) = 2\,E_{DR}(u),
\]
it is positive definite on the $\ell^2$-orthogonal complement of the functions which are constant on the connected components of the graph.\footnote{\ Two vertices $v, v'$ are in the same connected component if there exists a `path' $v_{i_1}, \dots, v_{i_m}$ such that $v_{i_1} = v$, $v_{i_m} = v'$ and $e_{i_li_{l+1}}>0$.}

We note that $L$ is represented by the matrix with entries $L$ with entries $L_{ij} = -w_{ij}$ for $i\neq j$ and $L_{ii} = \sum_{j\neq i}w_{ij}$ or $L = D-W$ where $W$ is the matrix with entries $w_{ij}$ if $i$ and $j$ are connected by an edge and zero if they are not, and $D$ is the diagonal matrix with entries $d_{ii} = \sum_{j=1}^n d_{ij}$.

The graph-Laplacian introduced here is unnormalized, but symmetric positive definite. We note that there are two other common notions of a `graph-Laplacian': 

\begin{enumerate}
\item The symmetric normalized graph Laplacian compensates for a high variation in the number of edges containing a node (the `degree' of nodes).  In matrix form, it is given by $L_{sn} = D^{-1/2} L D^{-1/2} = I - D^{-1/2}WD^{-1/2}$. 

\item The `random walk graph Laplacian' $L_{rw} := D^{-1}L = I - D^{-1}W$ is the generator of a random walk on $V$. It is usually not symmetric since the probability of jumping from $v_i$ to $v_j$ may be different from the probability of jumping from $v_j$ to $v_i$: The probabilities to jump from $v_i$ to $v_j$ in the next step have to sum up to $1$ over $j$ (since there are no other options), but there is no reason that they should sum up to $1$ over $i$. This Laplacian generalizes the link of the classical Laplacian to Brownian motion rather than its variational properties.
\end{enumerate}

For variational problems -- such as minization of the Ginzburg-Landau energy -- the unnormalized or symmetric normalized graph Laplacians are suitable. We select the unnormalized graph Laplacian as we desire constant functions to have zero energy $u^TLu$ (in particular, the constant functions that take values in the potential wells). We note however that the normalized graph Laplacian would be an admissible choice for a non-negative `energy' and may have other advantages.

\subsection{Semi-supervised learning}

Consider a data space $\mathcal X$ and label space $\mathcal Y$.
In semi-supervised learning applications, we receive a set of data points $S = \{x_i \in \mathcal X: i = 1, \dots, N\}$ and a collection of labels $\{y_i \in\mathcal Y: i\in I\}$ for a subset of the data. In general, the cardinality of the labeled set is much smaller than that of the dataset. Our task is to find a `good' function $u : S \to \mathcal Y$ such that $u(x) = y$ for a large proportion of $x\in S$, although we have no knowledge of the labels for the majority of points.

To decide what a `good' function $u$ is, we require further information. In general, we assume that a notion of similarity or distance between data points is available, which is meaningful in the sense that similar data points can generally be expected to have similar labels. We can use this notion to construct a graph with vertex set $V = S$ in which two nodes are connected by an edge if they are `similar enough'. 

Our applications fall into a common framework: We assume that $\mathcal X = \R^d$ for some $d$ and $\mathcal Y = \{1, \dots, k\}$ for some $k\in\mathbb N$. We either connect two vertices $x_i, x_j \in V = S$ by an edge if their Euclidean distance $\|x_i - x_j\|$ is below a cut-off length (which may be infinite), or we connect every point to its $K$ nearest neighbors for some $K\in \mathbb N$ (where the symmetry of the graph means that some vertices have degree $> K$ if they are the nearest neighbor to more than $K$ vertices). In this work, we either assign the weight $1$ to all edges, or we choose Gaussian edge weights
\[
w_{ij} = \exp\left(-\frac{\|x_i-x_j\|^2}{2\sigma^2}\right)
\]
for the edge connecting $x_i, x_j$ for a suitable $\sigma$. Of course, more advanced constructions are possible and, at times, required. 

The well-developed Allen-Cahn framework can be leveraged in semi-supervised learning applications by minimizing the `energy'
\[
F_\eps^{GL} : (\R^k)^N \to [0, \infty),\qquad
F_\eps^{GL}(u) = \frac1N\left( \frac\eps2\,\,\mathrm{tr}\big(u^TLu\big) + \frac{1}{\eps}\sum_{i=1}^N W(u_i)\right)
\]
subject to the (Dirichlet) `boundary condition' $u_i = \vec y_i$ if $i\in I$. The trace in the gradient term applies for labels in $\R^k$ with $k\geq 1$ since $u^TLu\in \R^{k\times k}$.

The vector $\vec y_i$ is the one-hot encoding of the label $y_i \in \{1, \dots, k\}$ (i.e.\ the vector which has a one in the $y_i$-th coordinate and zeros in all others). Naturally, we identify $u:S\to\R^k$ and $u\in \R^{N\times k}$. As in the Euclidean setting, minimizing $F_\eps^{GL}$ (approximately) corresponds to minimizing the size of the transition set as measured by the sum over edges between differently labeled points
\[
\frac1{2N}\,\sum_{i,j} w_{ij} 1_{\{u(x_i) \neq u (x_j)\}}.
\]
The energy $E$ is non-convex and in general, there exist many (local) minimizers. Using a gradient flow to minimize $F_\eps^{GL}$ corresponds to solving the graph Allen-Cahn equation \cite{bosch2018generalizing, bertozzi2019graph, bosch2020graph, mercado2020node}. In this article, we investigate the use of momentum methods for the same purpose.

\subsection{Ginzburg-Landau minimization and linear Laplacian methods}

\begin{figure}
\begin{tikzpicture}[scale=1.5]
\filldraw[fill = black!30, domain = 0 : 360, samples = 200] plot ( {1.3*cos(\x)}, {.5*sin(\x)} );

\draw[ultra thick, red, domain = -30 : 30, samples = 200] plot ( {1.3*cos(\x)}, {.5*sin(\x)} );
\draw[ultra thick, blue, domain = -30 : 30, samples = 200] plot ( {-1.3*cos(\x)}, {.5*sin(\x)} );

\node[left, blue] at (-1.3, 0){$u=-1$};
\node[right, red] at (1.3, 0){$u=1$};

\draw[very thick, dashed] (0, -.6) -- (0, .6);
\draw[very thick, dotted, purple] (1.1, -.4) -- (1.1, .4);
\draw[very thick, dotted, purple] (-1.1, -.4) -- (-1.1, .4);

\node at(0, -1.1){};
\node at(0, 1.1){};

\end{tikzpicture}\hfill
\vspace{-2.75cm}

\hfill
\begin{tikzpicture}[scale=1.5]
\draw[black!30](-2,0)--(2,0);
\filldraw[fill = black!30, domain = -2:2, samples = 100] (-2,0) -- plot ({\x}, {1/(\x*\x+1) + \x*\x/8 -.5}) -- (2,0);
\filldraw[fill = black!30, domain = -2:2, samples = 100] (-2,0) -- plot ({\x}, {-1/(\x*\x+1) - \x*\x/8 +.5}) -- (2,0);

\draw[ultra thick, color = red](2, .2)-- (2, -.2);
\draw[ultra thick, blue](-2, .2)-- (-2, -.2);
\node[left, blue] at (-2, 0){$u=-1$};
\node[right, red] at (2, 0){$u=1$};

\draw[very thick, dashed] (0, -.6) -- (0, .6);
\draw[very thick, dotted, purple] (1.35, -. 2) -- (1.35, .2);
\draw[very thick, dotted, purple] (-1.35, -. 2) -- (-1.35, .2);
\end{tikzpicture}
\caption{
\label{figure linear vs non-linear}
The function $u$ has a boundary condition $u=-1$ on the left boundary segment, $u=1$ on the right boundary segment, and no condition on the remainder of the boundary (i.e.\ only part of the boundary is `labeled'). The dashed black line illustrates the cut location between classes for any $p$-Laplacian energy with $p\in (1, \infty)$. The dotted purple lines indicate where a minimizer of the Ginzburg-Landau energy could separate classes (for small enough $\eps>0$). 
}
\end{figure}

The Ginzburg-Landau energy formalizes a trade-off: We seek functions $u$ with low Dirichlet energy $\frac12\int_\Omega\|\nabla u\|^2\dx$ which are close to $\pm 1$ on the majority of the domain $\Omega$. The precise balance is governed by the length-scale parameter $\eps>0$. 

In binary classification applications, we designate the set where $u>0$ as `class $1$' and the set where $u<0$ as `class $-1$' (with some tie-break mechanism for $u=0$). 
Since we are thresholding $u$ for classification, one may ask whether the condition that $|u|\approx 1$ before thresholding makes any difference, or whether we could just minimize a $p$-Laplacian energy $\int_\Omega |\nabla u|^p\dx$ without any double-well term. The $p$-Laplacian energy is convex for any $p\geq 1$, strictly convex for $p\in(1,\infty)$ and quadratic for $p=2$. Finding minimizers is therefore much easier computationally. This is the paradigm of Laplace learning and its refined version, Poisson learning \cite{calder2020poisson}.

There are situations where the presence of the double-well term leads to vastly different geometric effects. Namely, if we are minimizing a strictly convex energy of the gradient, there exists a unique energy minimizer. If the domain and boundary conditions are symmetric -- consider e.g.\ the examples in Figure \ref{figure linear vs non-linear} -- then the unique minimizer exhibits the same symmetry. In both examples, we separate classes vertically -- which is sensible -- but the decision boundary is in fact the {\em longest} vertical cut. 

By comparison, the non-convex Ginzburg-Landau energy has a symmetric set of minimizers, but the individual minimizers may not have the same symmetries as the data. Instead, we can think of minimizing the perimeter of the classes as a heuristic for small $\eps>0$. Rather than switching the classifier in the middle, we switch in either of the two narrow regions. By symmetry, there are two energy minimizers in this situation with vastly different class assignments. The Poisson-MBO scheme is a similarly non-convex extension of Poisson learning \cite{calder2020poisson}.

A similar counterexample can be constructed with a finite data graph, and the observation remains stable under small perturbations, including those that break symmetry. Which heuristic seems more appropriate and yields higher performance may depend on the situation.

\section{\texorpdfstring{The Allen-Cahn Equation and Accelerated Allen-Cahn Equation on $\R^d$}{The Allen-Cahn Equation and Accelerated Allen-Cahn Equation on Euclidean Spaces}}\label{section allen-cahn and acceleration}

\subsection{Basic properties}\label{section pde basics}

For the sake of convenience, we assume that the doublewell potential $W:\R\to[0,\infty)$ is $C^1$-smooth, vanishes only at $0$ and $1$, and only grows quadratically at $\infty$. This excludes the popular prototype $W(u) = u^2(1-u)^2$, but it includes for instance the potentials used in our simulations. It could be generalized with little additional effort, but theorem statements would become more involved. We define
\[
F_\eps^{GL}: L^2(\Omega)\to [0, \infty], \qquad F_\eps^{GL}(u) = \begin{cases}
    \int_\Omega \frac\eps2\,\|\nabla u\|^2 + \frac{W(u)}\eps\dx &\text{if }u\in H^1(\Omega)\\ +\infty &\text{else}.
\end{cases}
\]
The choice of defining $F_\eps^{GL}$ on the larger space $L^2$ provides a notion of dissipation both for the gradient flow and the momentum method. For an initial condition $u_0\in H^1(\Omega)$, the accelerated and time-normalized $L^2$-gradient flow of $F_\eps^{GL}$ is given by the evolution equation
\[
\begin{pde}
    (\partial_{tt} + \alpha \partial_t)u &= \Delta u - \frac1{\eps^2}\,W'(u) & t>0, x\in \Omega\\
    u &= u_0 &t=0\\
    \partial_tu &= 0 &t=0\\
    \partial_\nu u &=0 & t>0, \:x\in\partial\Omega.
\end{pde}
\]
Naturally, we can require Dirichlet boundary conditions by making $F_\eps^{GL}$ finite if and only if $u \in g + H_0^1(\Omega)$ for some $g\in H^1(\Omega)$ rather than for all $u\in H^1(\Omega)$. In this case, we would recover a Dirichlet boundary condition also for the evolution equation in place of the homogeneous Neumann boundary condition. The following statement applies to both settings.

\begin{theorem}[Total energy decrease]\label{theorem energy decrease}
    Let $\alpha,\eps>0$ and $W\in C^2(\R)$ a function such that $W(u) = 0$ if and only if $u\in \{0,1\}$. 
    Assume that $u\in C^2([0, \infty) \times\Omega)$ solves the PDE 
    \[
    (\partial_{tt} + \alpha \partial_t)u = \Delta u - \frac1{\eps^2}\,W'(u)
    \]
    in a domain $\Omega$ and that either
    \begin{enumerate}
        \item $\partial_t u(t,x)\equiv 0$ for all $t>0$ and $x\in \partial\Omega$ or
        \item $\partial_\nu u(t,x)\equiv 0$ for all $t>0$ and $x\in \partial\Omega$, or
        \item $\partial\Omega = \emptyset$.
    \end{enumerate}
    Then the total energy
    \[
    E_\eps(u) = F_\eps^{GL}(u) + \frac{\eps}2\,\|\partial_tu\|_{L^2(\Omega)}^2
    \]
    is monotone decreasing in time. 
\end{theorem}

The three scenarios correspond to Dirichlet boundary conditions, which do not depend on time (first case), homogeneous Neumann boundary conditions (second case), and periodic boundary conditions (third case).

\begin{proof}
    We compute 
    \begin{align*}
    E_\eps'(t) &= \frac{2\eps}2\,\langle u_t, u_{tt}\rangle_{L^2(\Omega)} + \int_\Omega \eps\,\langle\nabla u, \,\nabla u_t\rangle + \frac{W'(u)}\eps\,u_t\dx\\
        &= \int_\Omega \eps\, u_{tt}\,u_t + \mathrm{div}(u_t\nabla u) - u_t\,\Delta u + \frac{W'(u)}\eps\,u_t\dx\\
        &= \int_\Omega \left(\eps\,u_{tt} - \eps \Delta u+ \frac{W'(u)}\eps\right)u_t \dx + \int_{\partial\Omega}u_t\,\partial_\nu u\,\mathrm dA = -\alpha\eps\,\|u_t\|_{L^2(\Omega)}^2
    \end{align*}
    since the boundary integral vanishes if $u_t\equiv 0$ or $\partial_\nu u\equiv 0$ on $\partial\Omega$ (in particular if $\partial\Omega = \emptyset$). The domain integral is evaluated using the PDE. 
\end{proof}

We conjecture that the regularity assumptions can be relaxed. The existence of weak solutions to the accelerated Allen-Cahn equation is deduced as the continuous time limit of a stable discrete time algorithm in \cite{chen2024second}. The existence, uniqueness, and regularity of solutions to the accelerated Allen-Cahn equation are left for future work. 

While we do not prove that solutions to the accelerated Allen-Cahn equation have a long term limit, we establish that any limit if it does exist must be a critical point of the energy $F_\eps^{GL}$. 

\begin{theorem}[Conditional convergence to a critical point]\label{theorem conditional convergence}
Let $u$ be as in Theorem \ref{theorem energy decrease}
Assume that there exist $c, C$ and $R\geq 1$ such that 
\[
c |u|\leq \operatorname{sign}(u)\,W'(u) \leq C\,|u| \quad\text{i.e. }c\,u^2 \leq u\,W'(u)\leq C\,u^2 \qquad\forall\ |u|\geq R.
\]
    \begin{enumerate}
        \item There exist a sequence of times $t_n\to \infty$ and $u^*\in H^1(\Omega)$ such that $u(t_n, \cdot)\wto u^*$ weakly in $H^1(\Omega)$.
        
        \item Assume that $u(t,\cdot)\to u^\infty$ weakly in $H^1(\Omega)$ as $t\to\infty$. Then $-\Delta u^\infty + \frac{W'(u^\infty)}{\eps^2}=0$.
    \end{enumerate}
\end{theorem}

The growth on $W$ is technical and could be relaxed at the expense of a more complicated statement. For instance, for the classical example $W(u) = u^2(1-u)^2$, we would have to consider the more complicated space $H^1\cap L^4$.

\begin{proof}
    {\bf First claim.} The condition on $W'$ ensures that 
    \[
    W(u) \geq W(R) + c\int_R^u t \dt \geq c \left(\frac{u^2}2 - \frac{R^2}2\right) \qquad \Ra \quad u^2 \leq R^2 + \frac 2c\,W(u)
    \]
    if $u\geq R$ and similarly for $u\leq -R$. The same bound holds trivially for $|u|\leq R$, so
    \[
    \int_\Omega u^2\dx \leq R^2 |\Omega| + \frac2c\int_\Omega W(u)\dx.
    \]
    
    Since $E_\eps(u)$ remains bounded along the accelerated Allen-Cahn equation, we find that for any sequence $t_n\in(0,\infty)$, the associated sequence $u(t_n,\cdot)$ remains bounded in $H^1$. By \cite[Theorem 3.18 and Proposition 9.1]{brezis2011functional} there exists a subsequence $t_{n_k}$ of $t_n$ such that $u(t_{n_k}, \cdot)$ converges to a limit weakly in $H^1(\Omega)$.
    
    {\bf Second claim.} Assume for the sake of contradiction that $f := \Delta u^\infty + \frac{W'(u^\infty)}{\eps^2}\neq 0$ in $H^{-1}(\Omega)$. Let $\phi \in H^1_0(\Omega)$ such that $\bar c:= \langle f, \phi\rangle_{H^{-1}, H^1} >0$. 
    Denote 
    \[
    h(t):= \int_\Omega u(t,x)\,\phi(x)\dx.
    \]
    Since $u$ converges to a limit weakly in $L^2(\Omega)$ as $t\to\infty$, we find that $h$ converges to a limit. To obtain a contradiction, compute
    \begin{align*}
    \left(\frac{d^2}{dt^2} + \alpha\frac{d}{dt}\right)\int_\Omega u(t,x)\,\phi(x)\dx
        &= \int_\Omega \left(u_{tt}(t,x)+\alpha u_t\right)\,\phi(x)\dx\\
        &= \langle -\Delta u + W'(u)/\eps^2, \:\phi\rangle_{H^{-1},H^1}.
    \end{align*}
    We note that $-\Delta u(t,\cdot) \wto -\Delta u^\infty$ in $H^{-1}(\Omega)$ by definition. For the non-linear term $W'(u(t,\cdot))$, we note that $W'$ grows linearly, i.e.\ $W'(u(t_n, \cdot)) \to W'(u^\infty)$ strongly in $L^2$ by the compact Sobolev embedding \cite[Kapitel 6.7]{dobrowolski2010angewandte} and hence in $H^{-1}$.
        
    We conclude that for sufficiently large $t$ we have
    \[
    \frac{d}{dt} \left(h'(t) + \alpha h(t)\right) = h''(t) + \alpha h'(t)  
    \geq \bar c/2.
    \]
    In particular
    \[
    h'(t) + \alpha h(t) \geq h_* + \frac{\bar c}2\, (t-t^*)
    \]
    for $t>t^*$ with some $h_*, t_*\in \R$. Since $h$ converges by assumption, it also remains bounded. This means that $h'(t) \to+\infty$ as $t\to \infty$, also contradicting the fact that $h$ remains bounded.
    \end{proof}

Formally, the accelerated Allen-Cahn equation is a semi-linear version of the telegraph equation
\[
\big(\partial_{tt} + \lambda_1\partial_t + \lambda_2\big) u = c^2 \,\Delta u
\]
with a zeroth-order non-linearity. The telegraph equation is a hyperbolic partial differential equation of second order which arises when decoupling a system of PDEs modelling the electric flow in a transmission line. As an equation from electromagnetism, it is compatible with special relativity: Information cannot propagate faster than the speed of light. 

We establish a similar limit on the speed of propagation also for the accelerated Allen-Cahn equation, as it will be crucial for our geometric analysis below. While more general results are known and the method of proof is standard, we provide a brief proof for the reader's convenience.

\begin{theorem}[Finite speed of propagation]\label{theorem finite speed of propagation}
    Let $(\bar t, \bar x) \in (0,\infty)\times \Omega$ such that the initial condition satisfies $u_0\equiv 0$ on $B_{\bar t}(\bar x)\cap \Omega$ and the initial condition for $\partial_tu$ vanishes on $B_{\bar t}(\bar x)$. Assume that $u$ is as in Theorem \ref{theorem energy decrease}. Then $u(\bar t, \bar x) = 0$.
\end{theorem}

In optimization, it is standard to use the initial condition $u_t \equiv 0$ at time $t=0$.

\begin{proof}
Consider the `past light cone' of $(\bar t, \bar x)$, i.e.\ the family of shrinking domains $C(t) = B_{\bar t-t}(\bar x)\cap \Omega$ and the localized energy
\[
e(t) := \int_{C(t)} \frac12\,u_t^2 + \frac12\,\|\nabla u\|^2  + \frac{W(u)}{\eps^2}\dx.
\]
We can compute the derivative
\begin{align*}
e'(t) &= \int_{C(t)} u_tu_{tt} + \langle \nabla u, \nabla u_t\rangle + \frac{W'(u)}{\eps^2}\,u_t\dx
    - \int_{\Omega\cap \partial B_{\bar t-t}(\bar x)} \frac{u_t^2 + \|\nabla u\|^2}2 + \frac{W(u)}{\eps^2}\d A\\
    &= \int_{C(t)} \left(u_{tt}-\Delta u + \frac{W'(u)}{\eps^2}\right)u_t\dx - \int_\Omega \frac{u_t^2 + \|\nabla u\|^2}2 + \frac{W(u)}{\eps^2} \d A + \int_{\partial C(t)}u_t\,\partial_\nu u \d A
\end{align*}
since only the portion $\partial B_{\bar t-t}(x_0)$ of the boundary is moving. The second boundary integral is obtained by the divergence theorem and therefore goes over the whole boundary $\partial C(t)$, but the boundary conditions make $u_t\partial_\nu u \equiv 0$ on $\partial\Omega \cap \partial C(t)$, so
\begin{align*}
e'(t) &= - \alpha\eps\int_{C(t)}|u_t|^2 \dx + \int_{\Omega\cap \partial B_{\bar t-t}(x_0)} u_t\,\partial_\nu u - \frac{u_t^2+ \|\nabla u\|^2}2 - \frac{W(u)}{\eps^2} \d A\\
    &\leq \int_{\partial B_{\bar t-t}(x_0)} \frac12\big(u_t^2 + (\partial_\nu u)^2\big) - \frac{u_t^2+ \|\nabla u\|^2}2\d A \leq 0
\end{align*}
by Young's inequality, using that $|\partial_\nu u| \leq \|\nabla u\|$. In particular, since $\|\nabla u\| \equiv u_t \equiv W(u) \equiv 0$ in $C(0)$, we conclude that $e(t) = 0$ for all $t\leq \bar t$.
\end{proof}

Naturally, the same would be true if $u\equiv 1$ on $\Omega \cap B_{\bar t}(\bar x)$. 

Theorems \ref{theorem energy decrease} and \ref{theorem conditional convergence} serve as indicators that the accelerated Allen-Cahn equation can alternatively be used to as a tool to minimize the Ginzburg-Landau energy in place of the Allen-Cahn equation. To the best of our knowledge, the existence of a long time limit for solutions to the accelerated (and even the regular!) Allen-Cahn equation is open in the general case. Even for gradient flows in finite dimensions, a subsequential limit may depend on the sequence of times and a unique limit may not exist -- see for instance the summary of counterexamples in the introduction of \cite{dereich2021convergence}.

Theorem \ref{theorem finite speed of propagation} is a first step towards geometrically analyzing solutions to the accelerated Allen-Cahn equation. We will explore the geometry in greater detail in the following section by deriving the singular limit $\eps\to 0$ of the evolutions.

\subsection{Singular limit}\label{section singular}
Even if the initial condition takes values strictly between the potential wells, there is no guarantee that the same is true for the solution to the accelerated Allen-Cahn equation at a positive time. For instance, with homogeneous Neumann (or periodic) boundary conditions, if the inital condition $u^0 \equiv c$ is constant in space, so is the solution $u(t,\cdot) \equiv c(t)$ for all positive times, and the constant $c(t)$ satisfies the one-dimensional heavy ball ODE
\begin{equation}\label{eq harmonic oscillator}
c'' + \alpha c' = - \frac{W'(c)}{\eps^2}.
\end{equation}
If $W''(0), W''(1)>0$, $W$ behaves like a quadratic function at the minimizers and $c$ strongly resembles a harmonic oscillator. Unless $\alpha$ is sufficiently large -- roughly $2\sqrt{\min\{W''(0), W''(1)\}}/\eps$ for critical dampening -- it is well-known that $c$ `overshoots' the potential wells -- see e.g.\ \cite{siegel2023qualitative}.

\subsubsection{Singular limit: First attempt}
Despite this observation, we posit that a singular limit for the evolution of interfaces as $\eps\to 0$ exists for well-prepared initial data, and that it is given by the relation
\begin{equation}\label{eq singular limit}
\partial_t v = (1-v^2)\left( h - \alpha v\right)
\end{equation}
between the normal velocity $v$ and the mean curvature $h$ of the evolving hypersurface.
To derive \eqref{eq singular limit}, we make the same formal ansatz 
\begin{equation}\label{eq formal asymptotic ansatz}
u(t,x) = \phi\left(\frac{r(t,x)}\eps\right)
\end{equation}
as in Section \ref{section singular limit allen-cahn} for the regular Allen-Cahn equation and -- entirely heuristically -- attempt to derive an evolution equation for $r$. The procedure strongly resembles our sketch for deriving mean curvature flow as the singular limit of Allen-Cahn equations. In this case, the ansatz \eqref{eq formal asymptotic ansatz} cannot be justified rigorously since we know that the solution $u$ may not remain in $(0,1)$, even if the initial datum is. We will see below that a corrector term is required in this more complicated case, but even our simple ansatz provides useful information. We compute
\[
\partial_tu = \frac{\phi'}\eps\,\partial_tr, \qquad \partial_{tt}u = \frac{\phi''}{\eps^2}\,\big(\partial_tr\big)^2 + \frac{\phi'}{\eps}\,\partial_{tt}r 
\]
and analogously for the spatial derivatives
\[
\nabla u = \frac{\phi'}\eps\,\nabla r, \qquad \Delta u = \frac{\phi''}{\eps^2}\,\|\nabla r\|^2 + \frac{\phi'}{\eps}\,\Delta r.
\]
Since $W'(\phi) =\phi''$, the accelerated Allen-Cahn equation can be rewritten as
\begin{align}\label{eq accelerated allen-cahn rewritten}
0 &= \frac{\phi''}{\eps^2}\left((\partial_tr)^2 +1 - \|\nabla r\|^2\right) + \frac{\phi'}\eps\left(\partial_{tt}+\alpha\partial_t -\Delta\right)r.
\end{align}
Again, we assume that the coefficient of $\eps^{-2}$ is more important, i.e.\ that we should focus on establishing
\[
|\partial_tr|^2 + 1 - \|\nabla r\|^2 \equiv 0
\]
(at least as much so as we can achieve). To satisfy the equation, we make a separable ansatz: We assume that $r$ can be approximated sufficiently well by a function of the form
\[
r(t,x) = \omega\big(t, \,\pi_{\partial E(t)}(x)\big)\cdot \sdist\big(x, E(t)\big)
\]
where $E(t)$ is an evolving open set, $\pi_{\partial E(t)}$ is the closest point projection onto its boundary and $\sdist(\cdot, E(t))$ is the signed distance function to $\partial E(t)$, taken to be positive inside $E$ and negative outside. The function $\omega$ remains to be determined. We observe that
\begin{align*}
\partial_t r &= \big(\partial_t\omega + \nabla\omega \cdot \partial_t \pi_{\partial E(t)}(x)\big) \,\sdist_{E(t)} + \omega\,\partial_t \sdist_{E(t)}(x)\\
\nabla r &= \sdist\,D\pi^T\nabla \omega + \omega\,\nabla \sdist.
\end{align*}
As before, we identify $v = \partial_t\sdist$ as the normal velocity of the evolving interface $E(t)$ and deduce that
\begin{align*}
0 &= \omega^2v^2  + 2\omega v\,\big(\partial_t\omega +\nabla \omega\cdot\partial_t\pi\big)\sdist + 1 - \omega^2 -2 \big(\nabla \sdist^TD\pi^T\nabla \omega\big)\,\omega \,\sdist + O(\sdist^2)\\
	&= \big\{\omega^2v^2+ 1 - \omega^2\big\} + 2v\,\partial_t\omega\,\omega\,\sdist + O(\sdist^2)
\end{align*}
since $D\pi\nabla \sdist = 0$: the gradient of the signed distance is orthogonal to the interface, i.e.\ it points into the direction along which $\pi$ is constant. Making the leading order term zero yields
\[
\omega^2(1-v^2) = 1 \qquad \Ra\quad \omega = \big(1-v^2\big)^{-1/2},
\]
i.e.\ $\omega$ is the usual Lorentz factor (in coordinates where the speed of light is normalized as $c=1$).
Notably, we cannot make the coefficient vanish identically also at a distance from the interace, as we did for the Allen-Cahn equation. This will require a corrector term to the optimal interface which we introduce below.

To analyze the next order coefficient, we compute
\begin{align*}
\partial_{tt}r &= \frac{d}{dt}\big(\partial_t\omega + \nabla\omega \cdot \partial_t \pi\big) \,\sdist + 2 \big(\partial_t\omega + \nabla\omega \cdot \partial_t \pi\big) v + \omega\,\partial_tv\\
\Delta r &= \mathrm{div}\big(D\pi^T\nabla \omega\big) \,\sdist + 2 \,\nabla\sdist^T D\pi^T\nabla \omega + \omega\,\Delta \sdist.
\end{align*}
We only consider this equation on the interface $\{\sdist = 0\}$, i.e.\ we discard the terms in the equation which retain $\sdist$. As before, we note that $\nabla \sdist^TD\pi^T\nabla \omega =0$. Additionally, we observe that $\omega$ only varies {\em along} the interface and that $\partial_t\pi(x)$ is orthogonal to the moving interface $\partial E(t)$ for all $x\in \partial E(t)$. To see this, note that
\[
\pi_t(x) - \pi_{t+h}(x) = x- \pi_{t+h}(x) \bot \partial E(t+h)
\]
by the properties of the closest point projection. Thus for a sufficiently regular evolution we have
\[
\frac{d}{dt} \pi_t(x) = \lim_{h\to 0} \frac{\pi_t(x) - \pi_{t+h}(x)}h = \lim_{h\to 0} \beta(h)\,\nu_{\partial E(t+h)}(\pi_{t+h}(x)) = \beta_0 \nu_{\partial E(t)}(x)
\]
where $\beta$ is a real-valued function and $\nu_{\partial E(s)}(z)$ denotes the normal to the evolving interface at time $s$ and $z\in \partial E(s)$. This means that $\nabla\omega \cdot \partial_t \pi\equiv 0$.
Overall, we observe that
\begin{align*}
0 &= \frac{\phi'(r/\eps)}\eps \left(2v\,\partial_t\omega + \omega\,\partial_tv + \alpha\,\omega v-\omega \,h\right)+ \frac{\phi''(r/\eps)\cdot r/\eps}\eps\,2v\,\partial_t\omega + O(\sdist).
\end{align*}
Integrating by parts, we see that e.g.
\begin{align*}
-\int_\R \phi''\left(\frac r\eps\right)\,\frac r\eps\,\frac1\eps\dr &= \int_\R \phi'\left(\frac r\eps\right)\frac1\eps \dr = \lim_{r\to\infty}\big(\phi(r) - \phi(-r)\big) = 1\\
- 2\int_\R \phi''\left(\frac r\eps\right)\,\frac r\eps\cdot \phi'\left(\frac r\eps\right) \frac1\eps\dr &= - \int_\R \frac{d}{dr} \left|\phi'\left(\frac r\eps\right)\right|^2\,\frac r\eps\dr = \int_\R \left|\phi'\left(\frac r\eps\right)\right|^2\,\frac 1\eps\dr = \int_{-1}^1 \sqrt{2\,W(z)}\dz
\end{align*}
as noted in Appendix \ref{appendix corrector}. In particular, both terms impact the equation at the same order (when averaging over the distance to the interface), and neither can be ignored in favor of the other. However, since $\phi' (r) \neq r\phi''(r)$, we cannot simply compare or cancel their coefficients. A corrector is needed to compensate for the fact that the interface does not only change its width, but also its shape as it moves. The corrector defines the `exchange ratio' between the coefficients of $\phi'$ and $r\phi''$.

\subsubsection{Singular limit: Revised ansatz}
We claim that there exists a solution $\psi \in H^1(\R)$ of the `corrector equation'
\[
\psi'' - W''(\phi) \psi = \phi' +2 \,x\phi''(x),
\]
i.e.\ a solution for which both $\psi$ and $\psi'$ are square integrable. A proof of the claim is given in Appendix \ref{appendix corrector}. 

Rather than assuming that $u= \phi(r/\eps)$, we make the refined ansatz that 
\[
u= \phi(r/\eps) + \eps(v\partial_t\omega)\,\psi(r/\eps)
\]
for the same function $r = (1-v^2)^{-1/2}\sdist$ we found above. Using the approximation $W'(\phi+\eps\xi) = W'(\phi) + \eps\,W''(\phi)\xi$, we find that \eqref{eq accelerated allen-cahn rewritten} is heuristically estimated by
\begin{align*}
0     &= \frac{\phi''}{\eps^2}\left((\partial_tr)^2 +1 - \|\nabla r\|^2\right) + \frac{\phi'}\eps\left(\partial_{tt}+\alpha\partial_t -\Delta\right)r + (v\partial_t\omega) \frac{\psi''}\eps \big(\partial_tr^2 - \|\nabla r\|^2\big) \\
    &\qquad + \frac{W''(\phi)}\eps(v\partial_t\omega)\psi + O(1)\\
    &= \frac{\phi''\,\frac{\sdist}\eps}{\eps}\,2v\,\partial_t\omega + \frac{\phi'}\eps\left(2v\partial_t\omega + \omega\partial_tv + \alpha\omega v-\omega h\right) - (v\partial_t\omega) \frac{\psi'' - W''(\phi)\psi}\eps  + O(1)\\
    &= \frac{\phi''\,\frac{\sdist}\eps}{\eps}\,2v\,\partial_t\omega + \frac{\phi'}\eps\left(2v\partial_t\omega + \omega\partial_tv + \alpha\omega v-\omega h\right) - (v\partial_t\omega) \frac{\phi' + 2\,\phi''\,\frac{\sdist}\eps}\eps  + O(1)\\
    &= \frac{\phi'}\eps\left(v\partial_t\omega + \omega\partial_tv + \alpha\omega v-\omega h\right) + O(1)
\end{align*}
since we may replace $(\partial_tr)^2 - \|\nabla r\|^2$ by $-1$ up to leading order. Noting that
\[
v\partial_t\omega = v \partial_t(1-v^2)^{-1/2} = -2v^2\partial_tv\,\left(-\frac12\right) (1-v^2)^{-3/2} = \omega\,\frac{v^2\partial_tv}{1-v^2}, 
\]
we reformulate the obtained equation as
\begin{align*}
0 &= v\partial_t\omega + \omega\partial_tv + \alpha\omega v-\omega h = \omega\left( \frac{v^2}{1-v^2}\,\partial_tv+ \partial_tv +\alpha v - h\right)  = \omega\left( \frac{1}{1-v^2}\,\partial_tv + \alpha v - h\right)\\
	&= \omega^3 \left(\partial_t v+ (1-v^2)\big(\alpha v - h\big)\right).
\end{align*}
Since $\omega^3>0$, this is equivalent to the proposed law \eqref{eq singular limit}. The same derivation would go through for a time-dependent coefficient of friction $\alpha(t)$. 
No spatial derivatives of $v$ (or $\omega$) tangentially to the interface enter in the singular limit. This does not come as a surprise: Due to the finite speed of propagation, no parabolic term involving e.g.\ $\Delta v$ can be present.

\subsubsection{Numerical validation}

In Figure \ref{figure verifying singular limit}, we compare solutions to the accelerated Allen-Cahn equation to the predicted singular limit for the special case of a circle, where \eqref{eq singular limit} becomes an ordinary differential equation for the radius
\begin{equation}\label{eq singular limit circle}
\ddot r = (1-\dot r^2)\left(- \frac1r - \alpha\,\dot r\right)
\end{equation}
since the PDE \eqref{eq singular limit} is rotationally invariant.\footnote{\ So is the accelerated Allen-Cahn equation, up to boundary conditions. Our derivation is only valid inside the domain $\Omega$, i.e.\ for circles away from the boundary.}\ 
We consider constant vanishing friction and a non-zero coefficient of friction $\alpha \equiv 3$.

We can find the perimeter of (a phase field approximation to) an evolving disk in two ways: By considering the Ginzburg-Landau approximation to the perimeter, and by exploiting the relationship $P = 2 \sqrt{\pi A}$ between perimeter and area of the disk and estimating the area by $\int |u_\eps|\dx$. For the Allen-Cahn equation, both provide accurate approximations of the perimeter.

For the accelerated Allen-Cahn equation, an interface moving with speed $v$ is expected to resemble $u(x) = \phi\left( (1-v^2)^{-1/2}\, \frac{\sdist}\eps\right)$ to leading order, i.e.\ fast moving interfaces resemble an optimal transition, but ``at the wrong length scale'' and therefore count for more diffuse area than if they were at rest. The `velocity-adjusted area density' which arises as the limit of the Ginzburg-Landau energy of evolving interfaces should be 
\begin{align*}
\rho(v) &= \int_{-\infty}^\infty \frac12 \,\left| \frac{d}{dr} \phi \left((1-v^2)^{-1/2}r\right)\right|^2 + W\big(\phi(1-v^2)^{-1/2}r\big)\dr\\
	&= \frac12\left\{(1-v^2)^{-1} + 1\right) \int_{-\infty}^\infty \sqrt{2\,W\big(\phi(1-v^2)^{-1/2}r\big)} \,\phi'\big((1-v^2)^{-1/2}r\big)\dr \\
	&= \frac{(1-v^2)^{-1}+1}{2\,(1-v^2)^{-1/2}} \int_{-1}^1 \sqrt{2\,W(z)}\dz\\
	&= \frac{c_0}2 \,\left( (1-v^2)^{1/2} + (1-v^2)^{-1/2}\right).
\end{align*}
since $\phi' = \sqrt{2\,W(\phi)}$ for the optimal interface as noted in Section \ref{section ginzburg-landau} -- see also Appendix \ref{appendix corrector}.

Differential operators are implemented using a fast Fourier transform on a spatial spatial grid of $500\times 500$ points with a time step size of $\tau = 2.5\cdot 10^{-7}$. Further details of the numerical implementation can be found in Section \ref{section numerical approach}. We compare both estimates to (numerical approximations of) the predicted perimeter and `velocity-adjusted perimeter' $\pi r\, ((1-\dot r^r)^{1/2} + (1-\dot r^2)^{-1/2})$in Figure \ref{figure verifying singular limit} and find good agreement to leading order until the evolving sets vanish, at which point the asymptotic analysis ceases to be valid.

Solutions to the ordinary differential equations for the singular limit are found using a predictor/corrector scheme based on the Adams-Bashforth and Adams-Moulton formulas of order five. The initial values for the linear multi-step methods were found by the Runge-Kutta method of order four.

\begin{figure}
    \centering
    \includegraphics[clip = true, trim = 6mm 0 0 0, width=0.9\linewidth]{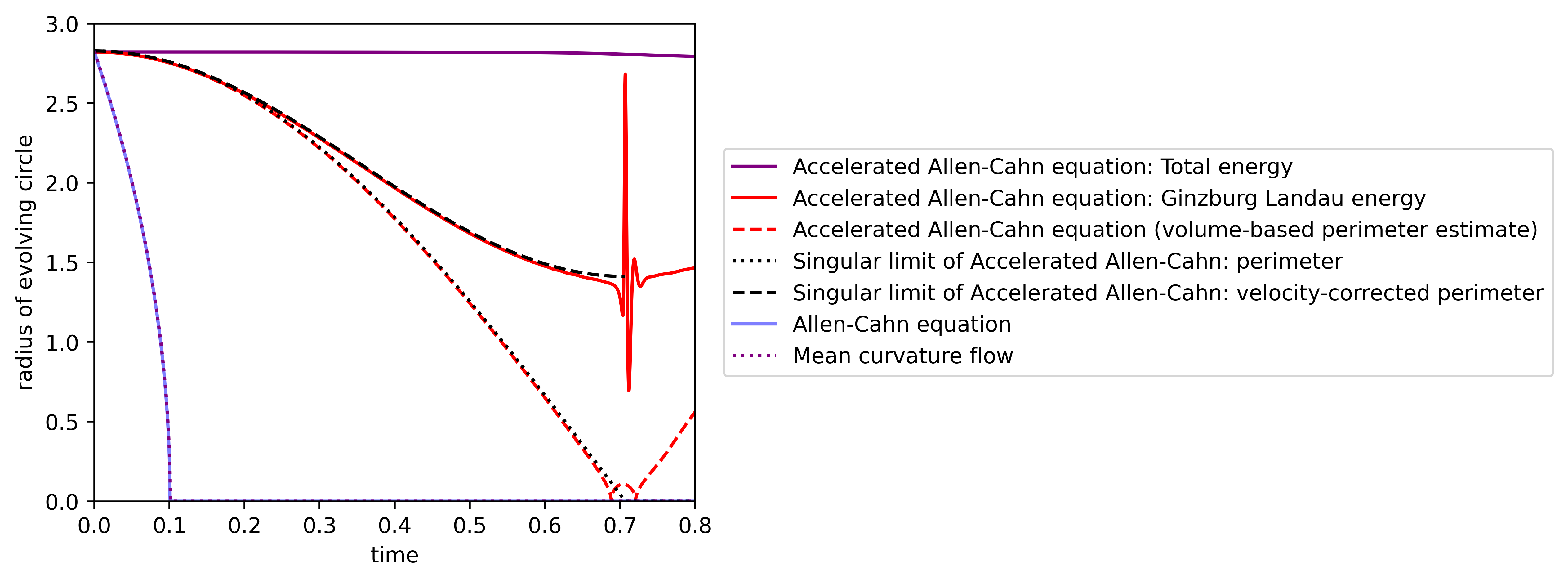}
    \includegraphics[clip = true, trim = 6mm 0 0 0, width=0.9\linewidth]{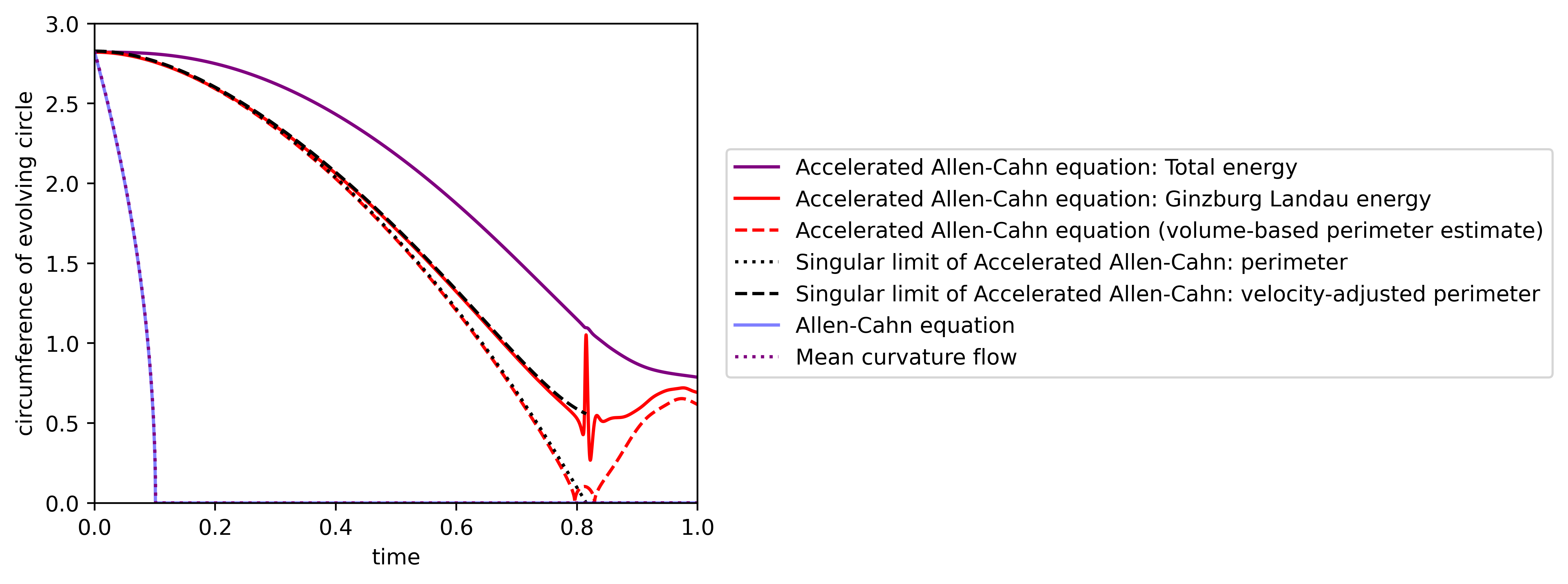}
     
    \caption{We numerically solve the Allen-Cahn equation and the accelerated Allen-Cahn equation for $\eps = 0.015$ on the unit square with periodic boundary conditions. In the top plot, the friction is essentially zero and energy is conserved (up to numerical viscosity). In the bottom plot, the parameter of friction is $\alpha=3$ and total energy (but not pure Ginzburg-Landau energy) is dissipated. In both experiments, the initial condition is a circle of radius $r= 0.45$ centered at $(0.5, 0.5)$. The time interval is chosen slightly differently for both plots to include the vanishing time of the disk.\\
    For the accelerated Allen-Cahn equation, we estimate the circumference of the circle both by the Ginzburg-Landau energy (solid red line) and by the area-perimeter relation of the disk (dashed red line).
    The estimates agree well with the predicted perimeter (dotted black line) and velocity-adjusted perimeter (dashed black line) of the singular limit \eqref{eq singular limit circle}.  Higher friction corresponds to slower dynamics and a slower vanishing disk. Unsurprisingly, the regular Allen-Cahn equation leads to a significantly faster perimeter decrease than the `accelerated' version.\\
    Notably, after vanishing in a point, the solution to the accelerated Allen-Cahn equation continues its evolution by expanding the circle again, increasing its perimeter. Briefly, before the disk expands again, the integral of $u$ becomes negative. Both solutions visually remain circular throughout their evolution, including after the singularity.\\
    }
    \label{figure verifying singular limit}
\end{figure}

\subsection{A brief comparison of the Allen-Cahn and accelerated Allen-Cahn equations}

There are several notable differences between the Allen-Cahn equation and its `accelerated' version.

\subsubsection{Interface velocity}

The maximal speed that an interface can have in the accelerated Allen-Cahn equation or its singular limit is $1$. This is quite different from the singular limit of the Allen-Cahn equation, where the velocity is routinely unbounded at topological singularities. 

Thus interfaces can move faster -- and the perimeter can decrease faster -- in the singular limit of the Allen-Cahn equation than in the singular limit of what we optimistically dubbed the `accelerated Allen-Cahn equation'. However, the analysis in continuous time may not be representative of discrete time simulations: If $f$ is $\mu$-strongly convex, then generically the gradient flow of $f$ decays as $f(x_t) - \inf f \sim \exp(-\mu t)$ while the `accelerated gradient flow' with $\alpha = 2\sqrt\mu$ decays like $f(x_t) - \inf f \sim \exp(-\sqrt\mu\,t)$. Which rate of decay is faster depends on whether $\mu<1$ or $\mu>1$. 

However, if $\nabla f$ is $L$-Lipschitz continuous, then the largest stable step size for the explicit Euler discretization of gradient flow scales as $1/L$, while we can allow larger timesteps $\sim 1/\sqrt{L}$ in Nesterov's discretization of the accelerated gradient flow \cite{su2014differential}. In discrete time, we obtain the decay $f(x_k) - \inf f \sim (1-\sqrt{\mu/L})^k$ for Nesterov's scheme after $k$ steps while gradient descent leads to markedly slower decay $f(x_k) - \inf f \sim (1-\mu/L)^k$.

A more conclusive answer to which algorithm is `faster' in practical applications therefore requires a discrete time analysis. For the highly non-convex problem of perimeter minimization, a fully satisfying answer is beyond the scope of this note, but numerical experiments suggest that momentum may be helpful if a suitable time-discretization is chosen.

\subsubsection{Interface width and shape}

For both equations, we make the ansatz that $u(t,x) = \phi(r(t,x)/\eps)$. For the Allen-Cahn equation, we find based on this that $r$ is the signed distance function from a spatially evolving hypersurface, i.e.\ $\|\nabla r\|\equiv 1$. For the accelerated Allen-Cahn equation, on the other hand, the interface width depends on its velocity $v$: If $v\neq 0$, then $\|\nabla r\| = 1/\sqrt{1-v^2}>1$ and the interface is thinner than it would be without acceleration (see Figure \ref{figure corrector}).

\begin{figure}
    \centering
    \includegraphics[width=0.44\linewidth]{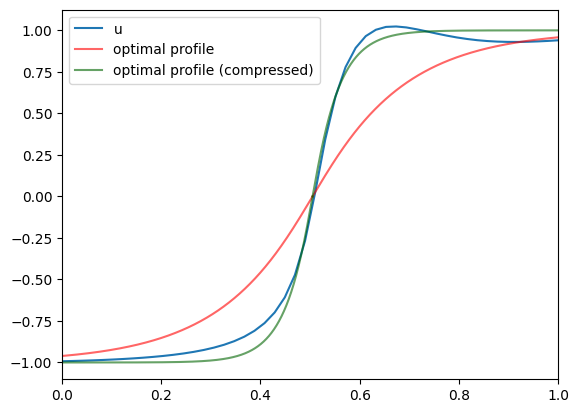}
    \includegraphics[width=0.44\linewidth]{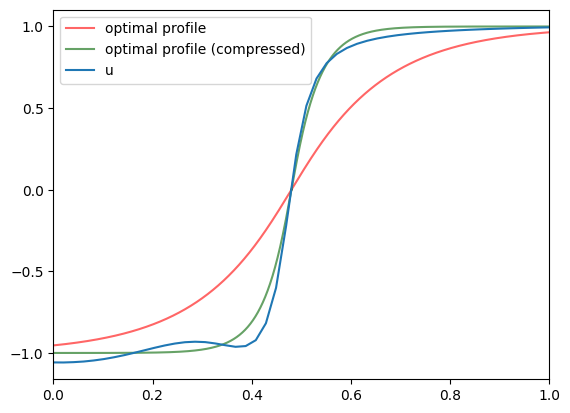}
    
    \includegraphics[width=0.44\linewidth]{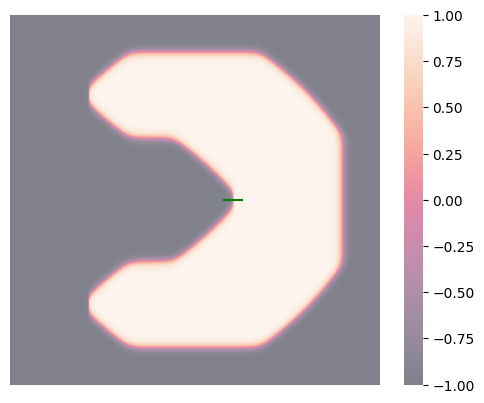}
    \includegraphics[width=0.44\linewidth]{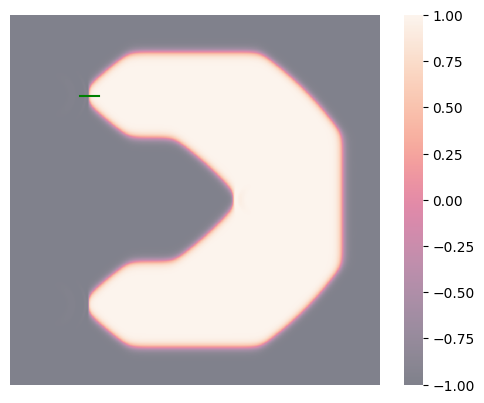}

    \caption{{\bf Top line:} An interface moving to the left (left image) and an interface moving right (right image). The optimal profile $\phi$ is included on the length scale of a stationary interface (red line) and scaled for a good visual fit (green line) to illustrate the compression of fast-moving interface in the Allen-Cahn equation with momentum. {\bf Bottom line:} The green line indicates where the one-dimensional slice which we are viewing in the image above is taken from in a simulation.
    \label{figure corrector}
    }
\end{figure}

In particular, an evolving `interface' may become much thinner than we would anticipate from the optimal transition shape, as we noted when deriving the `velocity adjusted perimeter.' Unless the spatial resolution of a numerical approximation is much finer than the thickness parameter $\eps$ -- and much finer than the Allen-Cahn equation would require in the same problem -- the accelerated Allen-Cahn equation may not be resolved accurately.

For the Allen-Cahn equation, the `diffuse area measures' $\mu_\eps = \frac\eps2\,\|\nabla u_\eps\|^2 + \frac1\eps\,W(u_\eps)$ approach varifold solutions to mean curvature flow \cite{mugnai2011convergence}. Here, the limit of the area measures appears to be a `velocity-adjusted' area density, illustrating the major geometric differences introduced by momentum.

\subsubsection{Ambient/intrinsic compatibility}

The limit of solutions to the Allen-Cahn equation -- a gradient flow of the Ginzburg-Landau energy -- as $\eps\to 0^+$ is (the indicator function of a set whose boundary is moving according to) a mean curvature flow, i.e.\ a gradient flow of the singular limit of the Ginzburg-Landau energies. More plainly: The limit of the gradient flows is the gradient flow of the limit. Whether such a result is true can in general be subtle \cite{sandier2004gamma, serfaty2011gamma, dondl2019effect}, and an analogous result is not true for the momentum method. Namely, the singular limit of the evolution equations is
\[
\dot v = (1-v^2) \big(h - \alpha v) \qquad \text{not}\quad
\dot v = h-\alpha v
\]
which we would get by imitating Newton's second law directly on the surface.

\subsubsection{Non-monotonicity of the energy}

As we observe in Figure \ref{figure verifying singular limit}, the Ginzburg-Landau energy $\Per_\eps(u(t))$ is monotone decreasing for the gradient flow (Allen Cahn), but not for the momentum dynamics (accelerated Allen-Cahn). This is unsurprising: Even for a quadratic function in one dimension, the momentum dynamics generally overshoot the minimizer if the coefficient of friction $\alpha$ is too low. What is perhaps more surprising is that a along the momentum dynamics, a circle can shrink, disappear -- and then reappear! From our preliminary analysis, it is unclear how the flow should continue past singularities.

\section{Solving the accelerated Allen-Cahn equation numerically}\label{section numerical approach}

In this Section, we introduce a computationally stable time discretization of the accelerated Allen-Cahn equation (CINEMA) and compare it to a version of FISTA with convex-concave splitting, using the ideas from Sections \ref{section convex algorithms} and \ref{section convex concave}. Our main algorithmic contribution is given in an abstract setting in Section \ref{section algorithm} and specialized to the accelerated Allen-Cahn equation in Section \ref{section allen-cahn algorithm}. An additional simplification is described in Section \ref{section ws}. A direct comparison between the two time-stepping methods is given in Section \ref{section cinema vs fista}.

\subsection{An unconditionally stable momentum algorithm}
\label{section algorithm}

We consider the version of \eqref{eq our algorithm} for which $g_n = \nabla F(x_{n+1}) + \nabla G(x_n)$, i.e.\ we treate $F$ implicitly and $G$ explicitly, but unlike FISTA, we evaluate the gradient of $G$ at $x_n$ rather than the `advanced' point $x_n + \tau v_n$. In terms of implementation, this merely corresponds to exchanging the order of the operations [{\tt compute gradient}] and [{\tt advance by velocity}] in each time step. We dub this version the Convex Implicit/Non-convex Explicit Momentum Algorithm (CINEMA).

We prove existence and energy-stability the CINEMA scheme in a slightly more general context which automatically covers the PDE setting.

\begin{theorem}[Existence in discrete time]\label{theorem descent existence}
    Let $H$ be a separable Hilbert space. Assume that
    \begin{enumerate}
        \item $F:H\to \R$ is a weakly lower semi-continuous and convex function with sub-differential $\partial F$ such that the domain of $\partial F$ is dense in $H$.

        \item $G:H\to \R$ is concave, continuous, Gateaux-differentiable and the Gateaux-derivative is continuous and linear, i.e.\ for every $x\in H$ there exists a vector $\nabla G(x)\in H$ such that
        \[
        \lim_{t\to 0}\frac{G(x+tv)-G(x)}t = \langle \nabla G(x), v\rangle\qquad \forall\ v\in H.
        \]
    \end{enumerate}
    Then given $x, v \in H$ and $\tau, \eta > 0$, there exists a unique $z^*\in H$ such that
    \begin{equation}\label{eq cinema time step}
    \frac{x + \tau v - z^*}\eta  - \nabla G(x) \in \partial F(z^*).
    \end{equation}
\end{theorem}

The proof is given in Appendix \ref{appendix cinema}. Assuming for the moment that $\partial F(z) = \{\nabla F(z)\}$ is a singleton, we re-write
\[
\frac{x-z+\tau v}\eta -\nabla G(x) \in \partial F(z^*) \qquad \Ra \quad x+\tau v -z - \eta\big( \nabla G(x) + \nabla F(z) \big)= 0
\]
which is compatible with the scheme \eqref{eq our algorithm} with $g_n = \nabla F(x_{n+1}) + \nabla G(x_n)$. 

\begin{theorem}[Stability in discrete time]\label{theorem descent scheme}
    Let $F$ and $G$ be as in Theorem \ref{theorem descent existence}. Given $(x_0, v_0)$ and $\eta, \tau >0$, and $\rho \in [0,1]$, there exists a unique sequence $(x_n, v_n)$ which obeys the CINEMA time-stepping scheme
    \begin{align*}
    \frac{x_n + \tau v_n-x_{n+1}}\eta  - \nabla G(x_n) &\in \partial F(x_{n+1})\\
    v_{n+1} &= \rho_n\big(v_n - \tau g_n\big)
    \end{align*}
    where $g_n \in \nabla G(x_n) + \partial F(x_{n+1})$ is the same element of the sub-differential as in the line above.
    Additionally, the `total energy' $e_n = (F+G)(x_n) + \frac1{2\rho^2}\,\|v_n\|^2$ satisfies
    \[
    e_{n+1} \leq e_n - \frac12(\rho^{-2}-1)\|v_n\|^2 + \left(\frac{\tau^2}2 - \eta\right)\|g_n\|^2.
    \]
    In particular, if $\eta \geq \tau^2/2$, the energy $e_n$ is monotonically decreasing.
\end{theorem}

Also this proof is given in Appendix \ref{appendix cinema}. We will see in the next section that the Ginzburg-Landau energy satisfies all of the conditions of Theorem \ref{theorem descent scheme}. A similar time discretization is derived in this setting in \cite{chen2024second} and used to prove existence in continuous time \cite[Section 3]{chen2024second}.

\subsection{Application to the accelerated Allen-Cahn equation}
\label{section allen-cahn algorithm}

Let $W$ be a double-well potential such that $W'$ grows at most linearly at infinity. We split the energy
\[
F_\eps^{GL} : L^2(\Omega)\to \R, \qquad F_\eps^{GL}(u) = \begin{cases} \int_\Omega \frac\eps2 \,\|\nabla u\|^2 + \frac{W(u)}\eps\dx & \text{if } u\in H^1(\Omega)\\ +\infty &\text{else}\end{cases}
\]
into a convex and a concave part
\[
F_\eps^{GL, convex}(u) = \int_\Omega \frac\eps2 \,\|\nabla u\|^2 + \frac{W_{convex}(u)}\eps\dx, \qquad F_\eps^{GL, concave}(u) = \int_\Omega \frac{W_{concave}(u)}\eps\dx
\]
where $W = W_{convex} + W_{concave}$ and naturally, $W_{convex}$ is convex and $W_{concave}$ is concave. 

\begin{lemma}
    Assume that $W_{convex}'$ grows at most linearly at infinity and that $W_{concave}$ is Lipschitz-continuous and $C^1$-smooth.
    Then the functions $F:= F_\eps^{GL, convex}$ and $G:= F_\eps^{GL, concave}$ satisfy the conditions of Theorem \ref{theorem descent scheme}.
\end{lemma}

The assumption on the derivatives is necessary since the convex-concave decomposition is never unique. Namely, if $f = g+ h$ is a decomposition into a convex and a concave function, then also $f = (g+\phi) + (h-\phi)$ where $\phi$ is any convex function.

\begin{proof}
 $F$ is convex by construction with domain $H^1(\Omega)$ and the singleton-valued sub-differential $-\Delta u$ on the smaller domain $H^2(\Omega)$ of the sub-differential (which is dense in $L^2$). To see that $F$ is lower semi-continuous, we note that if $u_n\wto u^*$ in $L^2(\Omega)$ and $\liminf_{n\to \infty}F(u_n) < \infty$, then the subsequence which realizes the liminf is indeed bounded in $H^1(\Omega)$ as well. A further subsequence then converges to a limit $u'$ weakly in $H^1(\Omega)$. By the compact embedding of $H^1$ into $L^2$, we find that $u_n\to u'$ strongly in $L^2$, yielding that $u' = u^*$. We conclude that
    \[
    \int_\Omega \|\nabla u^*\|^2\dx \leq \liminf_{n\to \infty}\int_\Omega \|\nabla u_n\|^2\dx 
    \]
    by the lower semi-continuity of the norm under weak convergence \cite[Proposition 3.5]{brezis2011functional} and that
    \[
    \int_\Omega W_{convex}(u^*)\dx \leq \liminf_{n\to \infty}\int_\Omega W_{convex}(u_n) \dx
    \]
    by convexity.
    $G$ is concave and continuous under $L^2$-strong convergence by construction. Furthermore
    \begin{align*}
        \lim_{t\to 0}\frac{G(u+t\phi) - G(u)}t = \lim_{t\to0}\int_\Omega\frac{W_{concave}(u+t\phi) - W_{concave}(u)}t\dx
        = \int_\Omega W_{concave}'(u)\phi\dx
    \end{align*}
    by the Dominated Convergence Theorem,
    so $G$ is Gateaux-differentiable and the derivative is a continuous linear functional at every point (and continuous as a map into $L^2$) since $W_{concave}'$ grows at most linearly by assumption.
\end{proof}

The momentum descent step in the scheme of Theorem \ref{theorem descent scheme} for the $u$-variable is 
\[
u_{n+1} = u_n + \tau v_n - \eta \left(-\Delta u_{n+1} + \frac{W_{convex}'(u_{n+1})}{\eps^2} + \frac{W_{concave}'(u_n)}{\eps^2}\right) 
\]
or equivalently
\[
\left(1- \eta \Delta\right) u_{n+1} + \eta \frac{W_{convex}'(u_{n+1})}{\eps^2} = u_n + \tau v_n - \eta \,\frac{W'_{concave}(u_n)}{\eps^2}.
\]
For a particularly simple implementation, we choose a double-well potential such that $W_{convex}(u) = u^2$ (scalar valued case) or $W_{convex}(u) = |u|^2 + \infty\cdot 1_{\{u\notin V\}}$ for the affine subspace $V = \{ u : \sum_{i=1}^k u_i = 1\}$ of $\R^k$ (vector-valued case). To compute the next CINEMA time-step, we only need to solve
\[
\left(1 + 2\,\frac{\eta}{\eps^2} - \eta \Delta\right) \tilde u_{n+1} = u_n + \tau v_n - \frac{\eta}{\eps^2}\,W_{concave}'(u_n)
\]
and $u_{n+1} = \tilde u_{n+1}$ (scalar-valued case) or $u_{n+1} = \Pi_V\tilde u_{n+1}$ where $\Pi_V$ denotes the orthogonal projection onto $V$. A similar expression holds true for FISTA.

The algorithm can be easily discretized by spectral methods or finite elements, leading to large but sparse linear systems. The same is true in the graph setting if the graph is sparse.

\subsection{Smooth double-well potentials with quadratic convex part}\label{section ws}

Computations simplify when we use a doublewell potential with quadratic convex part, i.e.\ $W$ such that $W(u) = u^2 + W_{conc}(u)$ where $W_{conc}$ is a differentiable concave function. Such smoothed versions
\[
W_R(u) =  u^2 - 2\,\sqrt{\frac{R+1}{R}}\,\sqrt{u^2 +\frac{1}{R}} + 1 + \frac2{R}, \qquad R>0
\]
of $\overline W(u) = (|u|-1)^2 = u^2 - 2|u|+1$ are constructed in \cite{related_splitting_paper} with $\lim_{R\to\infty} W_R = \overline W$. The potentials $\overline W, W_R$ have potential wells at $\pm 1$ rather than at $0, 1$. When we want the potential wells at $0,1$ in simulations, we work with $4\,W((1+u)/2)$, which also has the quadratic convex part $u^2$.

A multi-well potential for the vector-valued case can be derived from $W$ as in Section \ref{section vector-valued}.

\subsection{CINEMA vs FISTA}\label{section cinema vs fista}

\begin{figure}
    \centering
    \includegraphics[width=0.24\linewidth]{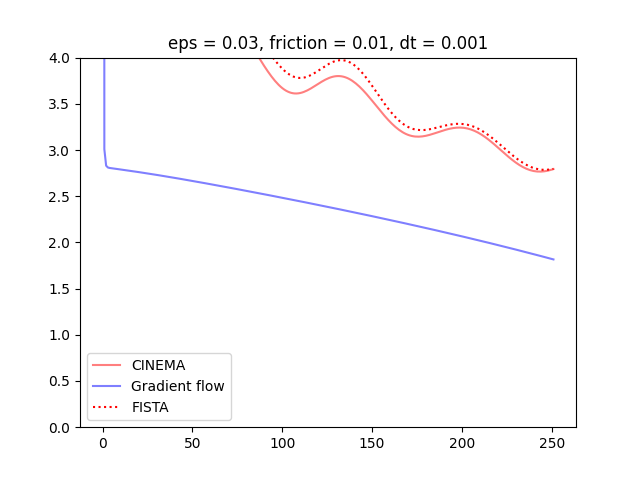}\hfill
    \includegraphics[width=0.24\linewidth]{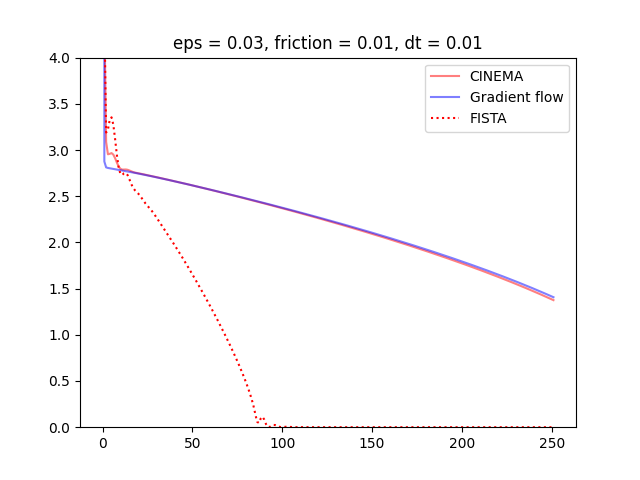}\hfill
    \includegraphics[width=0.24\linewidth]{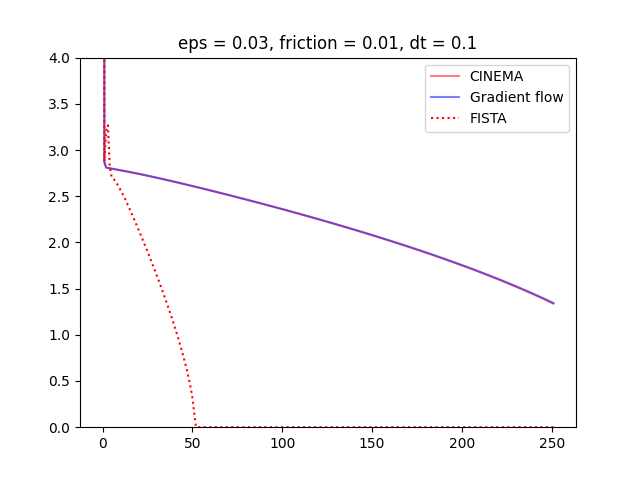}\hfill
    \includegraphics[width=0.24\linewidth]{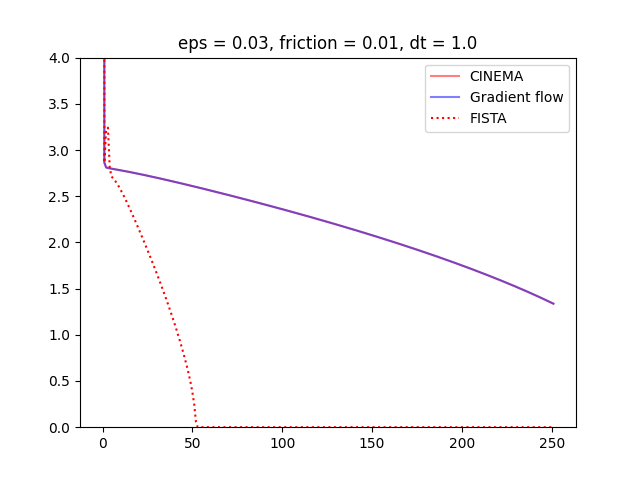}
    
    \includegraphics[width=.33\linewidth]{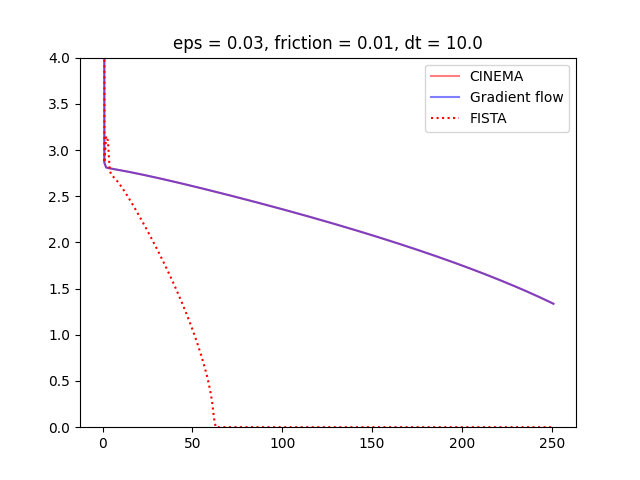}\hfill
    \includegraphics[width=.33\linewidth]{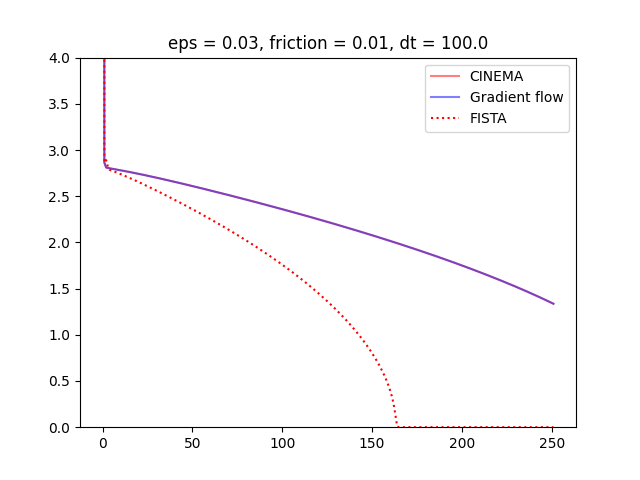}\hfill
    \includegraphics[width=.33\linewidth]{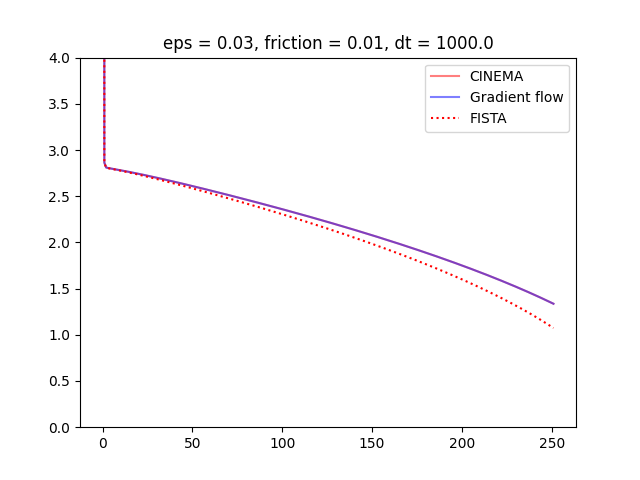}

    \caption{
    \label{figure cinema vs fista pde}
    We compare FISTA, CINEMA and gradient descent where for all schemes, the convex part of the Ginzburg-Landau energy is treated implicitly and the concave part is treated explicitly. The $x$-axis counts the number of iterations, not physical time, and the $y$-axis reflects the Ginzburg-Landau energy (not total energy).
    }
\end{figure}

Let us compare CINEMA to the well-established FISTA scheme. In Figure \ref{figure fista doesnt work}, we compare three different versions of the time-stepping scheme \eqref{eq our algorithm} for minimizing the double-well potential $W_R:\R\to\R$ described in Section \ref{section ws} with $R=2$ (which corresponds to minimizing the Ginzburg-Landau energy among constant functions).

In all three versions, $F$ is the convex part and $G$ is the concave part of $W$, but the algorithms differ in where the gradients are evaluated: For Nesterov's algorithm, we evaluate both gradients at $x_n+\tau v_n$. For FISTA, we evaluate $\nabla G$ at $x_n + \tau v_n$ and $\nabla F$ (implicitly) at $x_{n+1}$. For CINEMA, we evaluate $\nabla G$ at $x_n$ and $\nabla F$ at $x_{n+1}$. In all algorithms, $\eta = \tau^2$ and $\rho = 1/(1+\alpha\tau)$ with $\alpha = 0.01$. 

Unsurprisingly, the explicit Nesterov scheme quickly seizes to be stable. More surprisingly, FISTA fails to reduce the `total energy'  $e_n = W(u_n) + \frac1{2\rho^2}|v_n|^2$ over a range of time step sizes, even as the energy is low enough to ensure that we are indeed in a region where $W$ is convex. Other monotone quantities exist. CINEMA is decreasing the energy fastest among the algorithms and reliably over several orders of magnitude in the time-step size.

\begin{figure}
    \centering
    \includegraphics[width=0.19\linewidth]{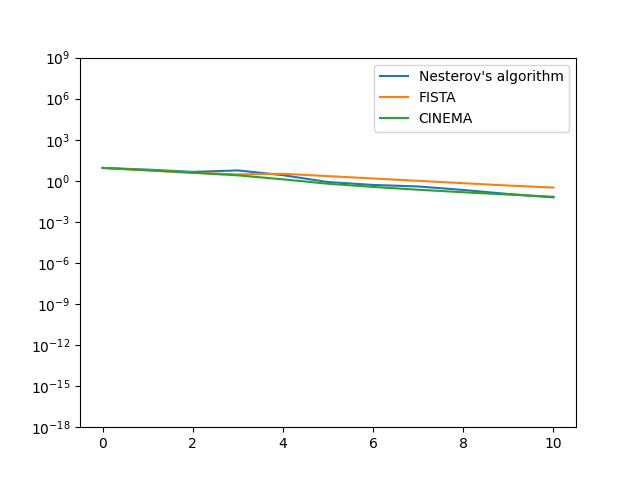}
    \includegraphics[width=0.19\linewidth]{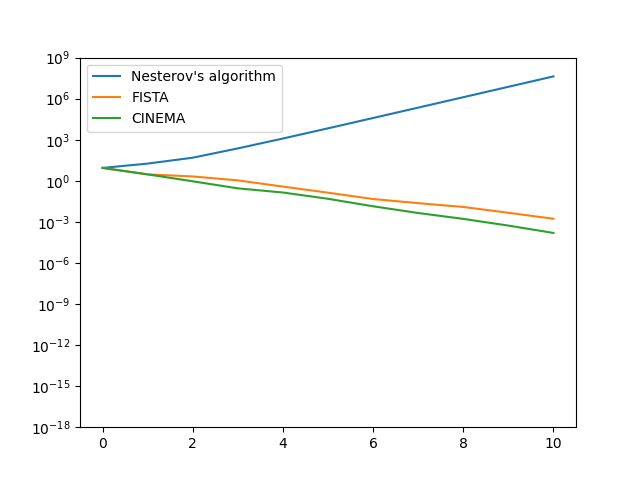}
    \includegraphics[width=0.19\linewidth]{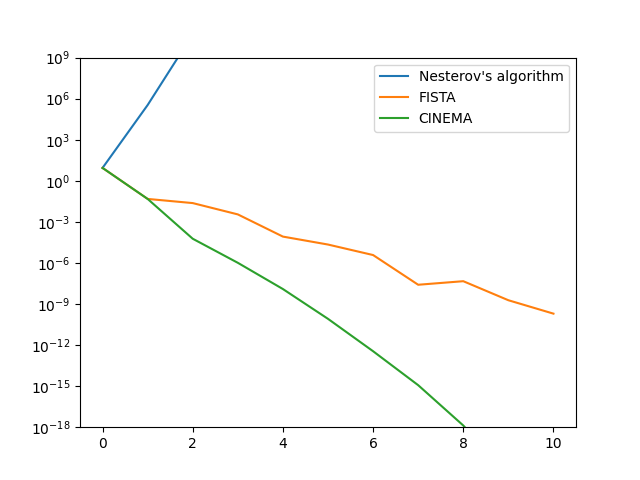}
    \includegraphics[width=0.19\linewidth]{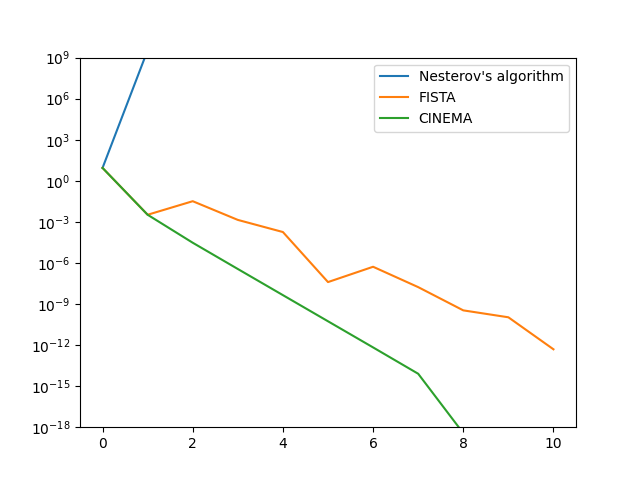}
    \includegraphics[width=0.19\linewidth]{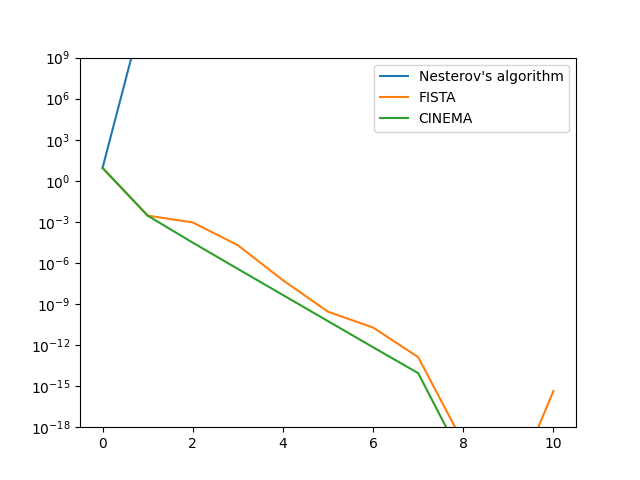}
    \caption{
    \label{figure fista doesnt work}
    We plot the total energy $e_n = W(u_n) + \frac1{2\rho^2}|v_n|^2$ for Nesterov's algorithm (blue line), FISTA (orange line) and the algorithm proposed above (green line) for time step sizes $\tau \in \{0.5, 1, 10, 100, 1000\}$ (from left to right). 
    }
\end{figure}

The situation is quite different in the PDE-experiments we present in Figure \ref{figure cinema vs fista pde}. Here we compare the FISTA and CINEMA discretizations of the accelerated Allen-Cahn equation starting from a circle in the same experimental setting as outlined in Section \ref{section numerics r2}. We see that for small time-steps ($\tau = 10^{-3}$), FISTA and CINEMA behave very similarly, and for very large time step sizes ($\tau = 10^3$), both FISTA and CINEMA essentially recover the behavior of a pure gradient descent scheme with convex-concave splitting.

Over the vast range of intermediate time step sizes (roughly $10^{-2}\leq \tau \leq 10^2$), we find that FISTA leads to much faster decrease of the Ginzburg-Landau energy while CINEMA never accelerates substantially over the gradient descent scheme with convex-concave splitting. The momentum-free scheme stagnates and increasing the time-step size beyond a certain threshold has no noticeable impact, as expected in light of \cite{related_splitting_paper}.

To interpret this observation, we compare the two time-stepping schemes. CINEMA corresponds to solving
\[
u_{n+1} = u_n + \tau v_n + \eta\left(\Delta u_{n+1} - \frac{2u_{n+1} + W_{R,conc}'(u_n)}{\eps^2}\right)
\]
or equivalently
\begin{equation}\label{eq cinema for ginzburg-landau}
\left(-\eta\,\Delta +1 + 2\eta\right) u_{n+1} = u_n +\tau v_n - \frac{\eta}{\eps^2}\,W_{R,concave}'(u_n).
\end{equation}
Formally, all that changes in FISTA is the place where the derivative of $W_{R,concave}$ is evaluated:
\[
\left(-\eta\,\Delta +1 + 2\eta\right) u_{n+1} = u_n +\tau v_n - \frac{\eta}{\eps^2}\,W_{R,concave}'(u_n+\tau v_n).
\]
Both schemes have comparable computational complexity. However, we conjecture that their different performance can be explained as follows: The FISTA scheme factors into a momentum step and a gradient descent step
\[
u_{n+1/2} = u_n+\tau v_n, \qquad (1+ \eta/\eps^2 - \eta\Delta)\,u_{n+1} = u_{n+1/2} - \frac{\eta}{\eps^2}\,W_{concave'}(u_{n+1/2}).
\]
The momentum step is not energy-driven and may lead to an intermediate `loss of regularity' which is regained by the gradient descent step. The velocity -- which concentrates on the narrow interface -- is not regularized. By comparison, the two steps happen simultaneously and cannot be separated for CINEMA. This leads to unconditional stability, but also comes at the cost of oversmoothing the velocity with the Laplacian by coupling the momentum step and the gradient descent step. See also Figure \ref{figure large time step} for a visualization of the FISTA scheme with large time step size.

\section{Numerical experiments}\label{section numerical experiments}

\subsection{The Accelerated Allen-Cahn Equation in Euclidean Spaces}

We present numerical solutions to the accelerated Allen-Cahn equation in two and three dimensions to develop geometric intuition for its behavior.

\subsubsection{Curves in Two Dimensions}\label{section numerics r2}

\begin{figure}
    \begin{flushleft}
    \includegraphics[height = 3.5cm, clip = true, trim = 0 0 2.3cm 0]{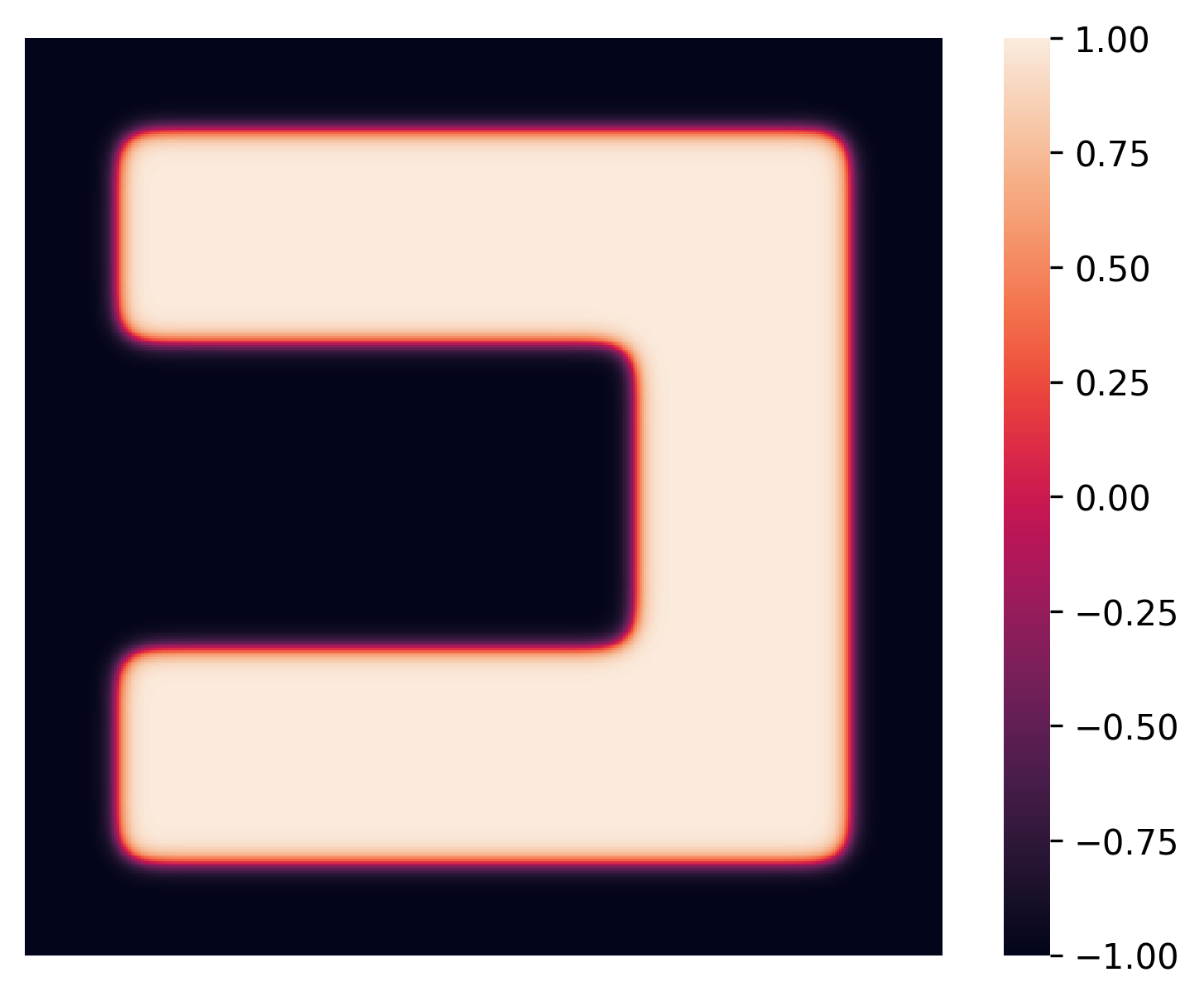}
    \includegraphics[height = 3.5cm, clip = true, trim = 0 0 2.3cm 0]{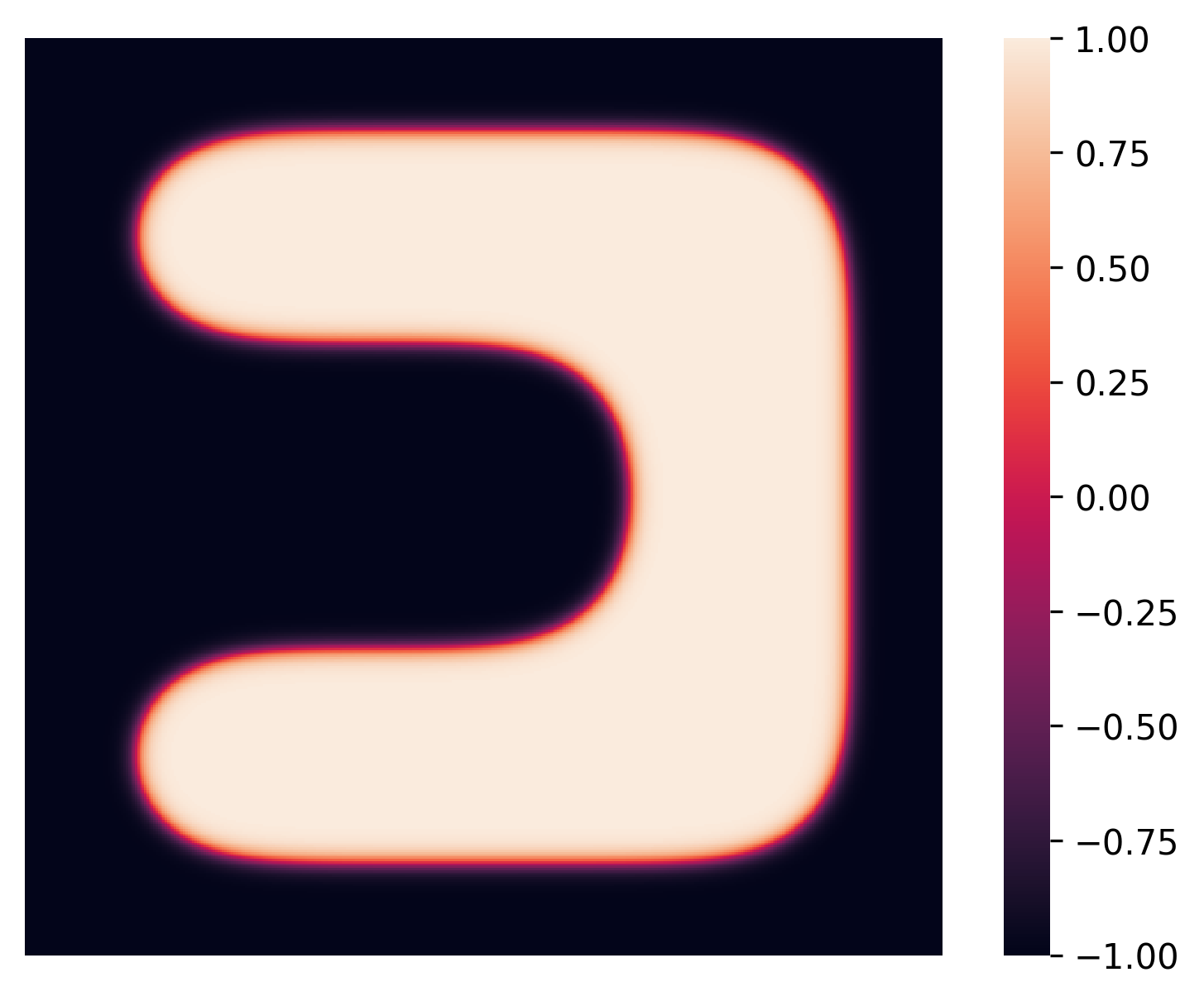}
    \includegraphics[height = 3.5cm, clip = true, trim = 0 0 2.3cm 0]{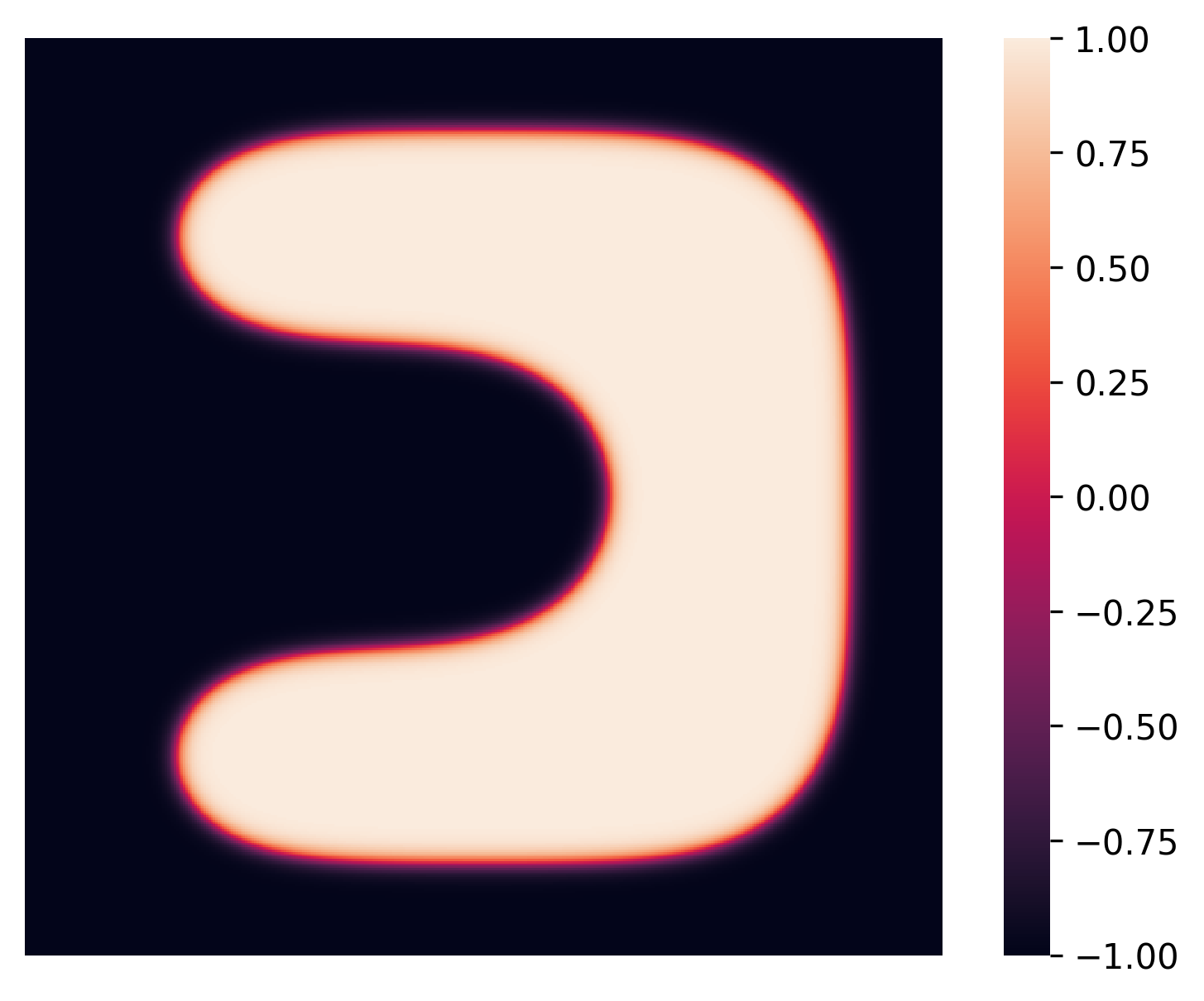}
    \includegraphics[height = 3.5cm, clip = true, trim = 0 0 2.3cm 0]{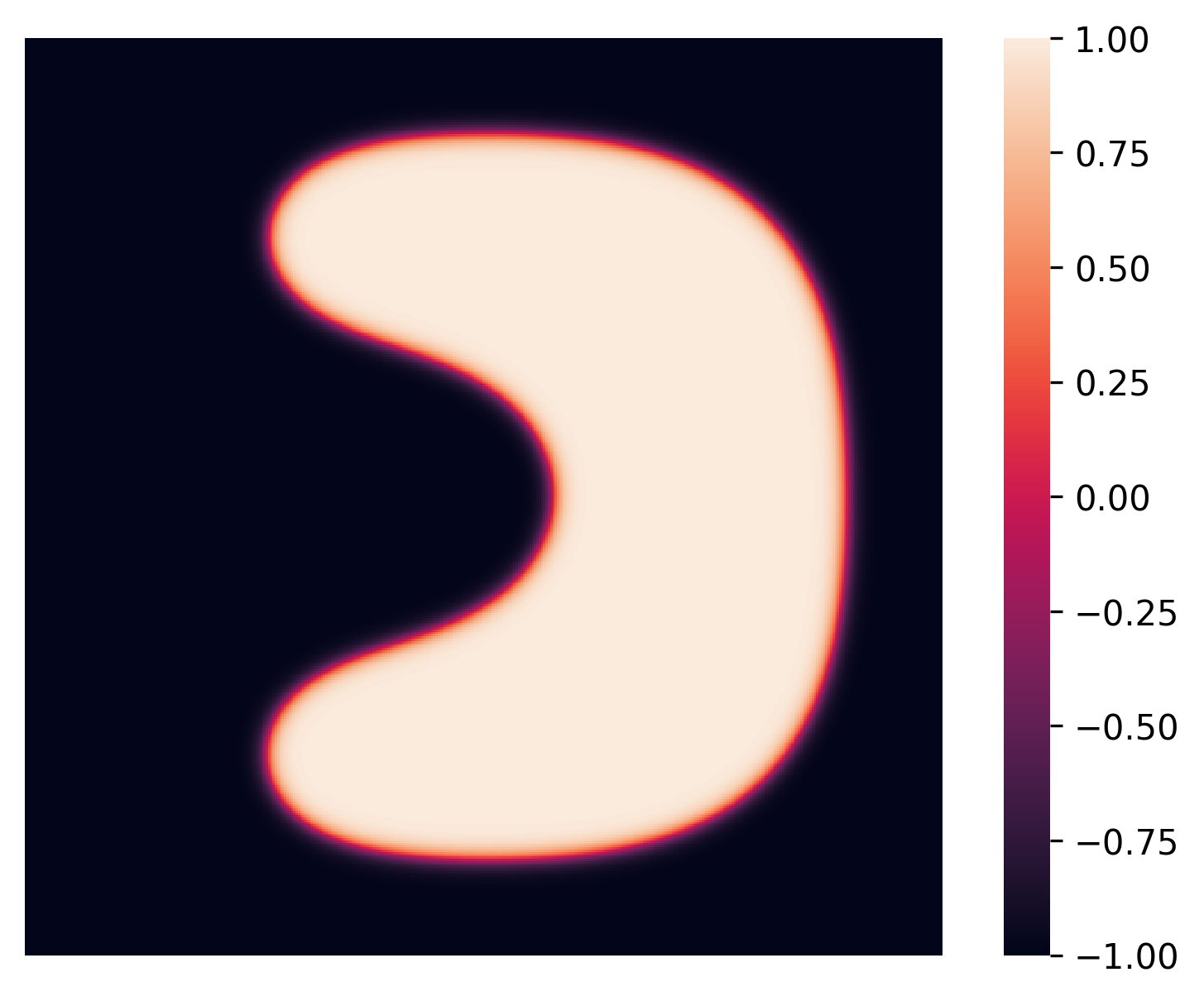}

    \includegraphics[height = 3.5cm, clip = true, trim = 0 0 2.3cm 0]{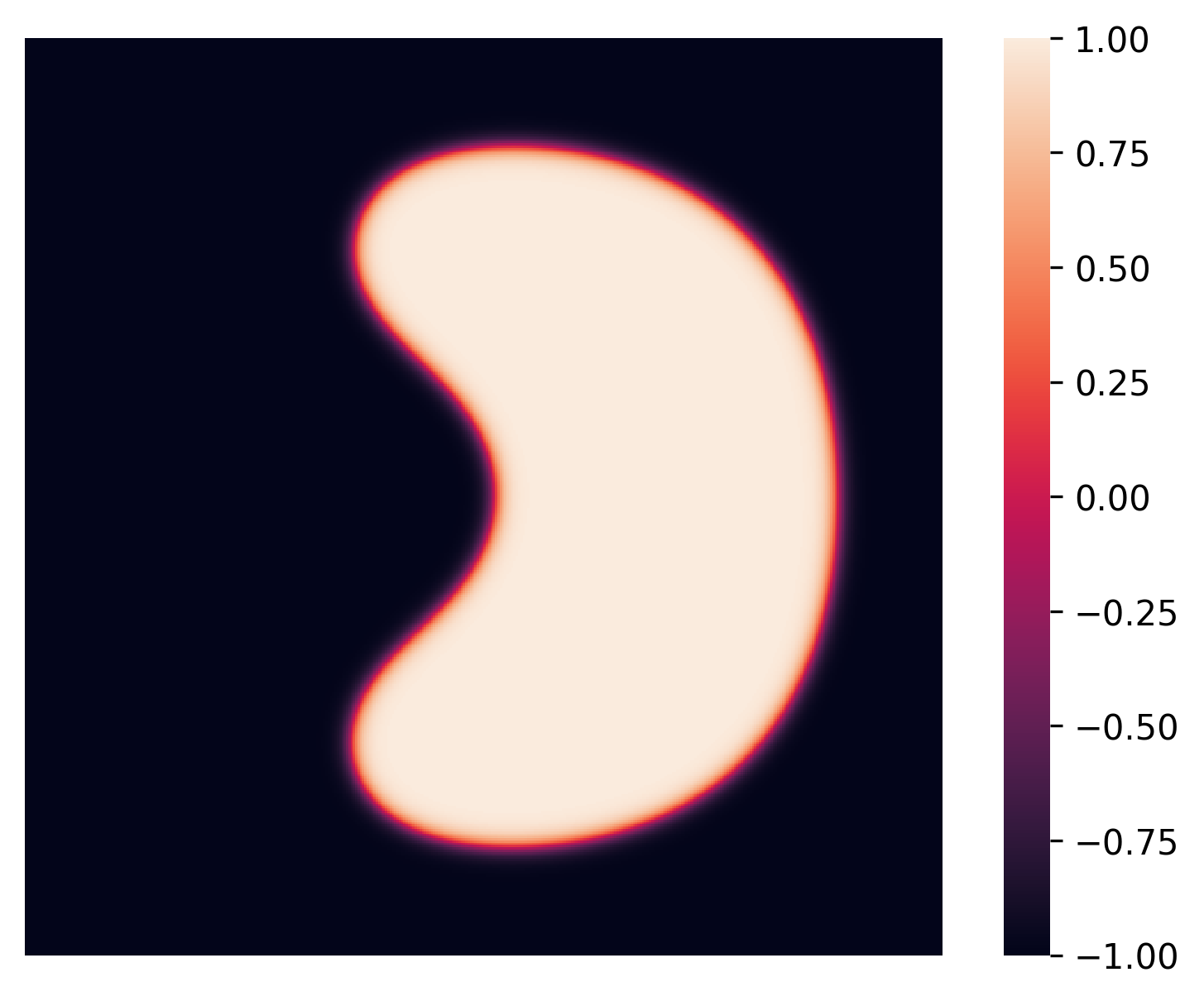}
    \includegraphics[height = 3.5cm, clip = true, trim = 0 0 2.3cm 0]{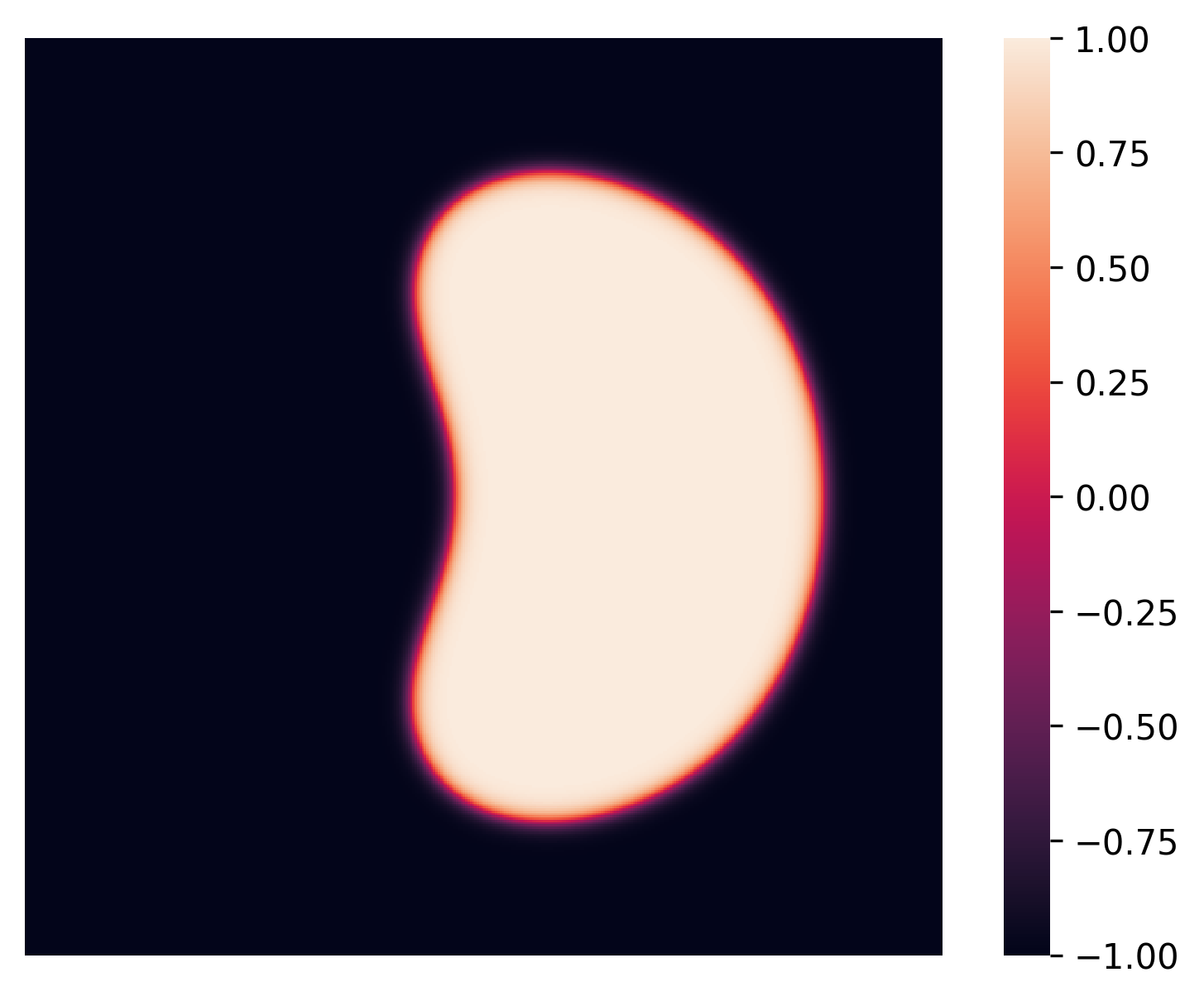}
    \includegraphics[height = 3.5cm, clip = true, trim = 0 0 2.3cm 0]{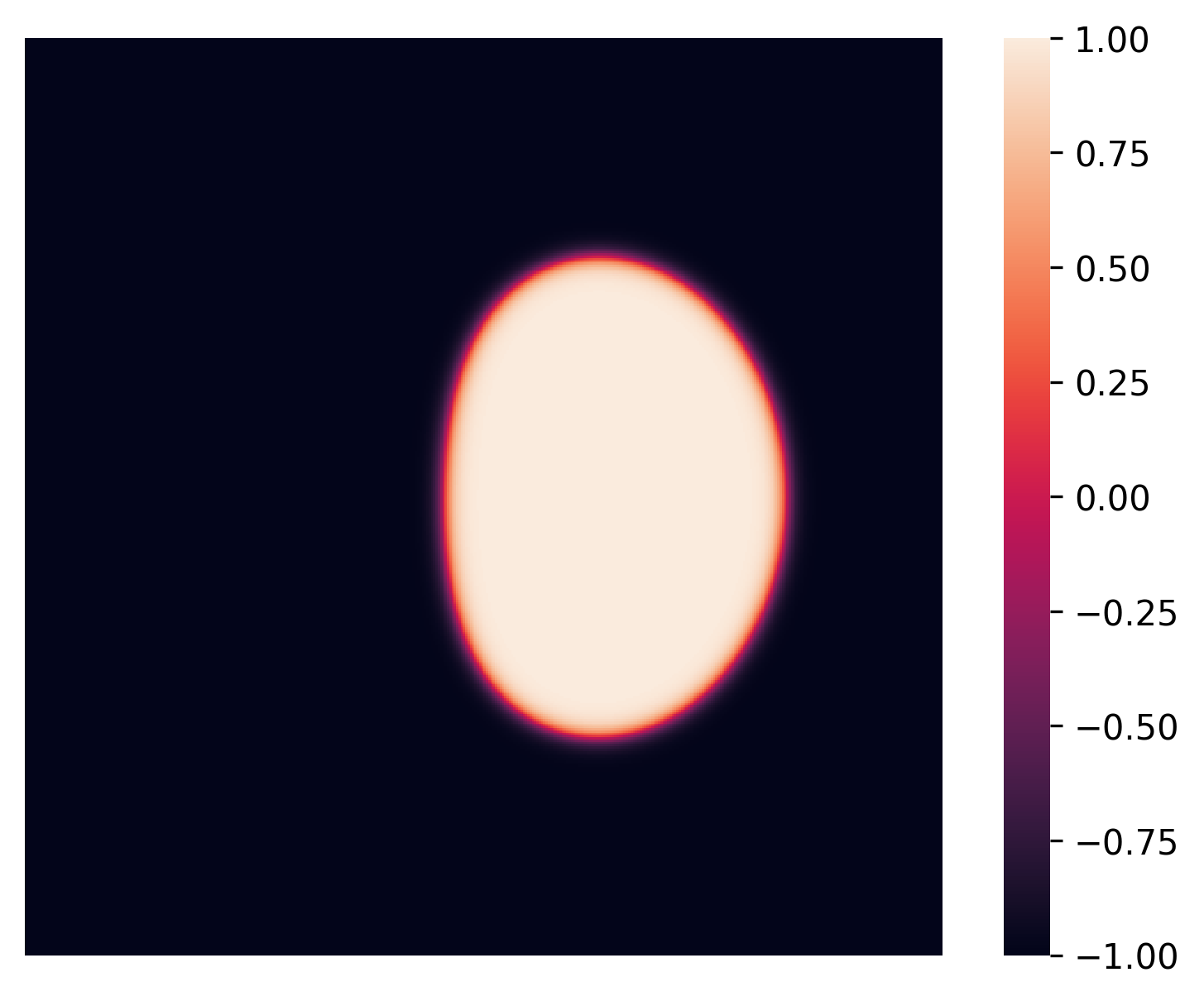}
    \includegraphics[height = 3.5cm]{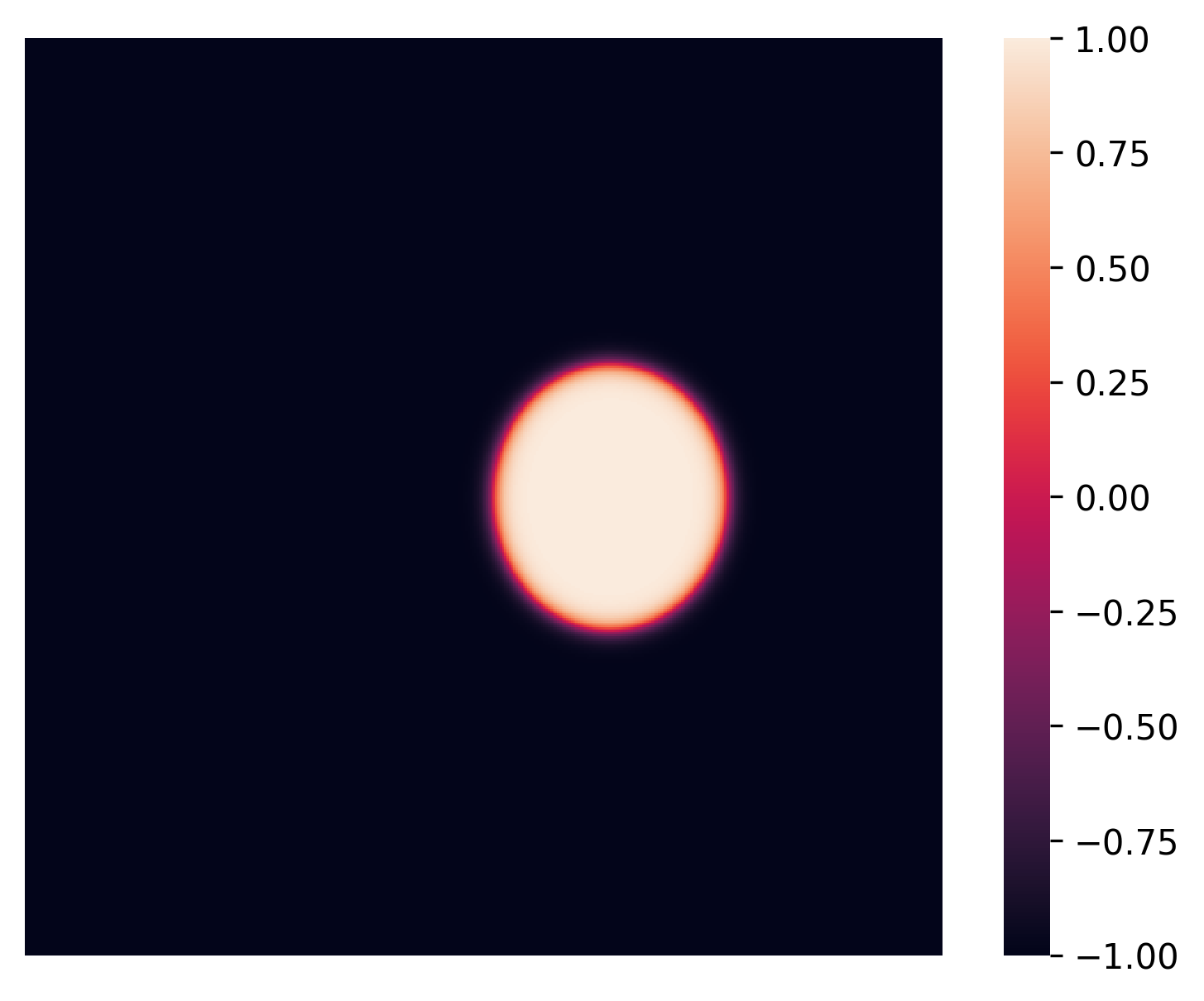}
    \end{flushleft}

\caption{
\label{figure mcf}
Evolution of a Jordan curve under the Allen-Cahn approximation to curve shortening flow at times $t\in \{0, 0.005, 0.01, 0.02, 0.03, 0.04, 0.06, 0.08\}$ (left to right, top to bottom).
}
\end{figure}

We compare the gradient flow and accelerated flow for the Ginzburg-Landau functional on the unit square with periodic boundary conditions. The gradient flow is implemented by a convex-concave splitting and the accelerated gradient flow is implemented as CINEMA on a $n\times n$-grid with $n=400$ with $\eps = 0.01$ and the coefficient of friction $\alpha = 3$ for the accelerated version. The time step size is $\tau = 10^{-5}$. The equation \eqref{eq cinema for ginzburg-landau}
is solved in the Fourier domain, transforming the right hand side by an FFT and the solution by the inverse FFT, and analogously for the regular Allen-Cahn equation.

The initial condition is (a relaxation of) the characteristic function of a C-shaped set. In order to start from a finite energy configuration, we take ten Allen-Cahn time steps with step size $\tilde\tau = 10^{-4}$ from the discontinuous initialization. 

In Figure \ref{figure mcf}, we illustrate the numerical behavior of the Allen-Cahn approximation to mean curvature flow: corners are instantaneously smoothed out and the curve convexifies before disappearing in a `round point' (i.e.\ the shape becomes approximately circular at the time of disappearance). This behavior is representative of MCF for curves in dimension two (also known as curve shortening flow) where it is known that no singularities occur if the inital curve is embedded (Gage-Hamilton-Grayson Theorem, \cite{gage1983isoperimetric, gage1984curve, gage1986heat, grayson1987heat, andrews2011curvature}).

In Figure \ref{figure accelerated mcf}, we illustrate the accelerated Allen-Cahn approximation to the geometric motion governed by \eqref{eq singular limit}. There are notable geometric differences to the regular Allen-Cahn equation.

At the ninety degree corners in the initial condition, the curvature is infinite (or essentially infinite after the Allen-Cahn relaxation), so the corner begins to instantaneously move inwards with high velocity. On the straight interfaces, on the other hand, the curvature is zero, so the interfaces remain initially unchanged due to the finite speed of propagation of information. In effect, this creates new corners with sharp angles of forty-five degrees. There is no parabolic `smoothing' effect in the evolution.

In fast-moving segments of the boundary, the transition between the potential wells is markedly steeper than in stationary segments. This increased sharpness is particularly visible in the vertical segments of the boundary in the second row (also compare Figure \ref{figure corrector}).

The `vertical' momentum in the accelerated Allen-Cahn equation translates into `horizontal' momentum for the interface: After the white area disappears in the bottom row, it reappears again.

\begin{figure}
\begin{flushleft}
    \includegraphics[height = 3.5cm, clip = true, trim = 0 0 2.3cm 0]{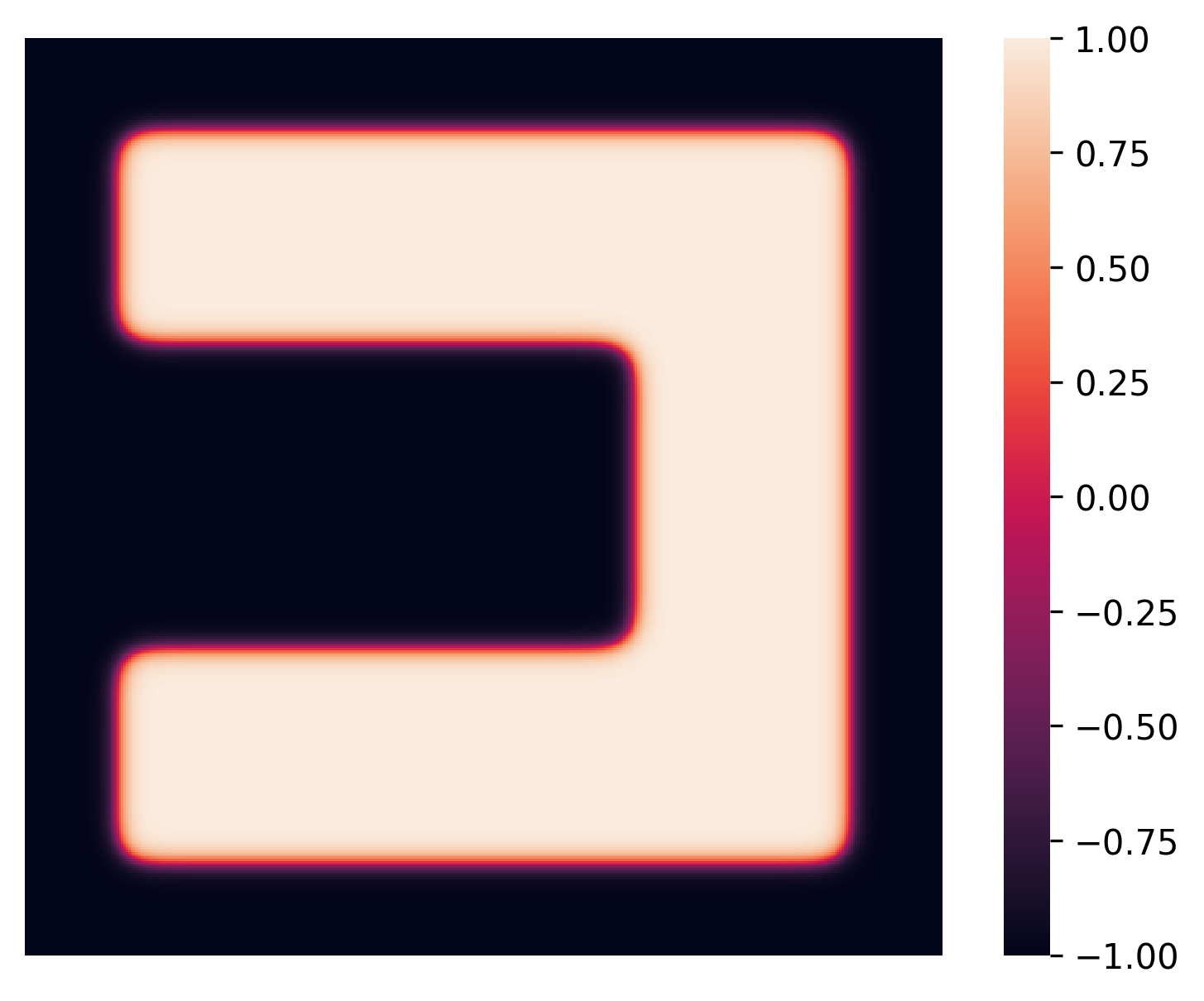}
    \includegraphics[height = 3.5cm, clip = true, trim = 0 0 2.3cm 0]{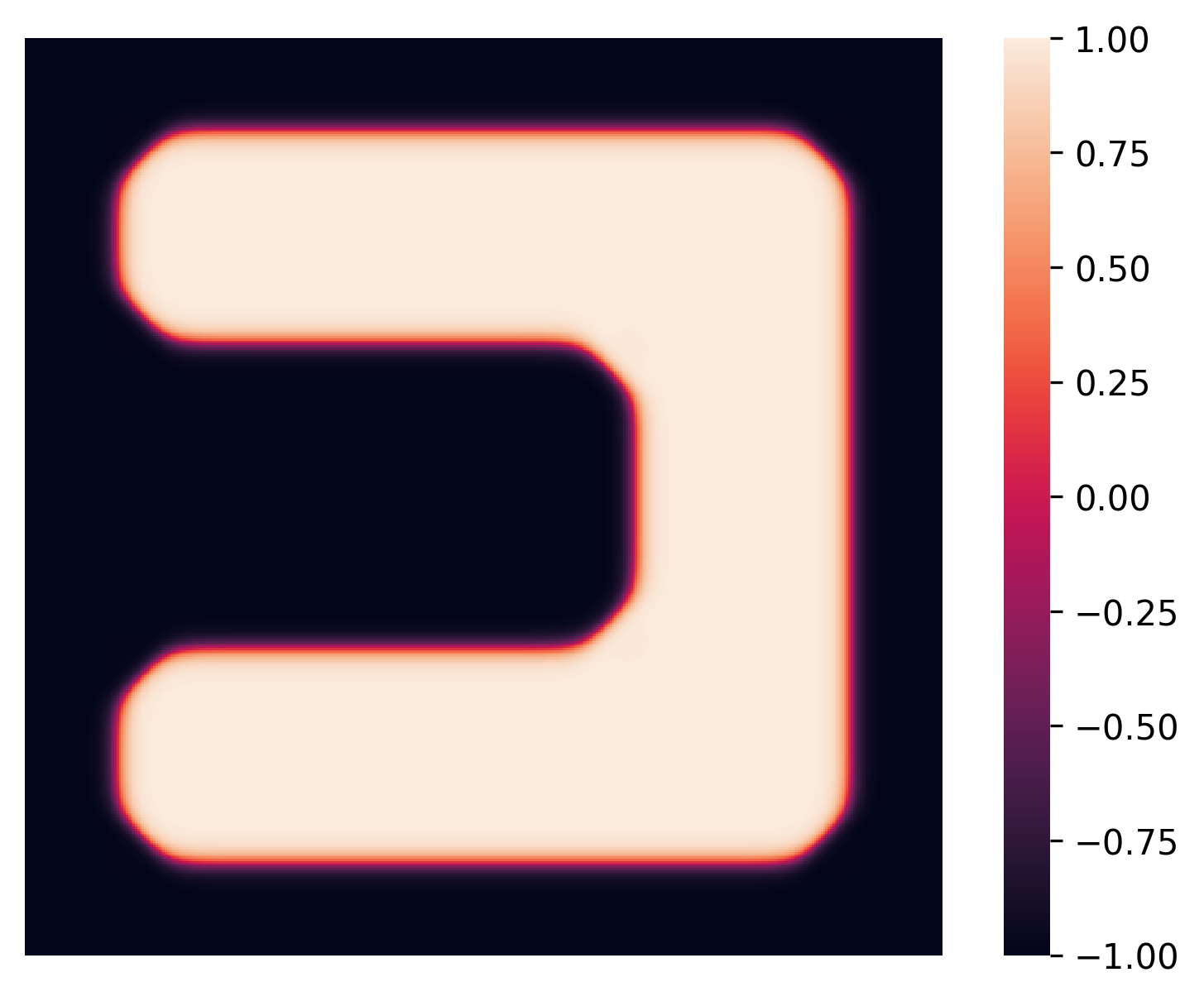}
    \includegraphics[height = 3.5cm, clip = true, trim = 0 0 2.3cm 0]{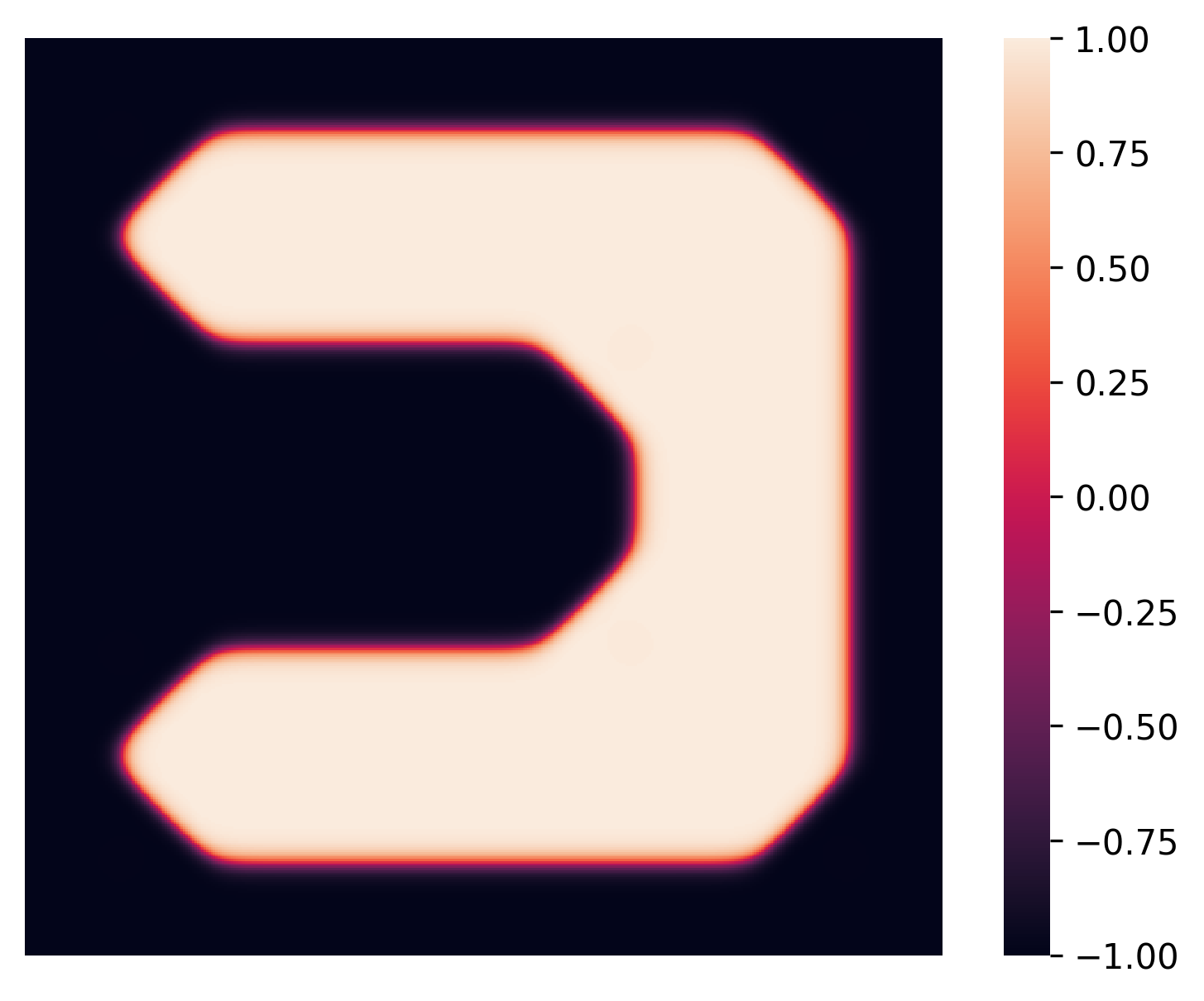}
    \includegraphics[height = 3.5cm, clip = true, trim = 0 0 2.3cm 0]{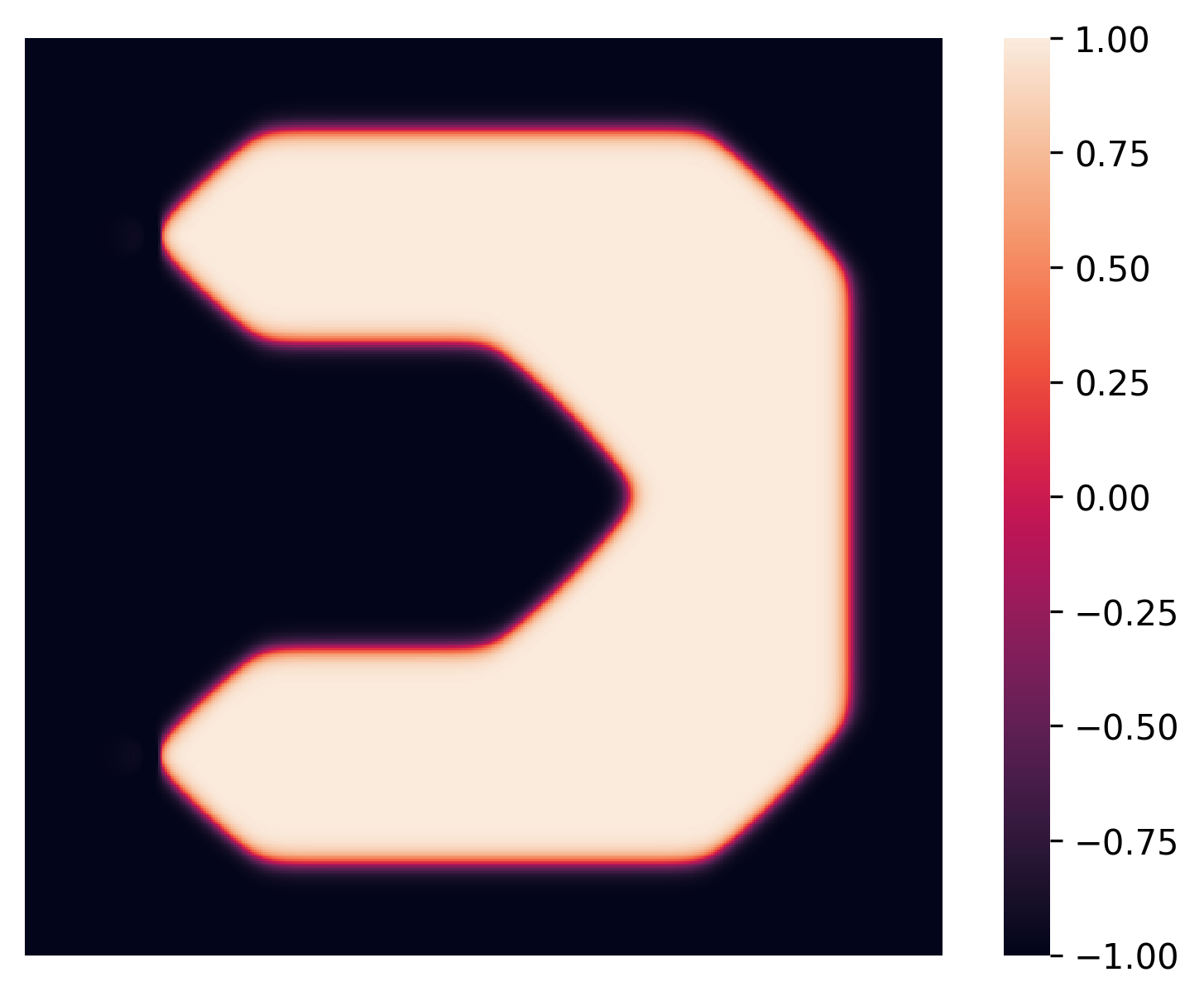}

    \includegraphics[height = 3.5cm, clip = true, trim = 0 0 2.3cm 0]{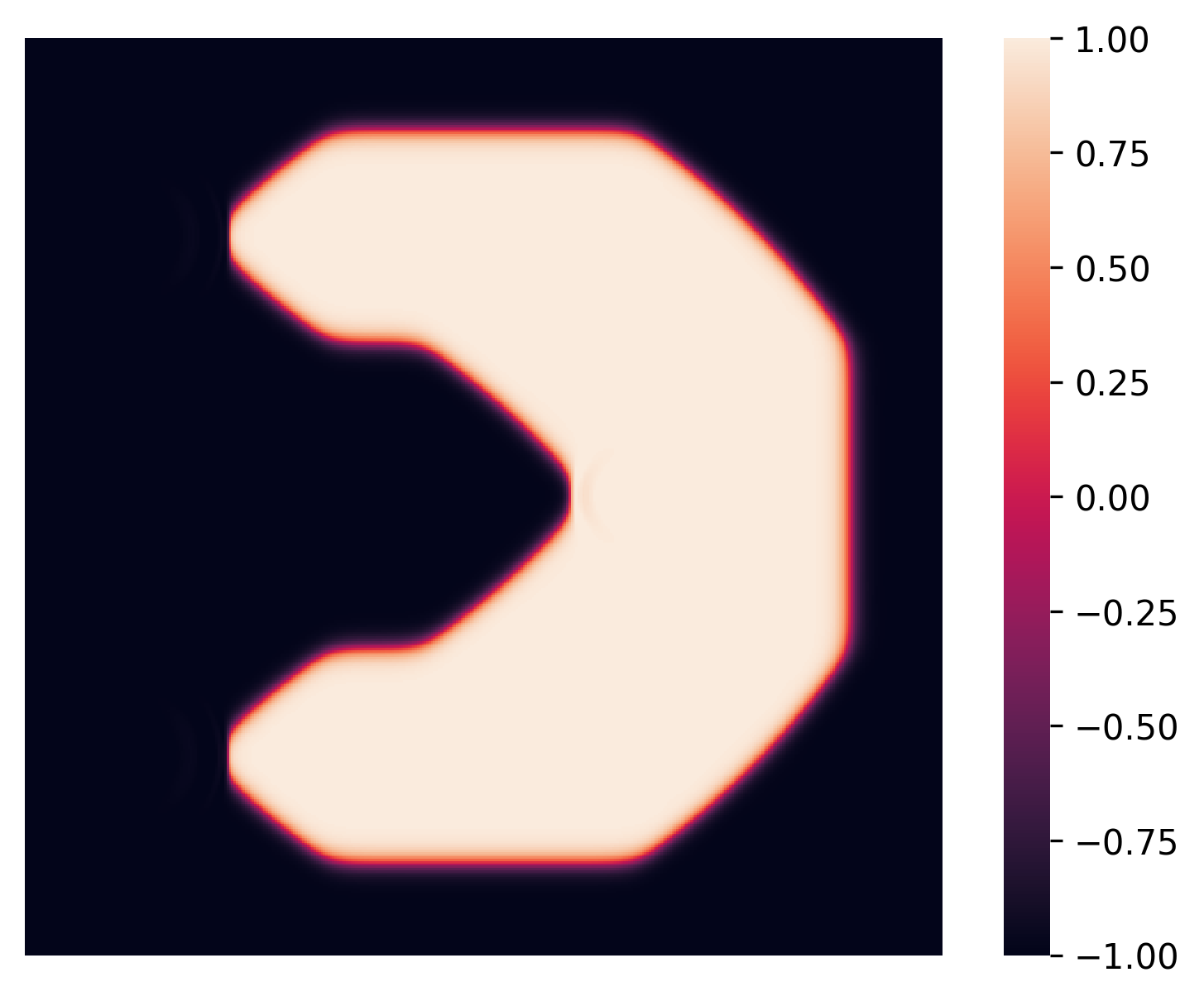}
    \includegraphics[height = 3.5cm, clip = true, trim = 0 0 2.3cm 0]{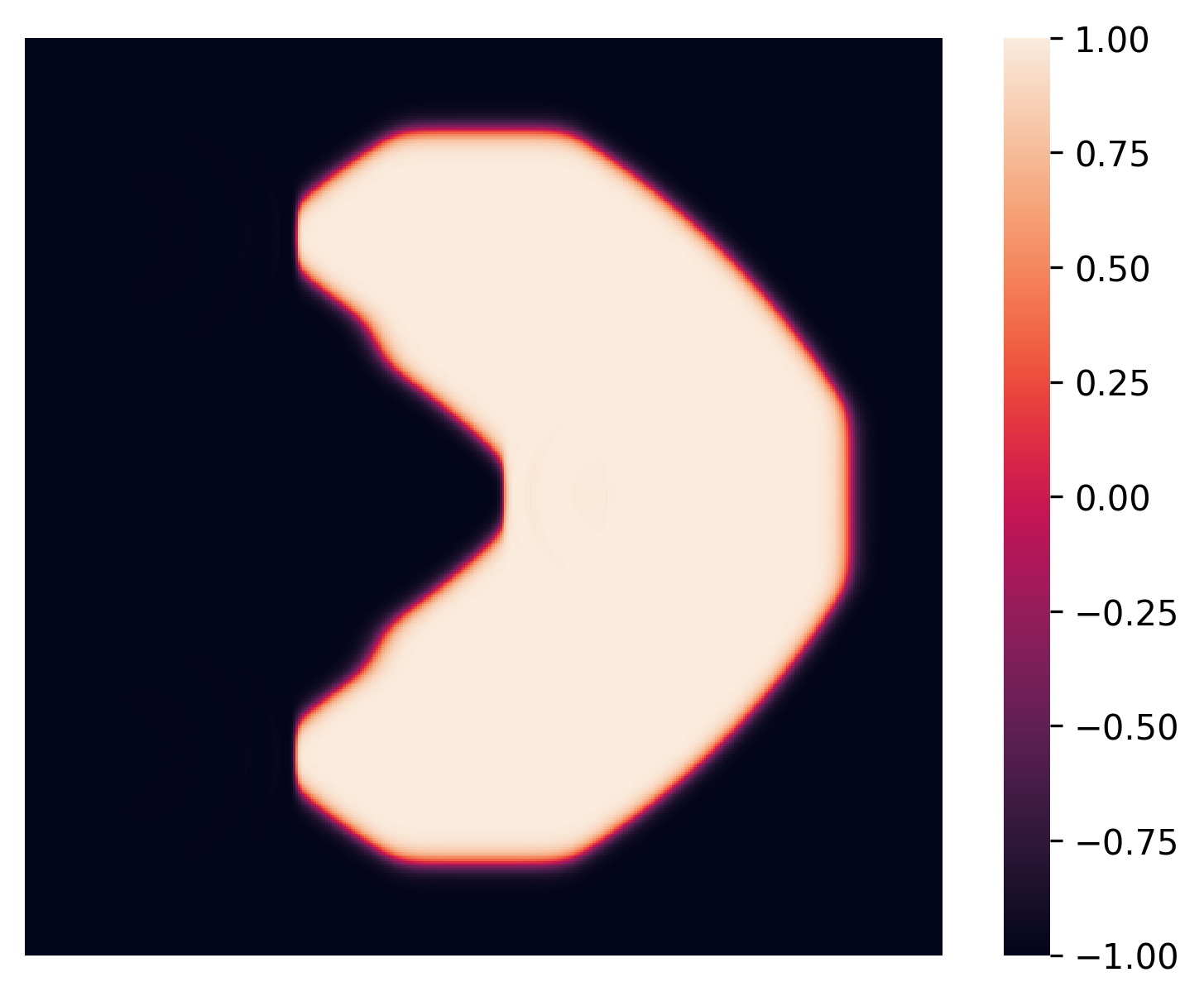}
    \includegraphics[height = 3.5cm, clip = true, trim = 0 0 2.3cm 0]{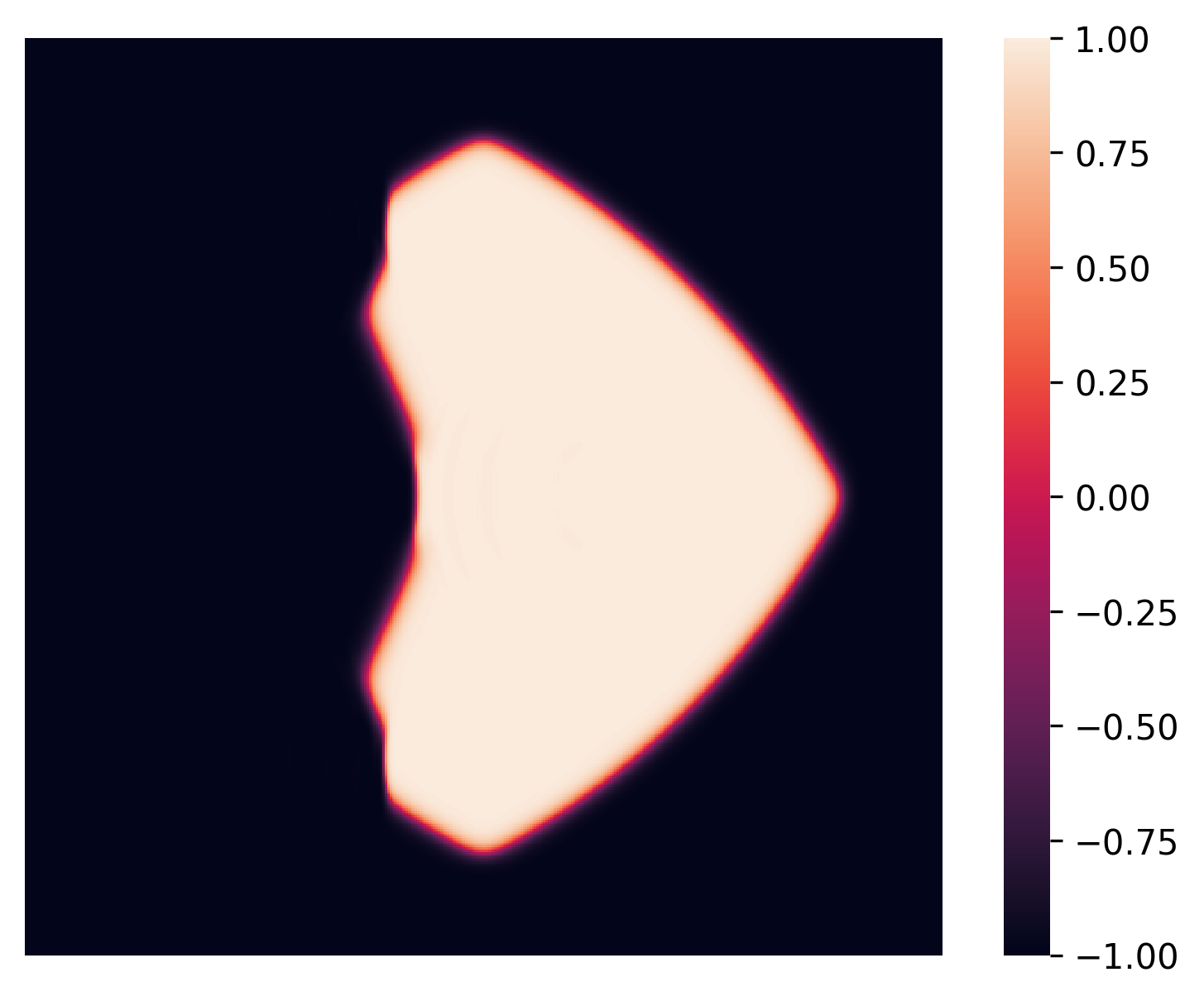}
    \includegraphics[height = 3.5cm, clip = true, trim = 0 0 2.3cm 0]{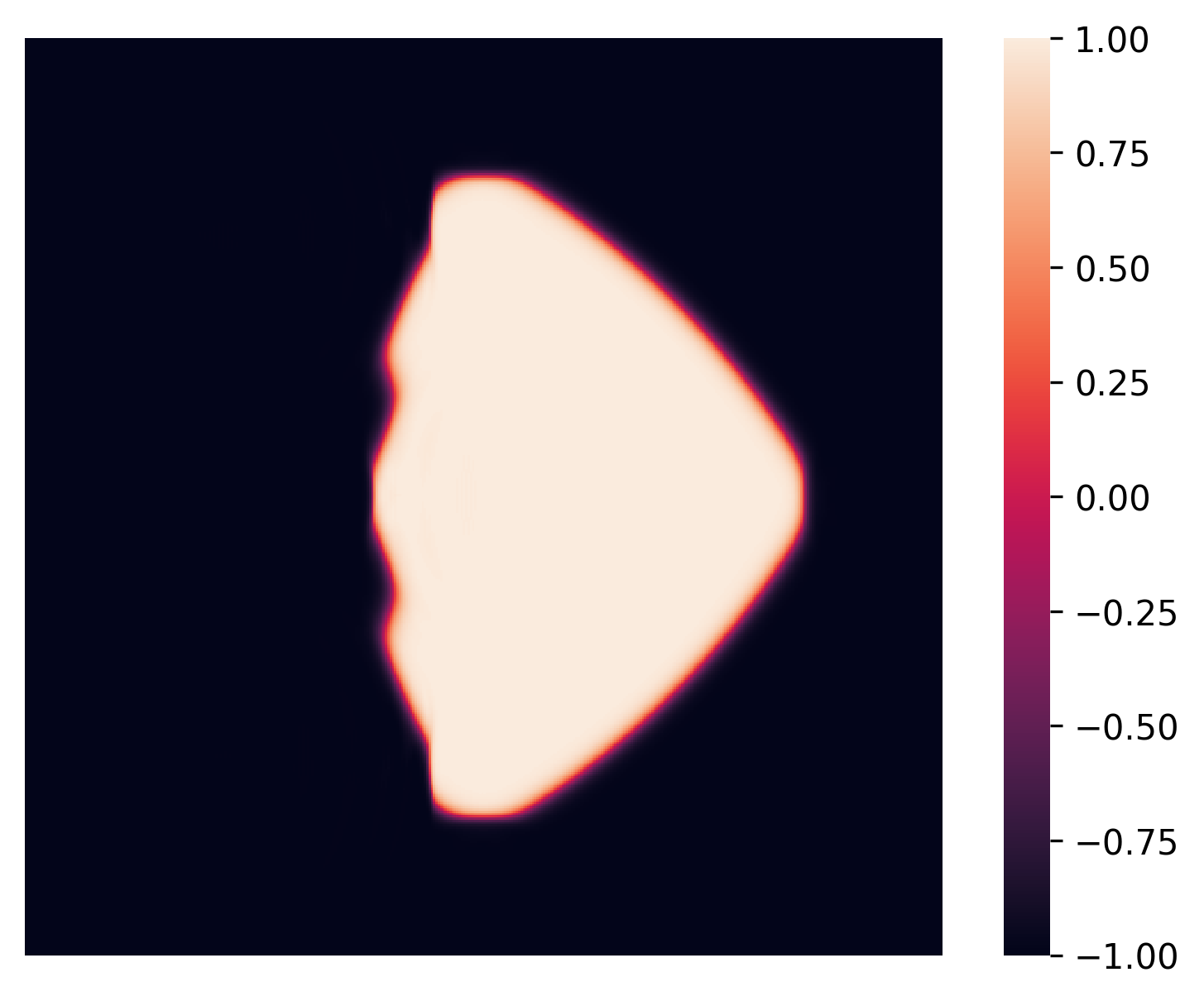}

    \includegraphics[height = 3.5cm, clip = true, trim = 0 0 2.3cm 0]{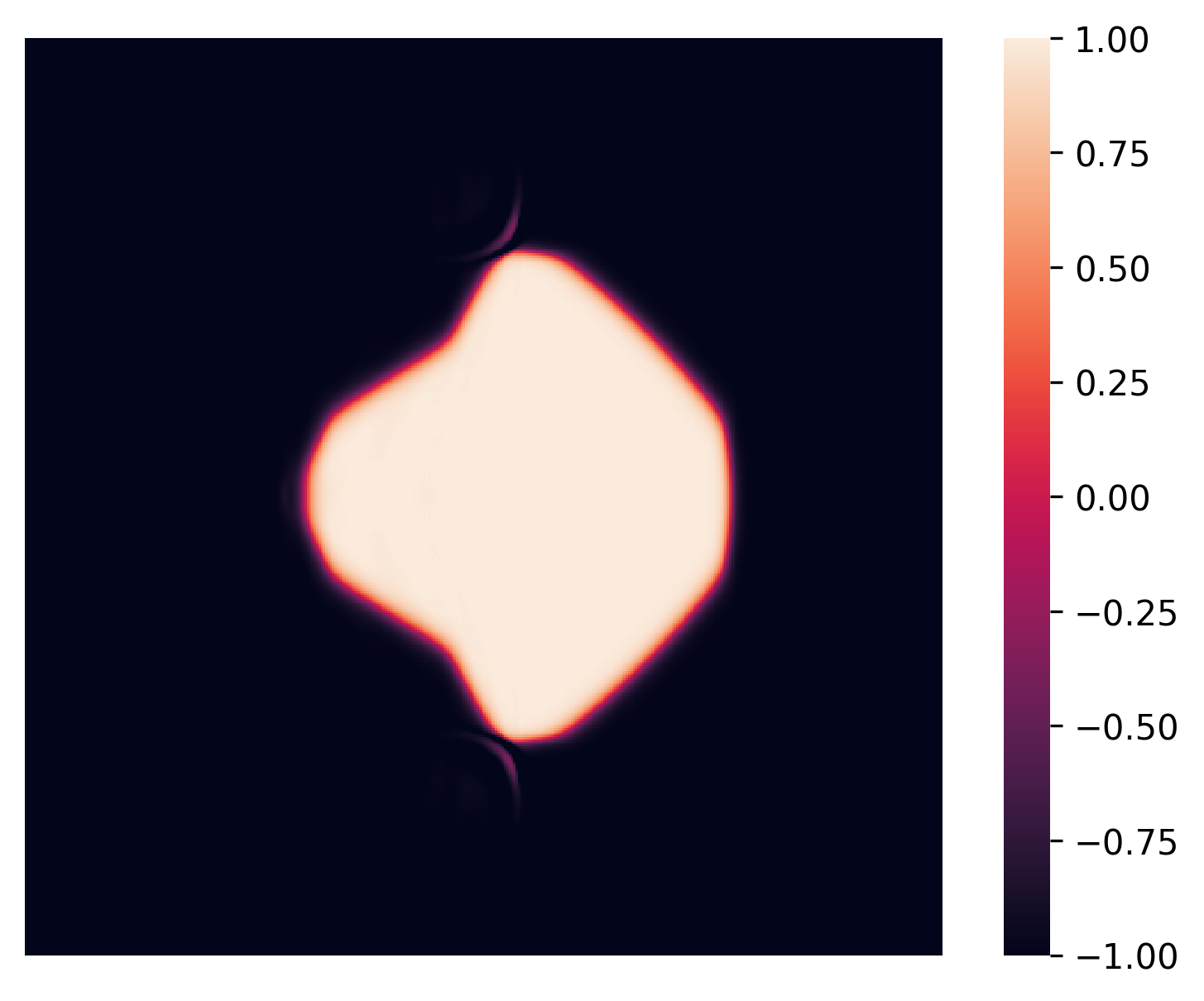}
    \includegraphics[height = 3.5cm, clip = true, trim = 0 0 2.3cm 0]{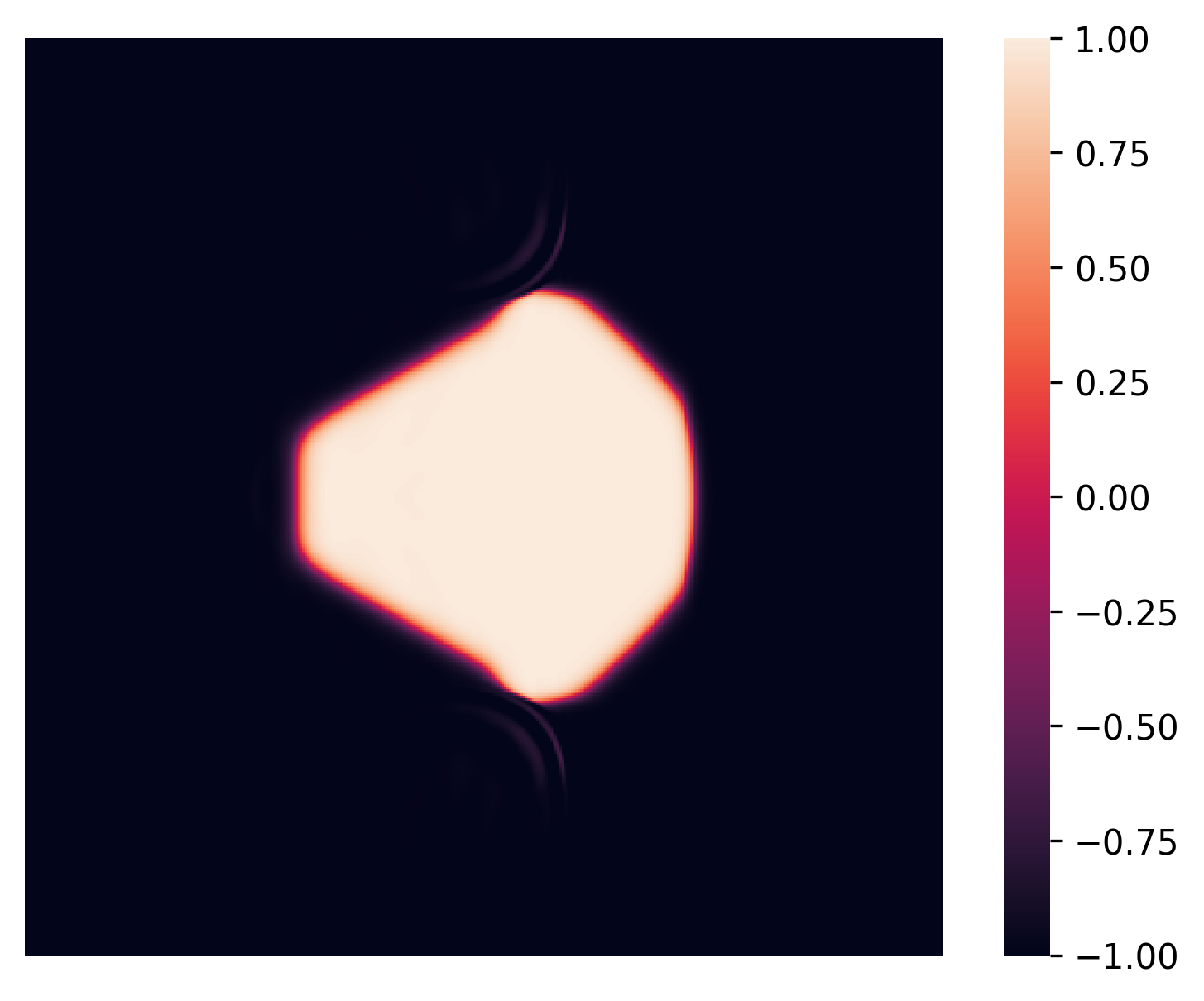}
    \includegraphics[height = 3.5cm, clip = true, trim = 0 0 2.3cm 0]{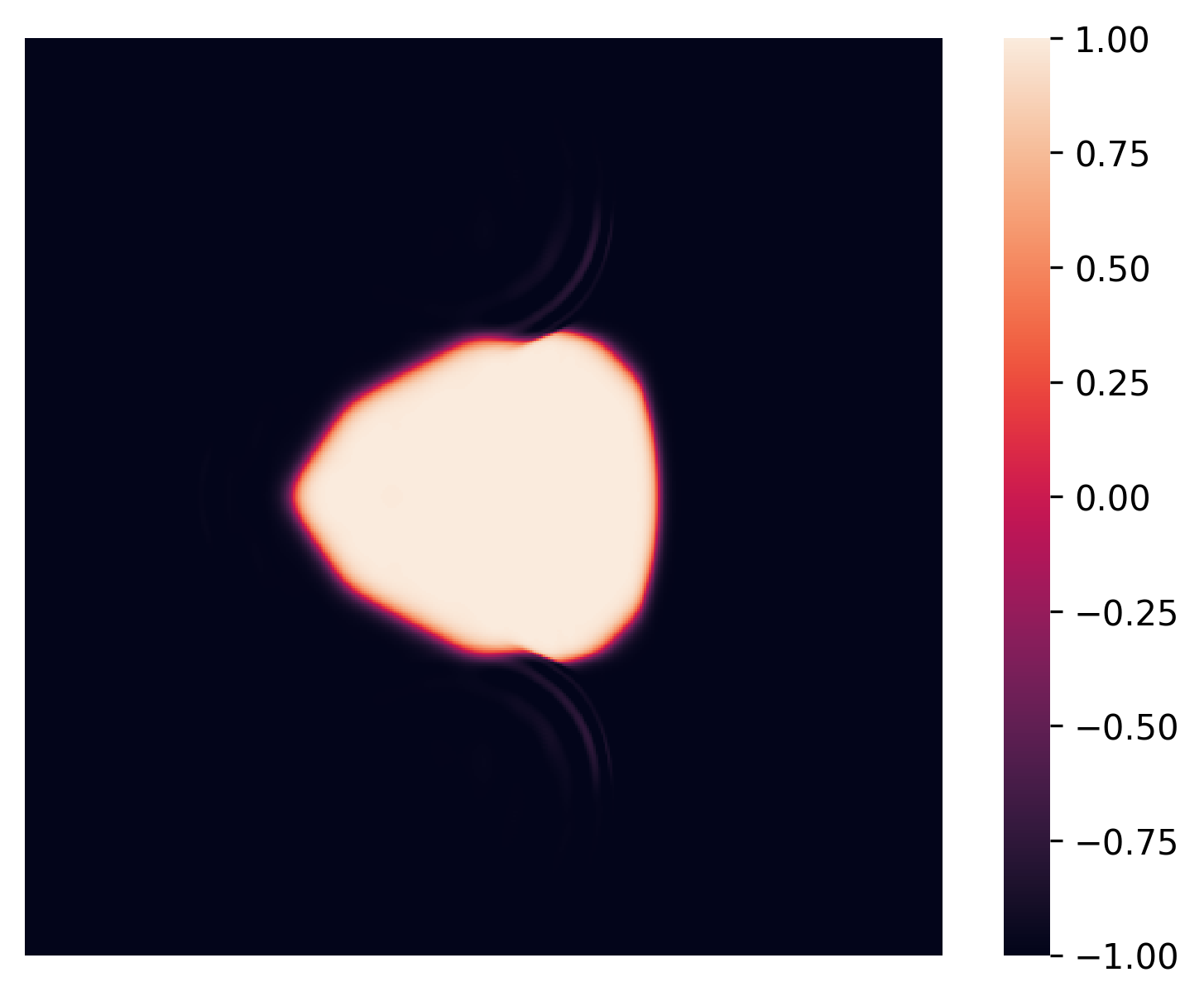}
    \includegraphics[height = 3.5cm, clip = true, trim = 0 0 2.3cm 0]{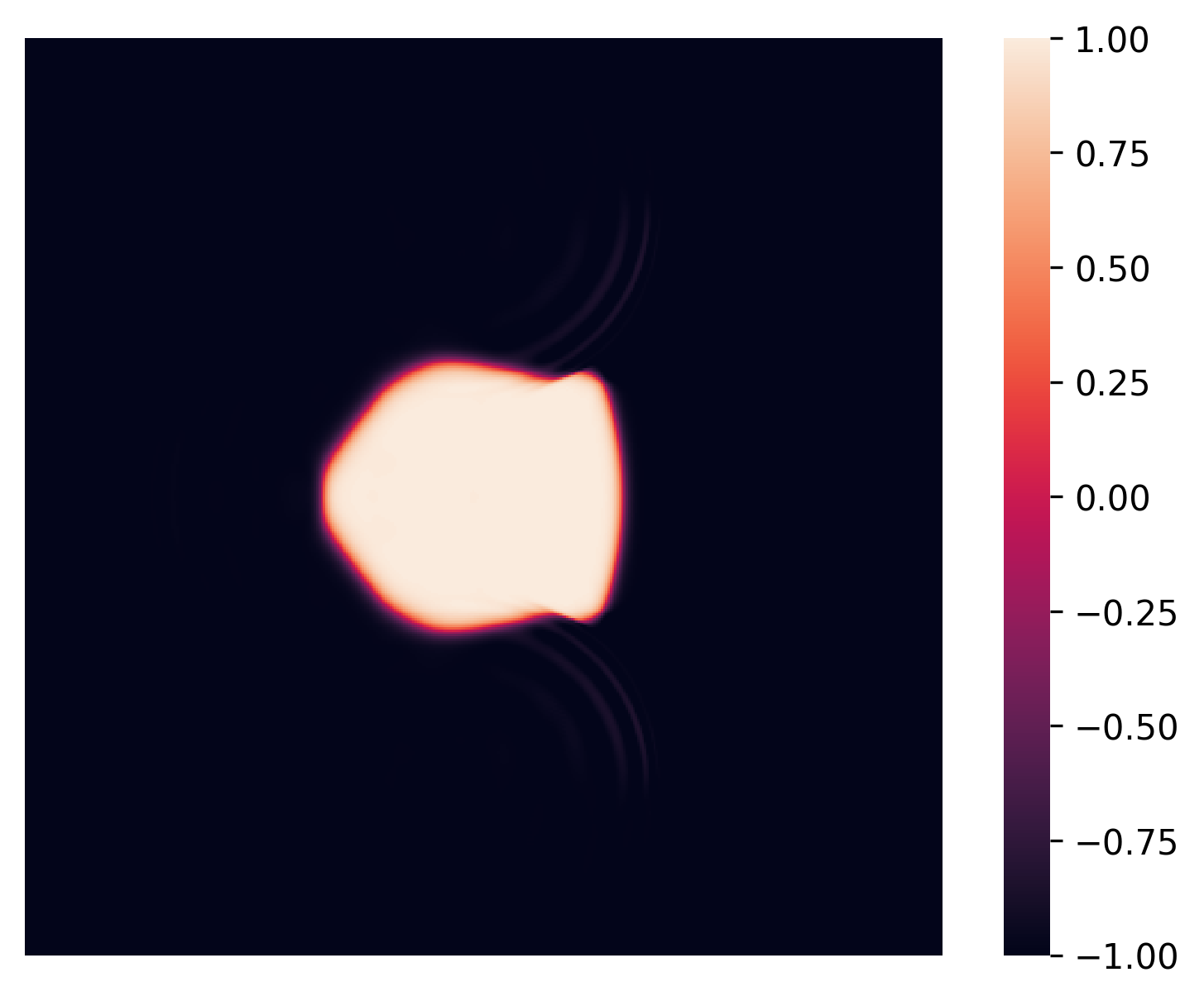}

    \includegraphics[height = 3.5cm, clip = true, trim = 0 0 2.3cm 0]{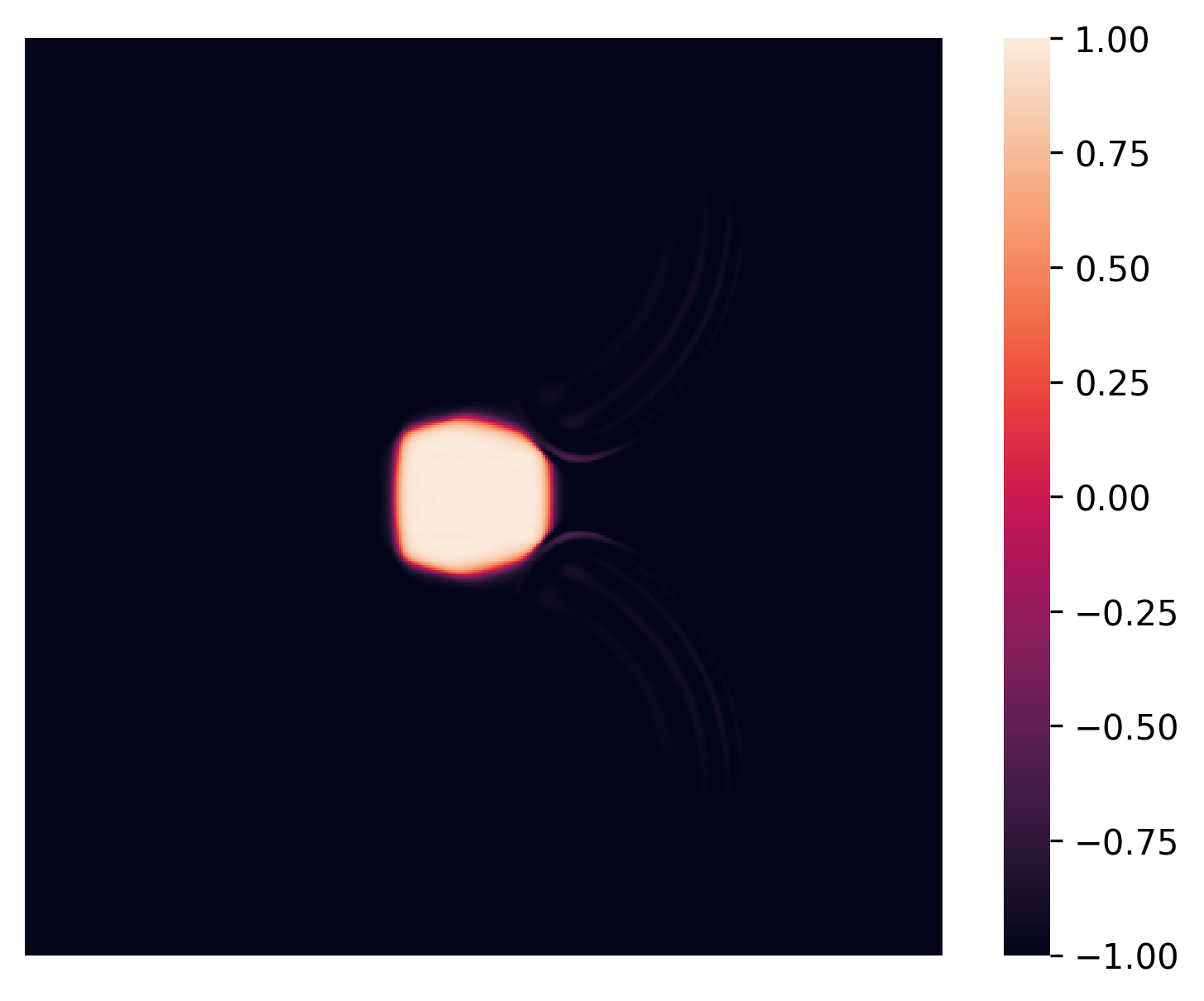}
    \includegraphics[height = 3.5cm, clip = true, trim = 0 0 2.3cm 0]{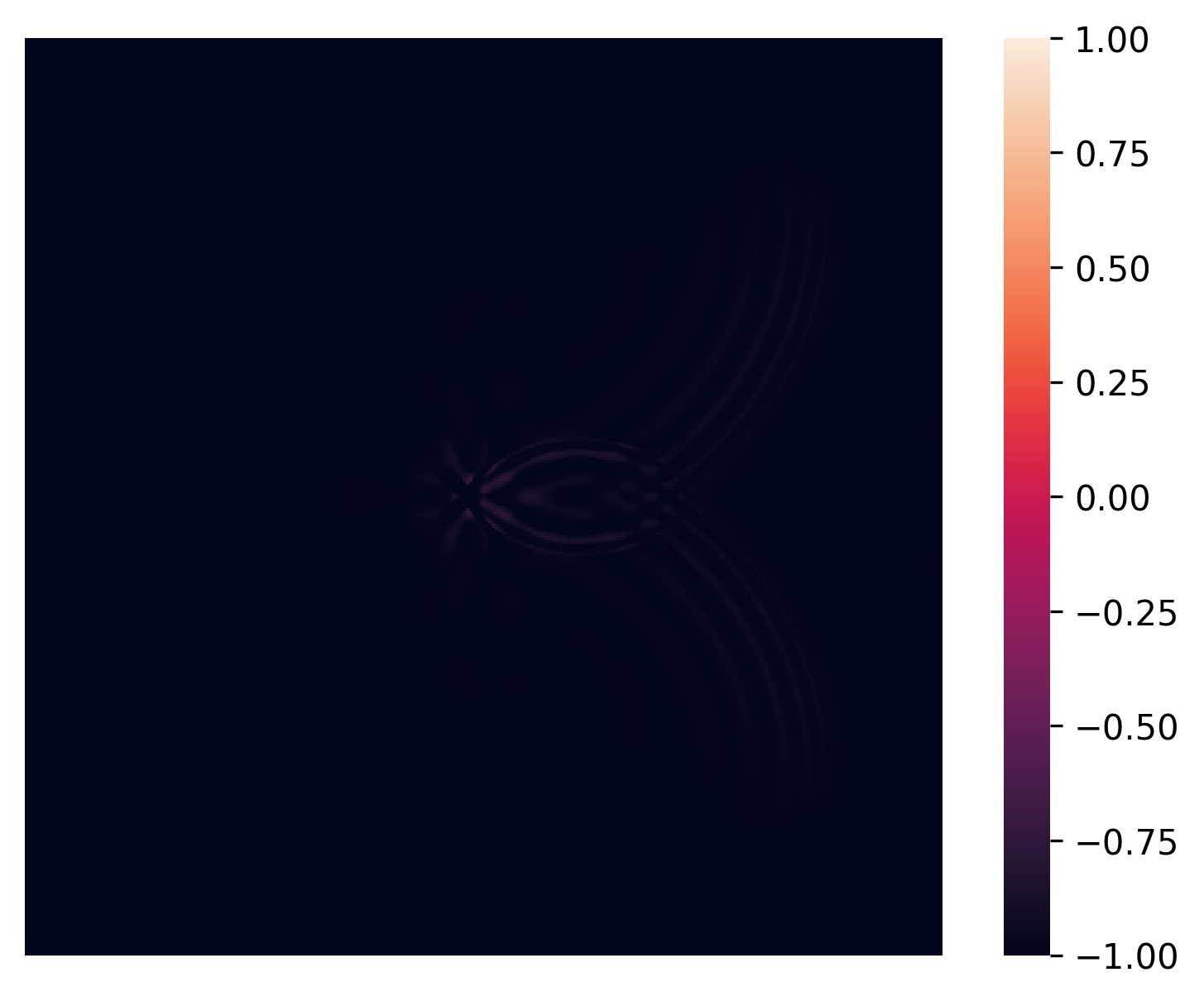}
    \includegraphics[height = 3.5cm, clip = true, trim = 0 0 2.3cm 0]{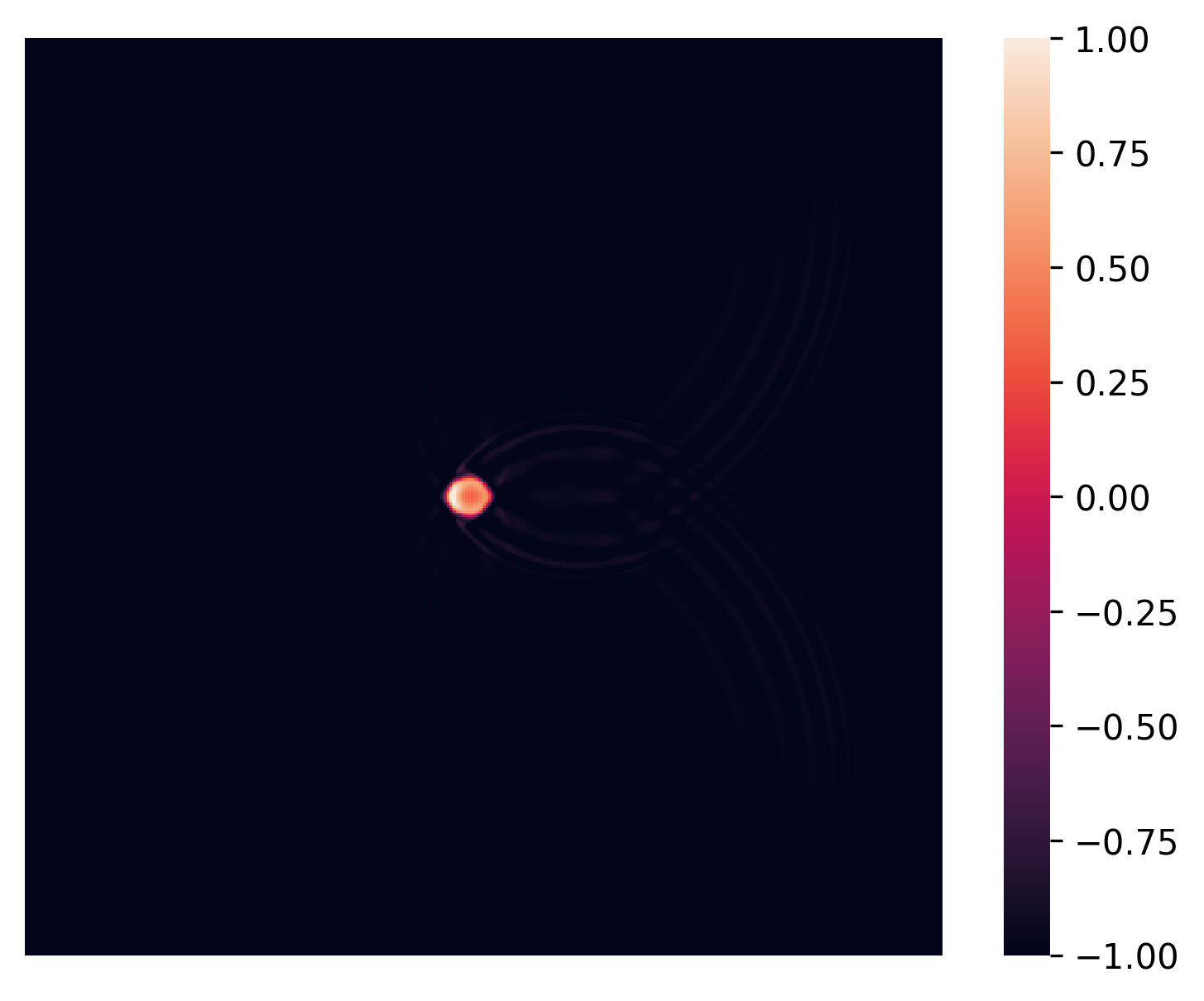}
    \includegraphics[height = 3.5cm]{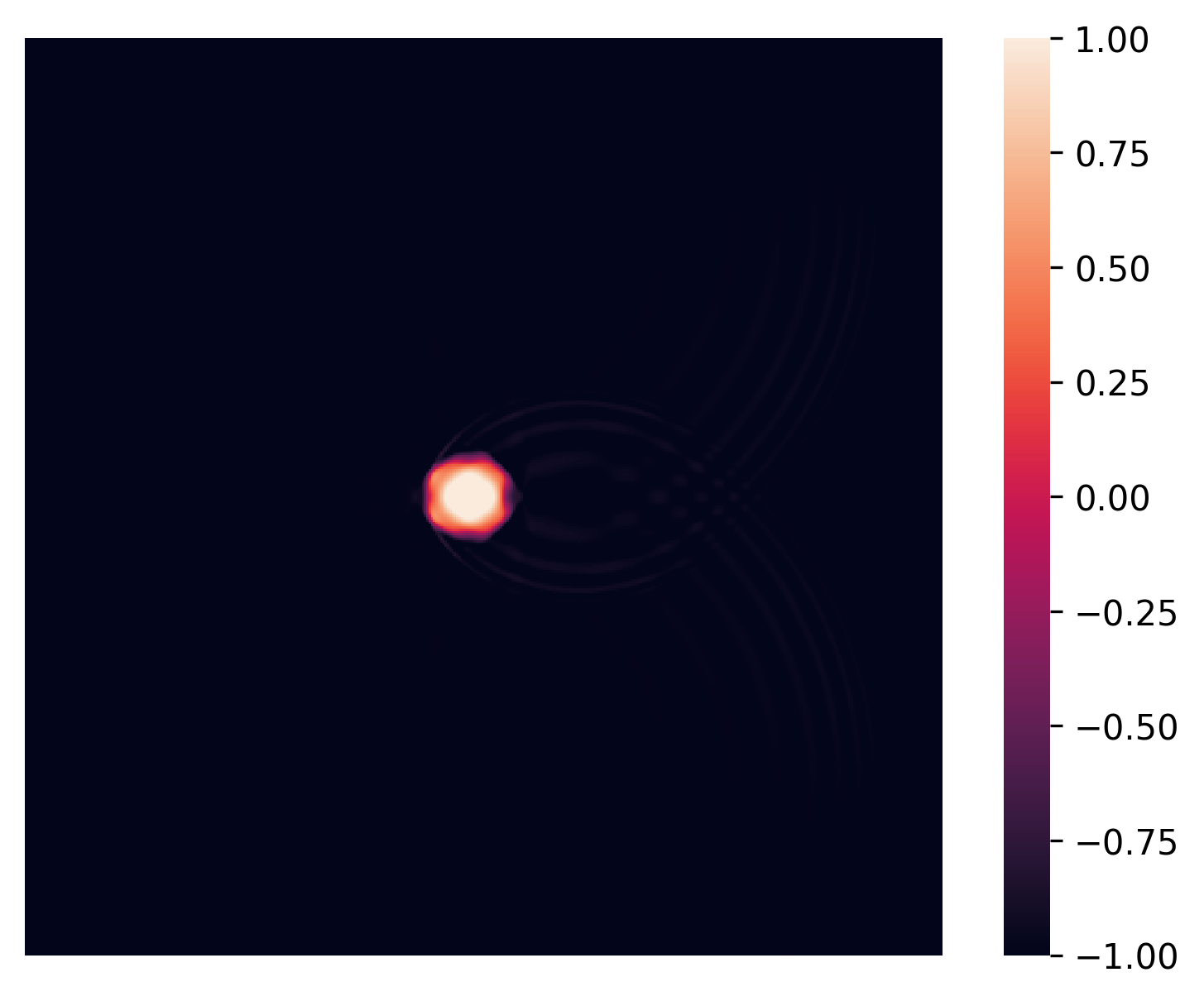}
\end{flushleft}

\caption{
\label{figure accelerated mcf}
Evolution of a Jordan curve under the accelerated Allen-Cahn approximation to `accelerated curve shortening flow' at times $t\in \{0, 0.05, 0.1, 0.15\}$ (top row, left to right), $t\in\{0.225, 0.3, 0.4, 0.45\}$ (second row) $t\in\{0.55, 0.6, 0.65, 0.7\}$ (third row) and $t\in\{0.8 0.9, 0.925, 0.95\}$ (bottom row).
}
\end{figure}

We observe the decrease of energy in Figure \ref{figure verifying singular limit}. Notably, the Ginzburg-Landau energy is not monotone decreasing along the accelerated Allen-Cahn equation, but the total energy (the sum of Ginzburg-Landau energy and kinetic energy) is. There is a singularity in the Ginzburg-Landau energy curve at the time that the shape `disappears' briefly, but it is less pronounced for large $\alpha$. Unsurprisingly, the energy decrease is much faster along the Allen-Cahn equation in `physical' time.


In Figure \ref{figure large time step}, we consider the evolution of the same initial condition with the FISTA discretization and the much larger time step size $\eta = \tau^2 = \tau =1$ and $\rho = 1/ (1+\alpha \tau)$ for $\alpha = 0.1$. In this regime, the momentum method geometrically resembles mean curvature more than the accelerated mean curvature flow and does not develop non-smooth interfaces. The curve shrinks significantly faster than under the convex-concave splitting discretization of the Allen-Cahn equation, which is known to be limited by $\eps^2$ in this setting due to \cite{related_splitting_paper}.

\begin{figure}
    \centering
    \includegraphics[width=0.24\linewidth]{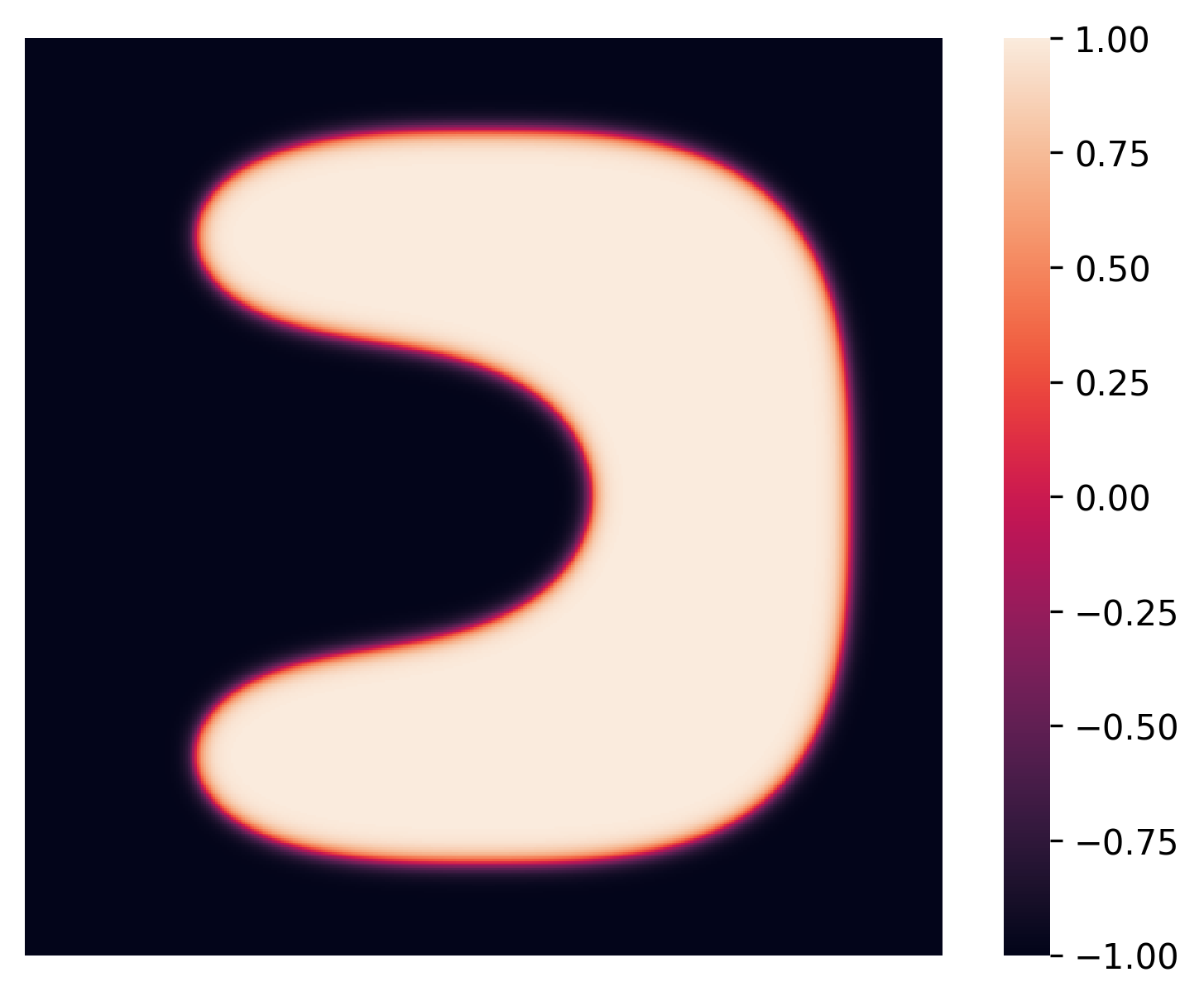}
    \includegraphics[width=0.24\linewidth]{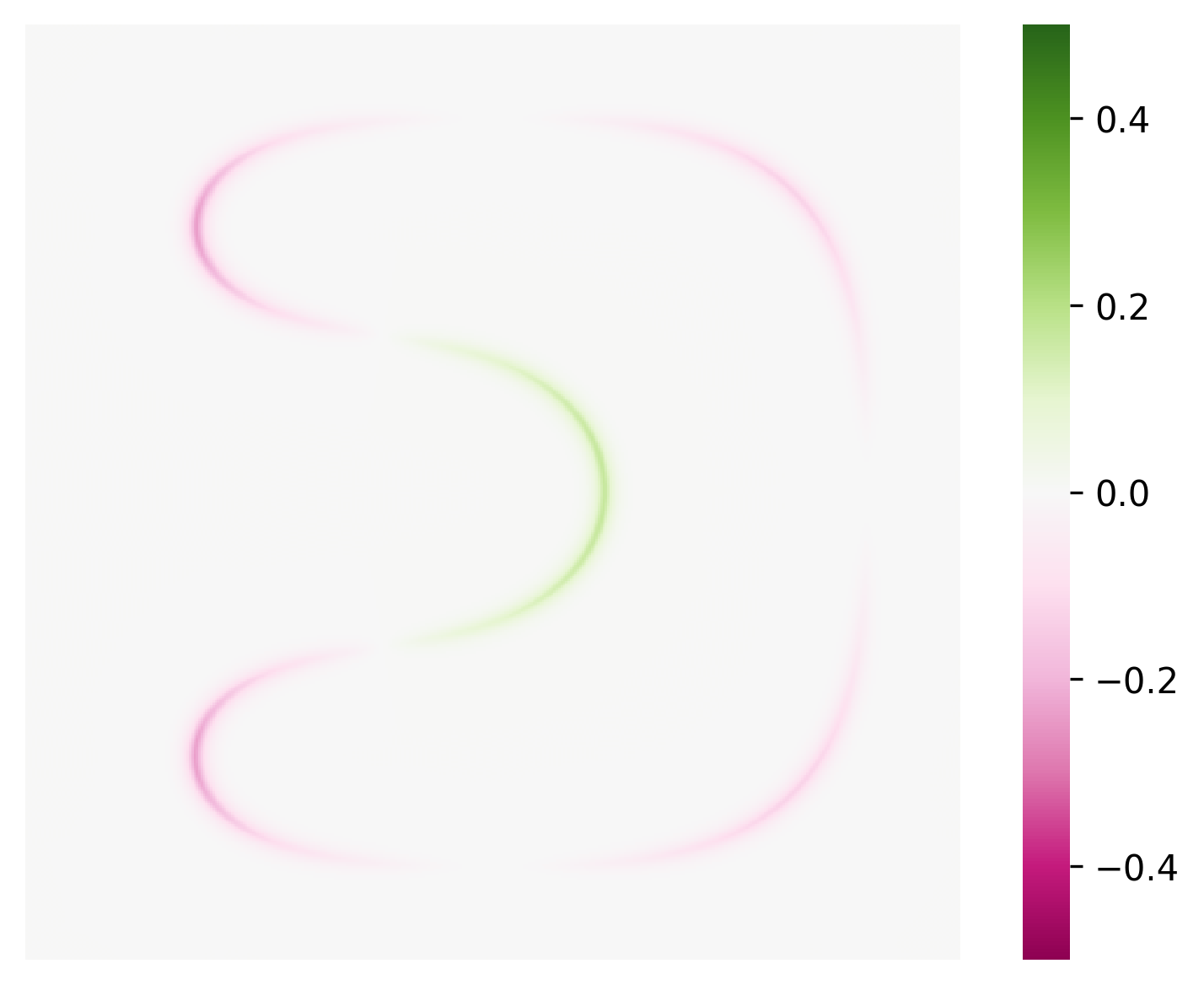}
    \includegraphics[width=0.24\linewidth]{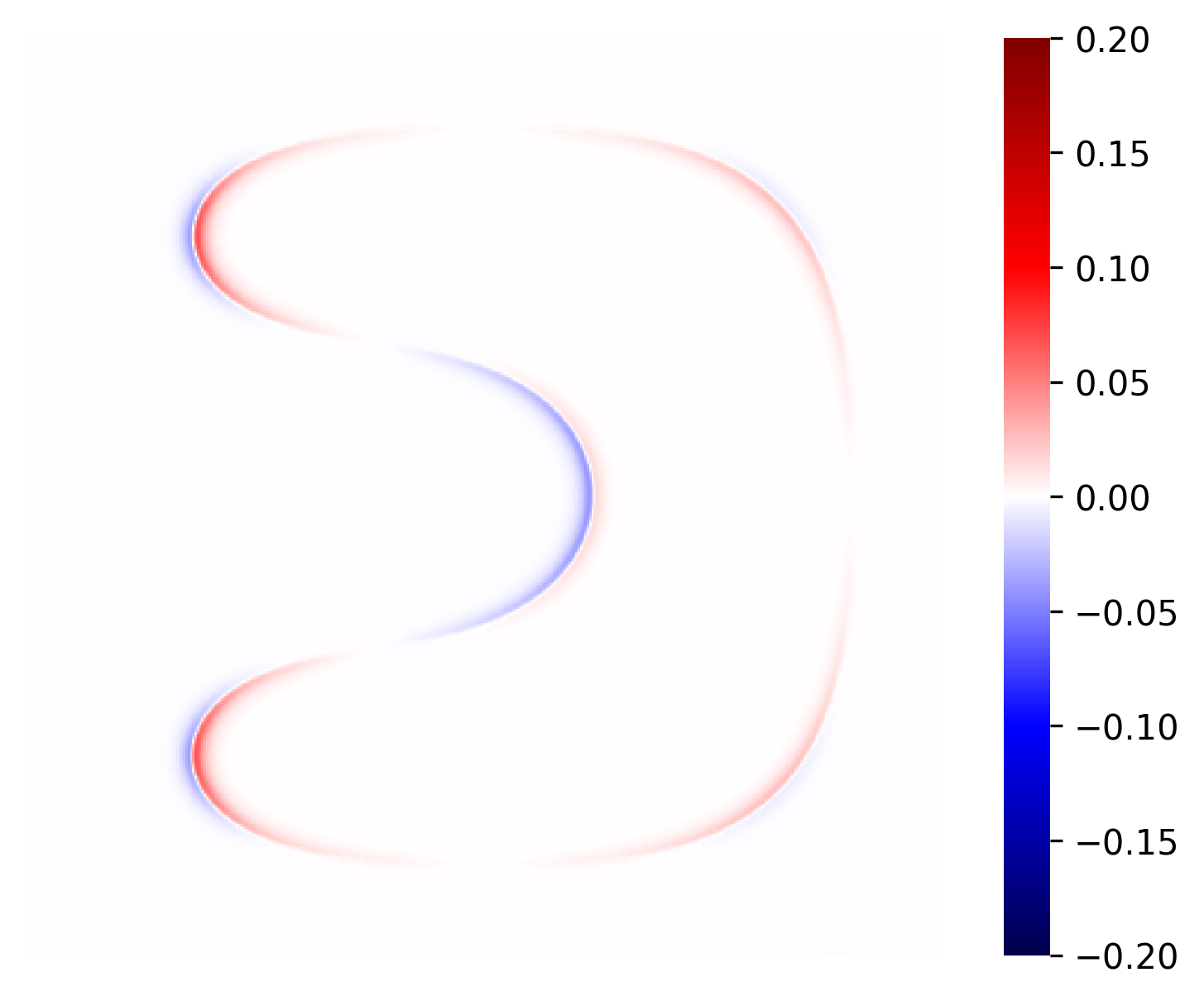}
    \includegraphics[width=0.24\linewidth]{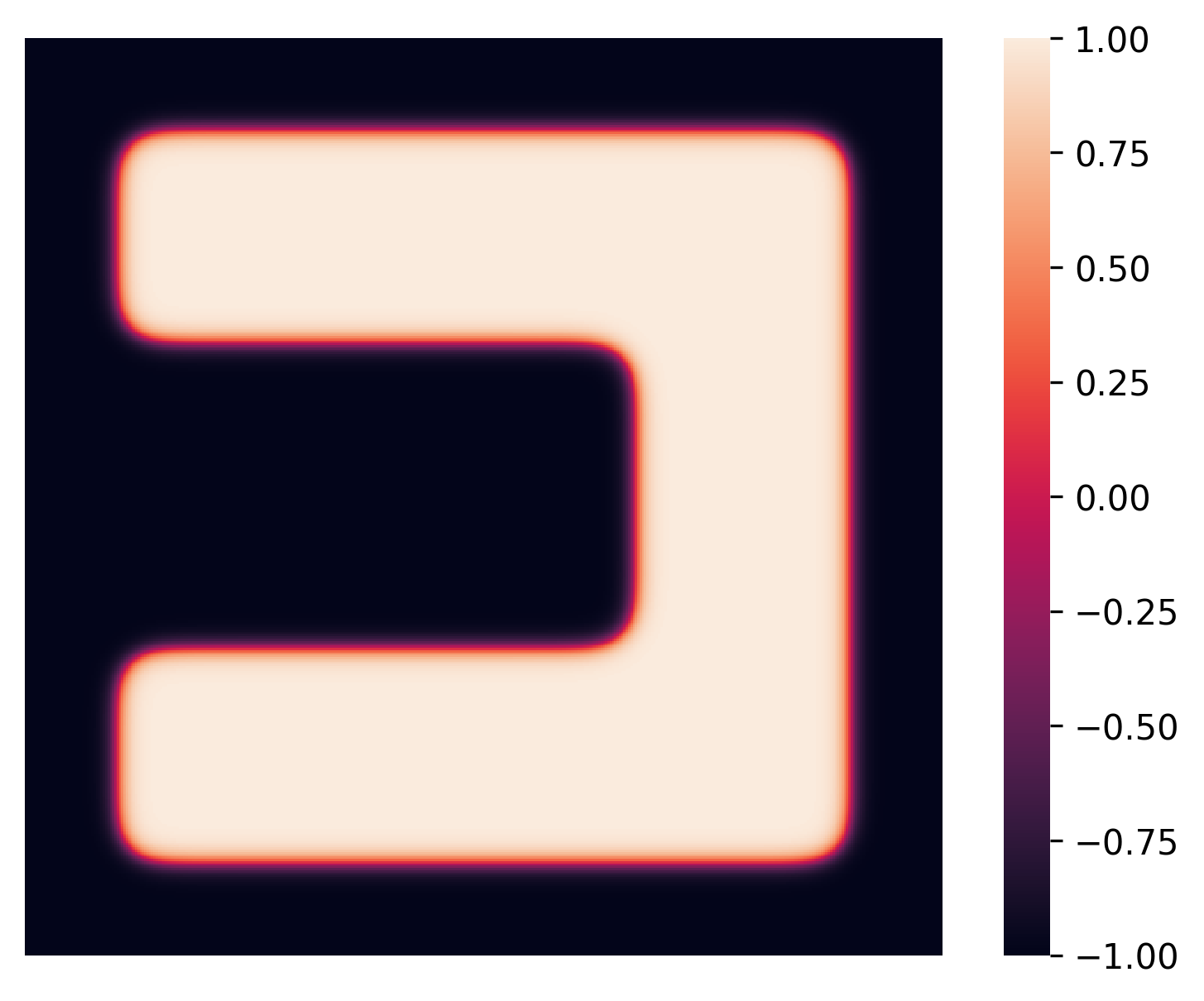}

    \includegraphics[width=0.24\linewidth]{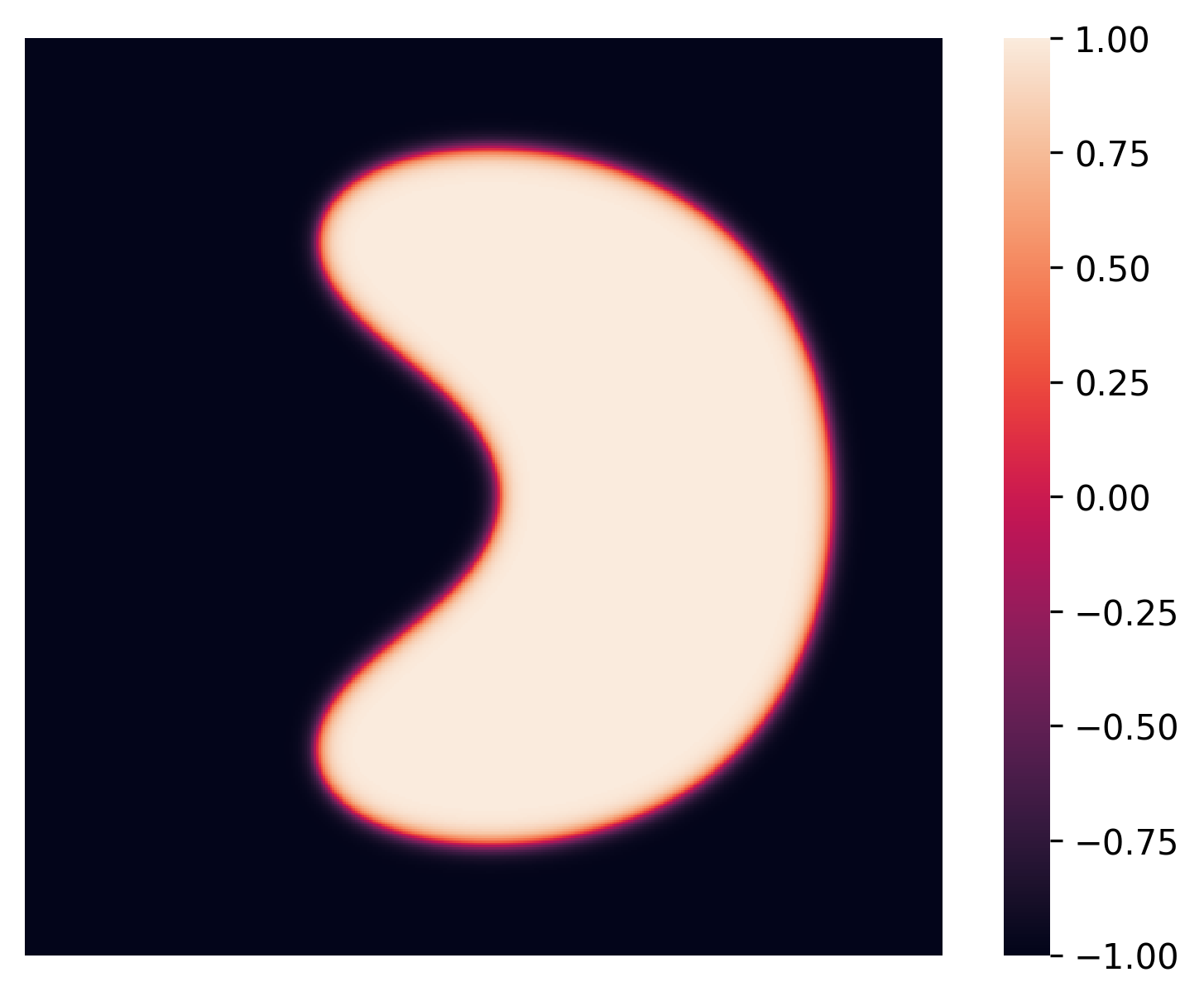}
    \includegraphics[width=0.24\linewidth]{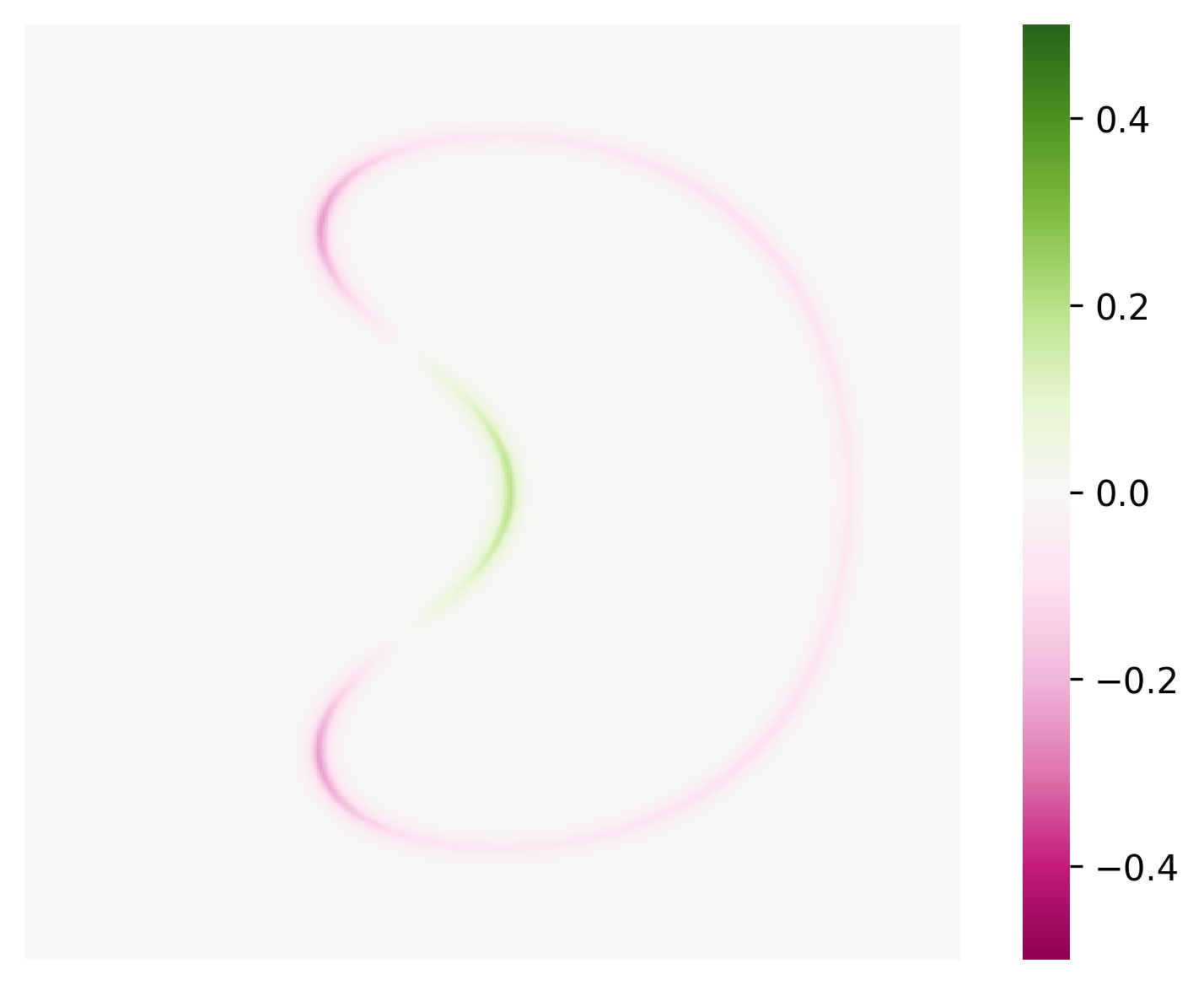}
    \includegraphics[width=0.24\linewidth]{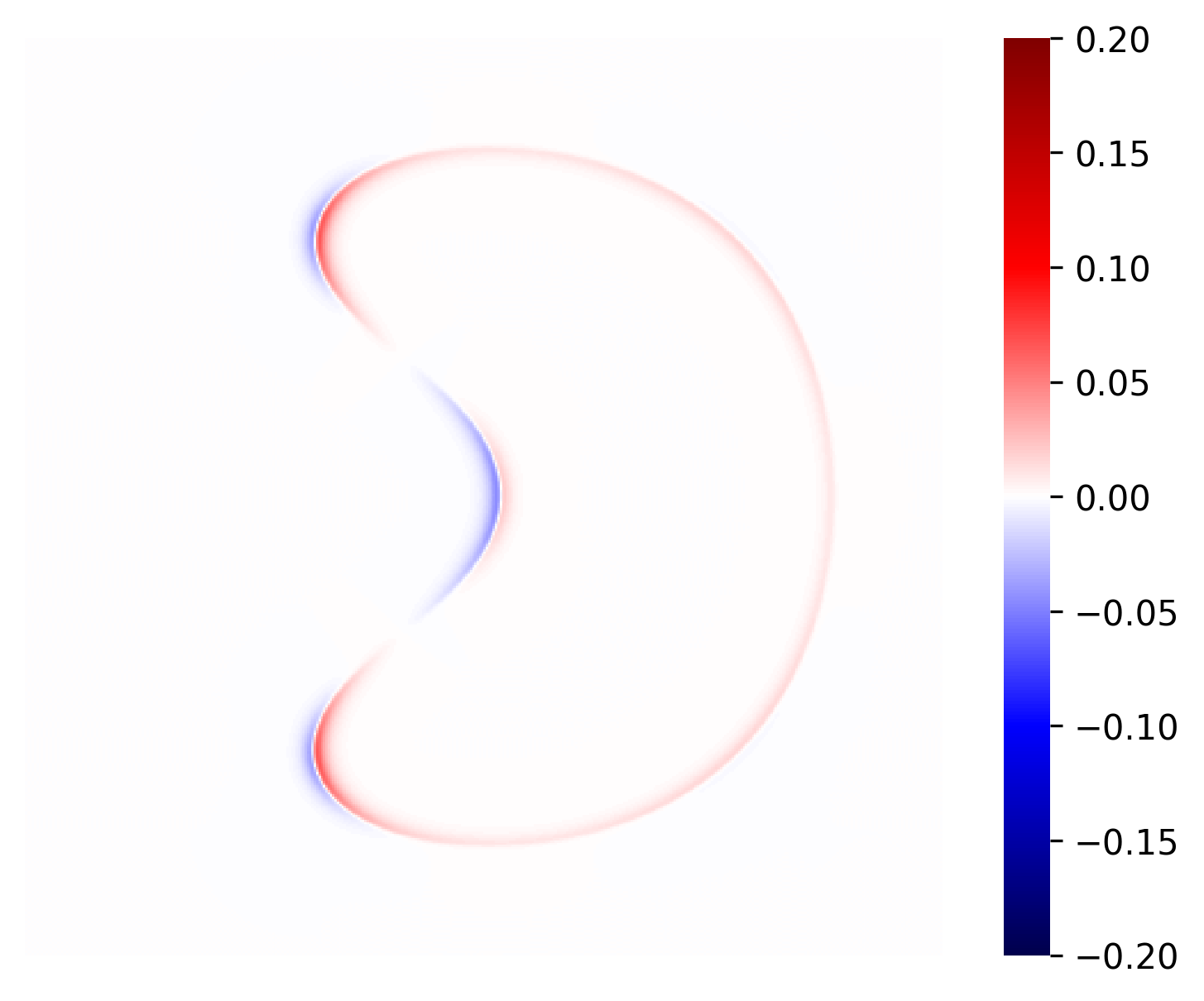}
    \includegraphics[width=0.24\linewidth]{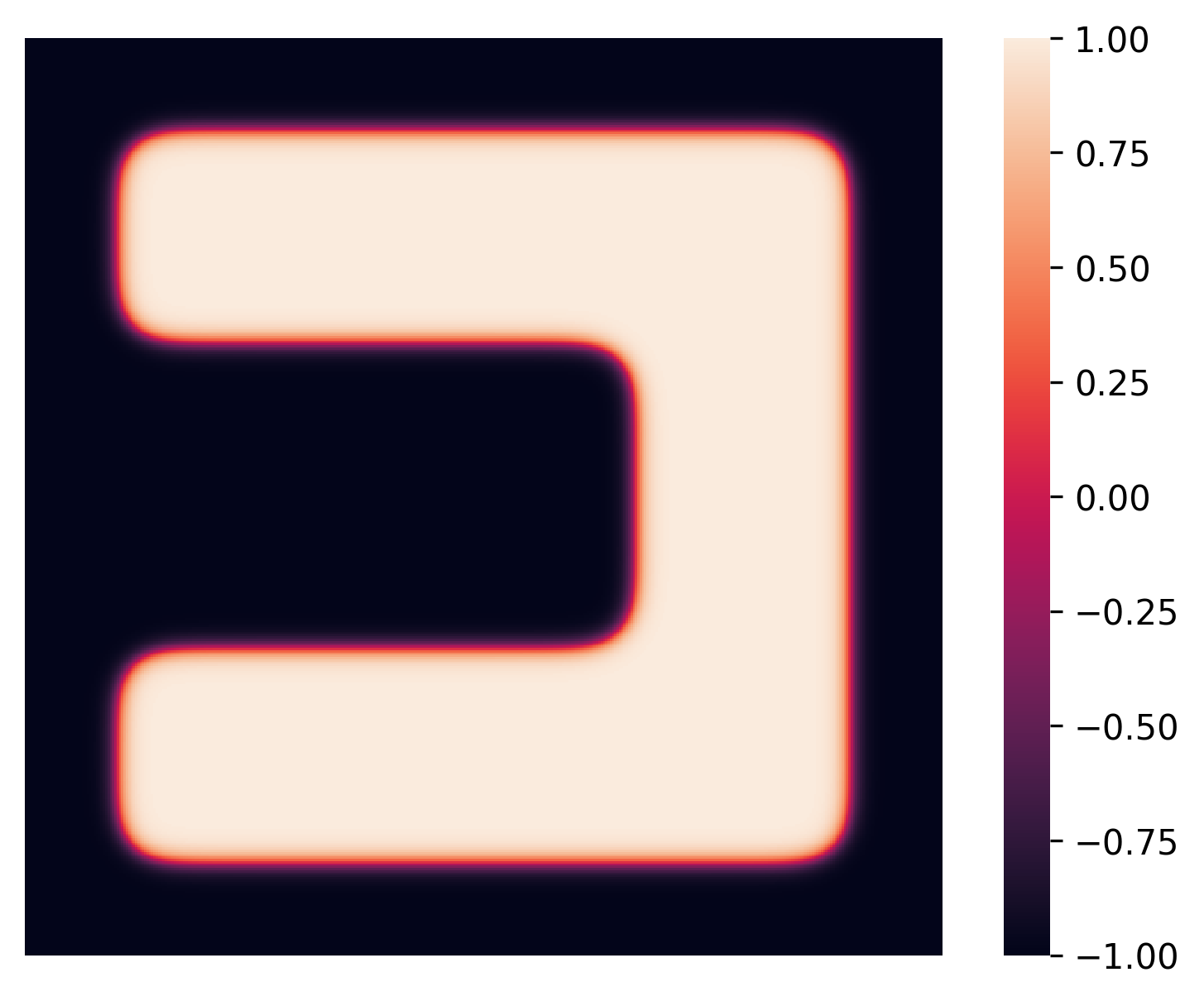}

    \includegraphics[width=0.24\linewidth]{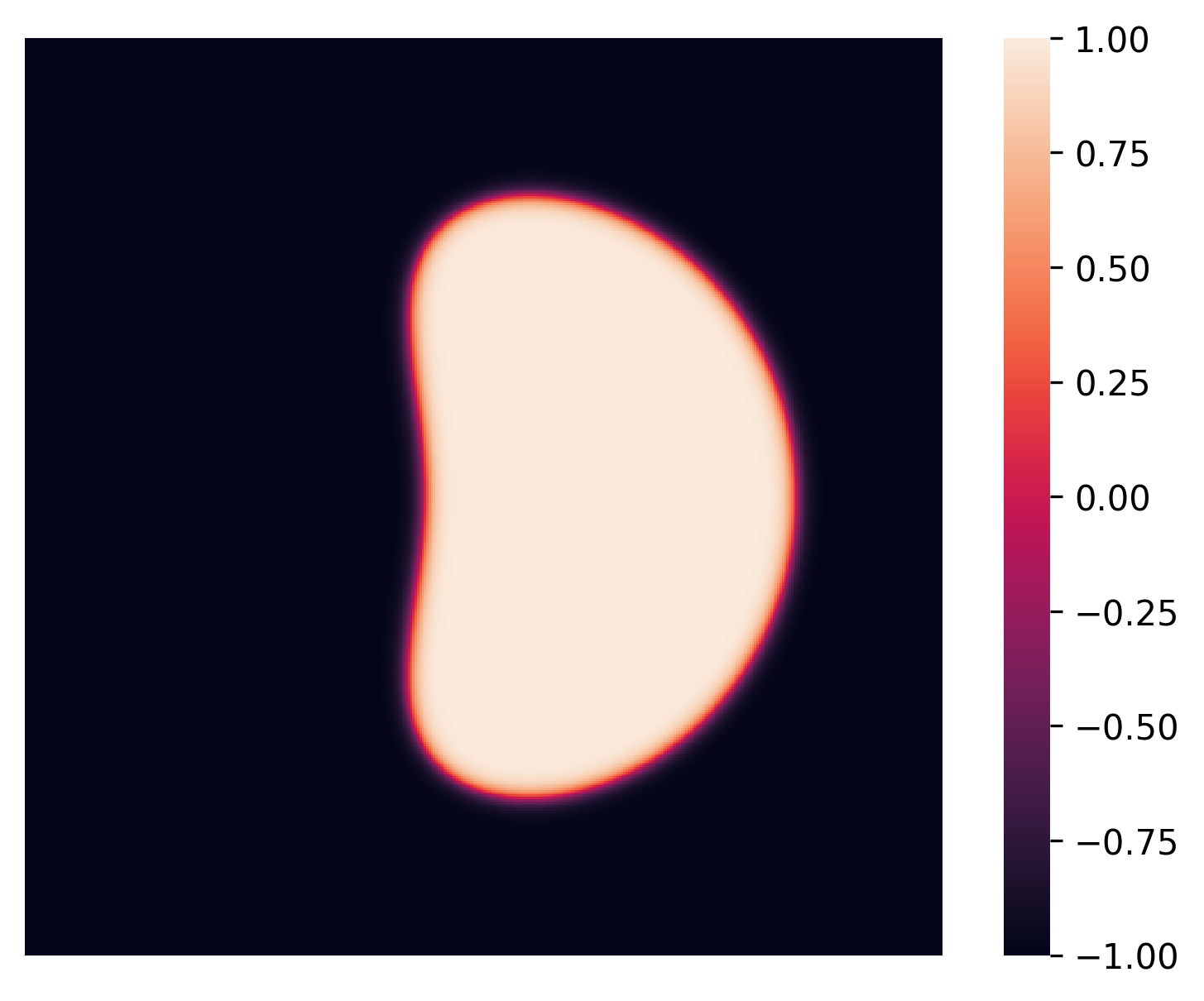}
    \includegraphics[width=0.24\linewidth]{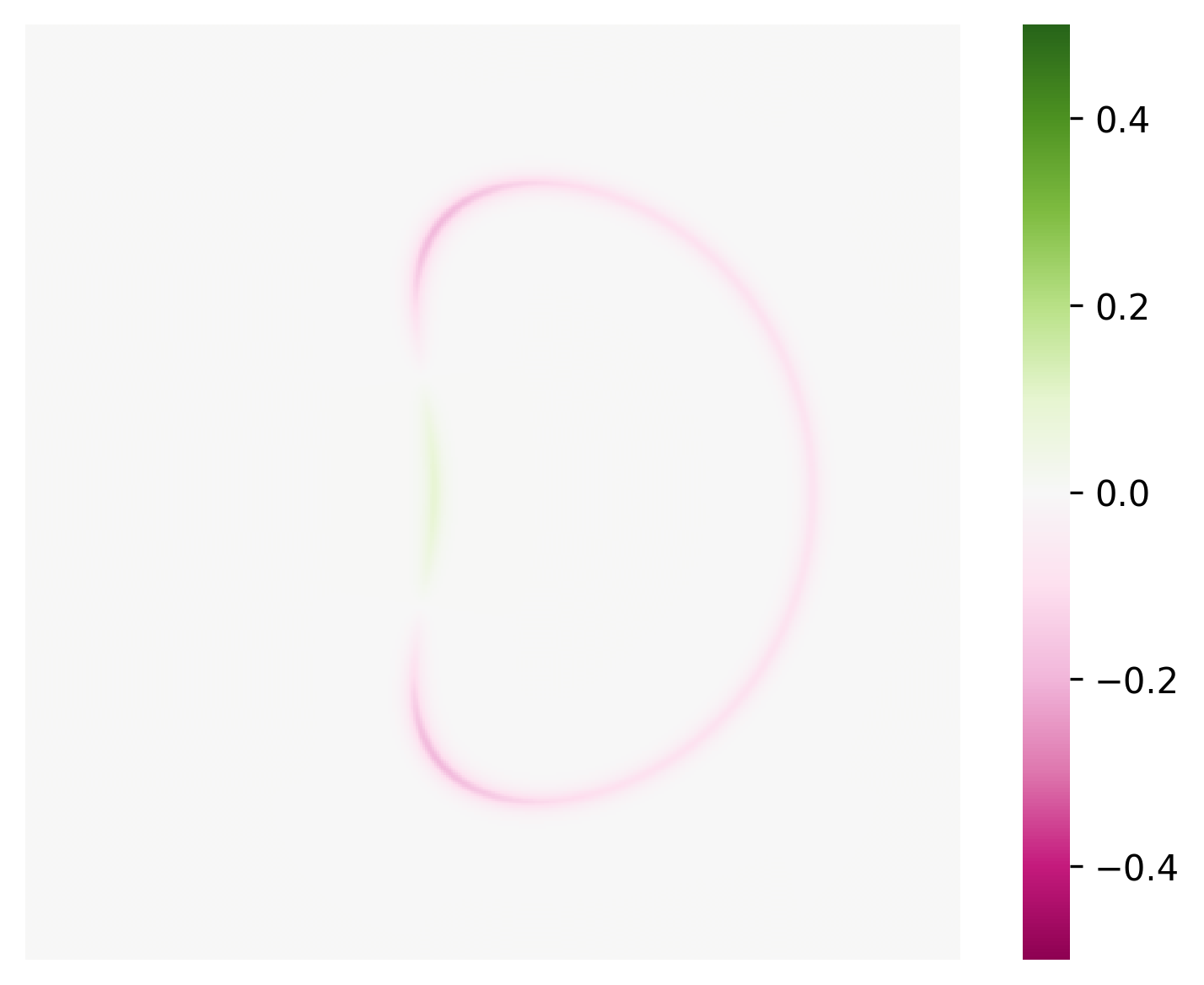}
    \includegraphics[width=0.24\linewidth]{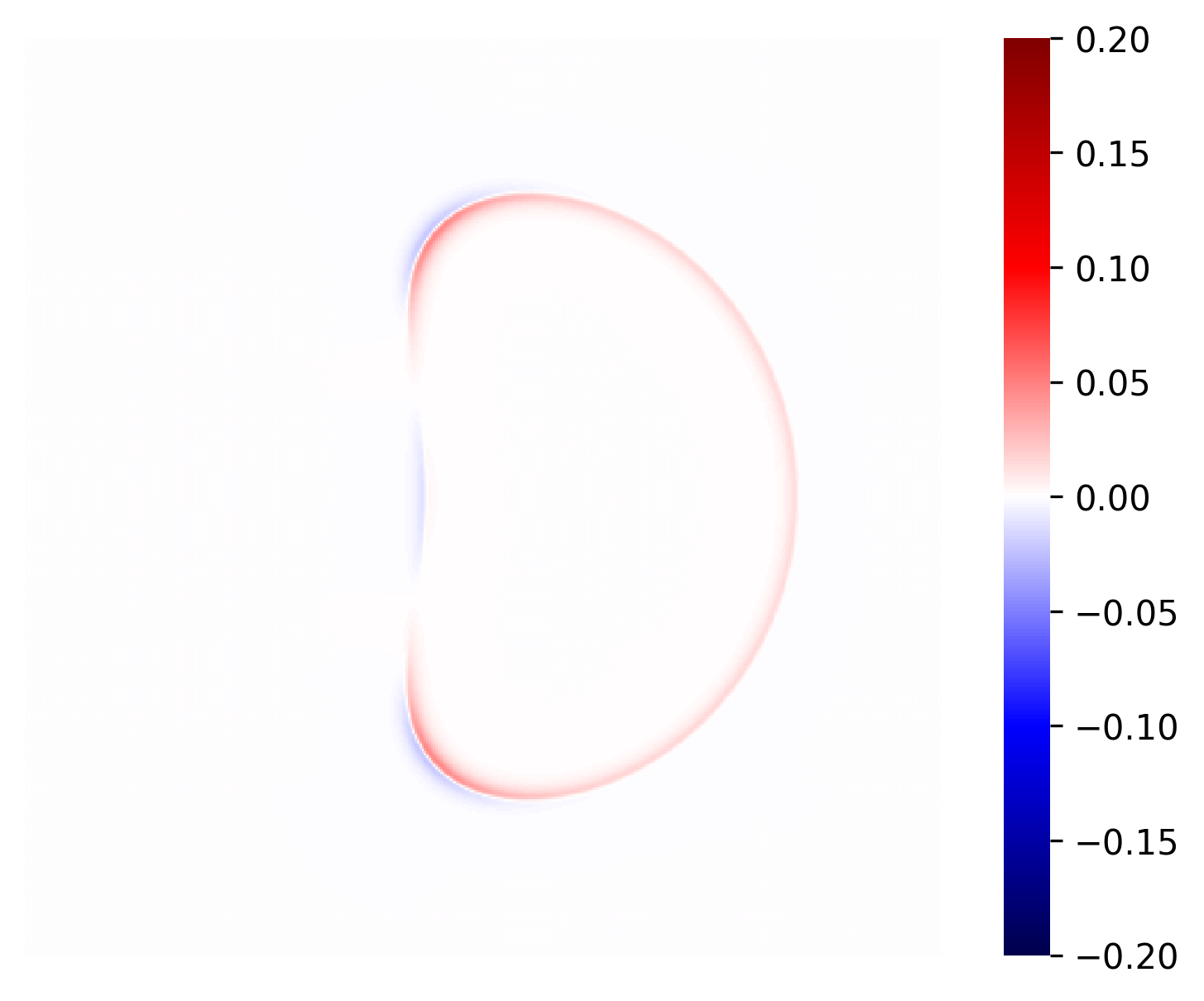}
    \includegraphics[width=0.24\linewidth]{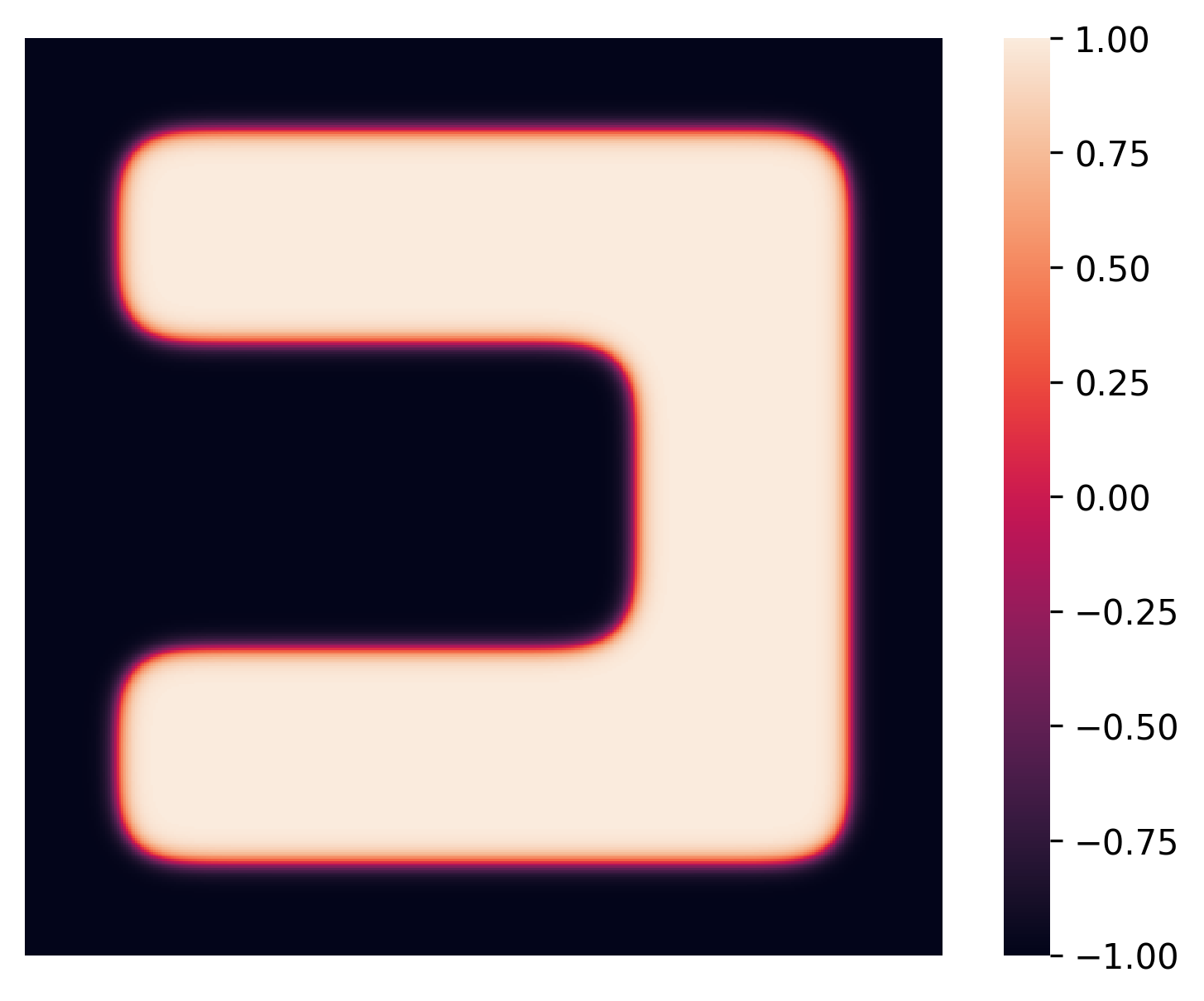}

    \includegraphics[width=0.24\linewidth]{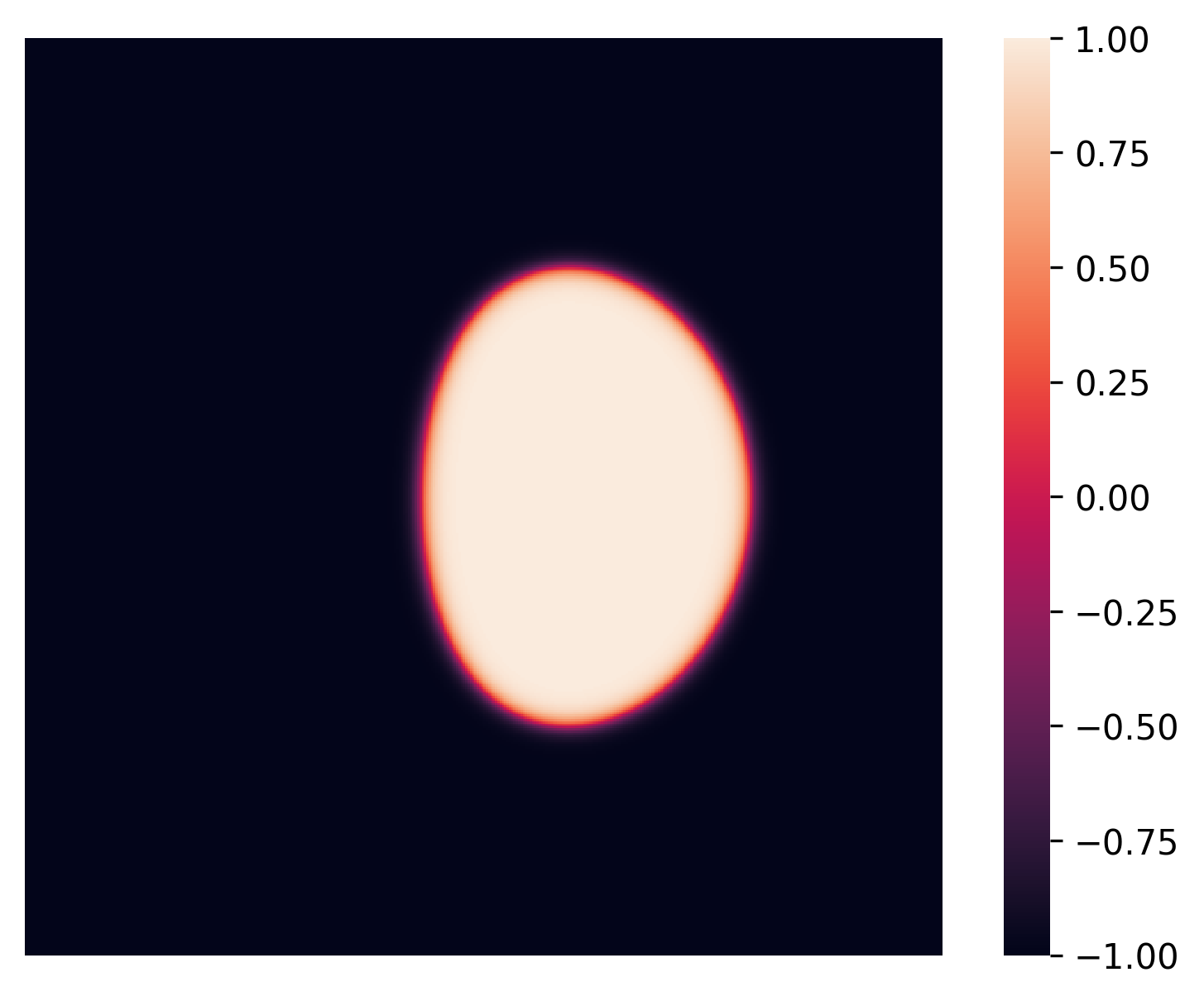}
    \includegraphics[width=0.24\linewidth]{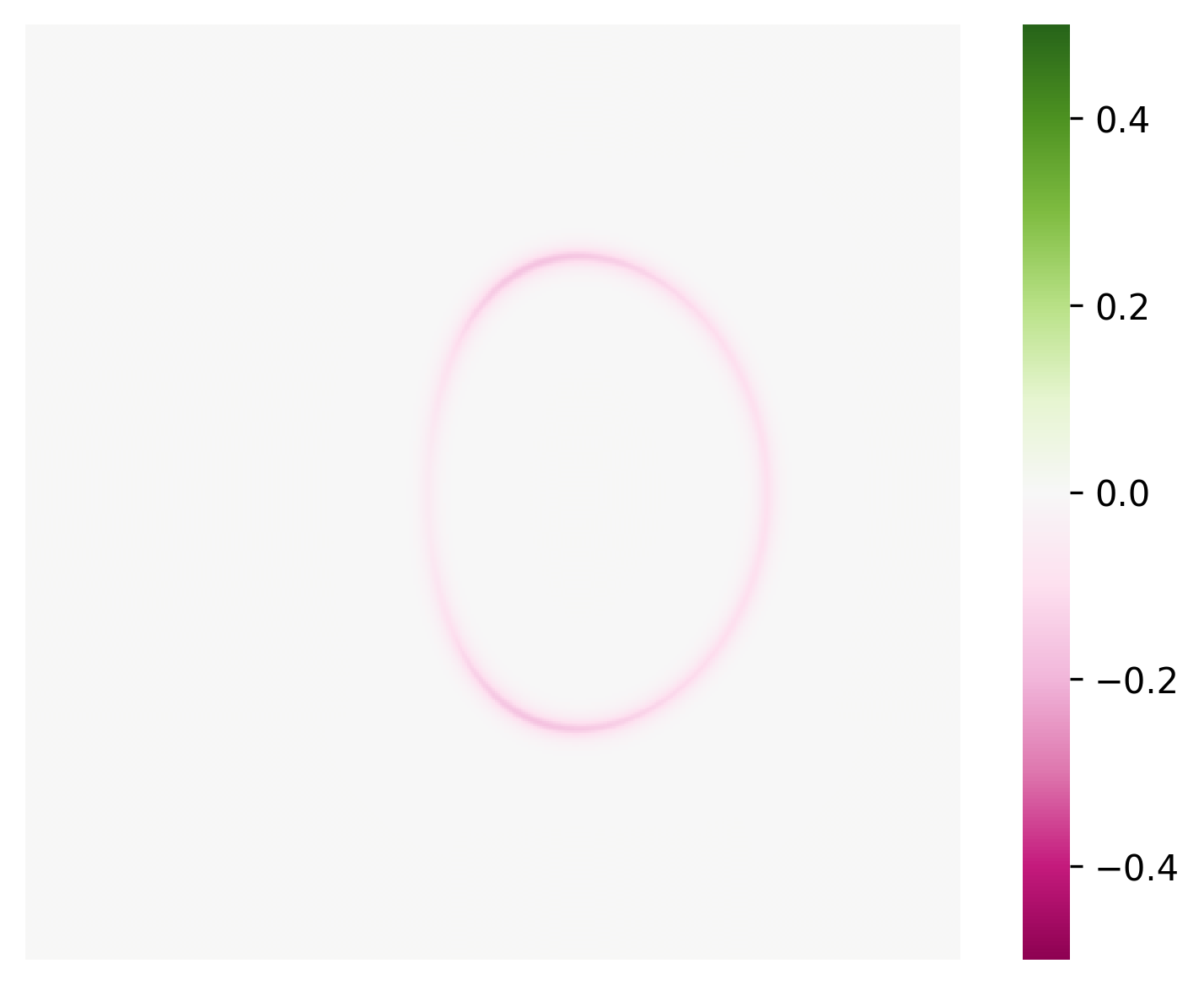}
    \includegraphics[width=0.24\linewidth]{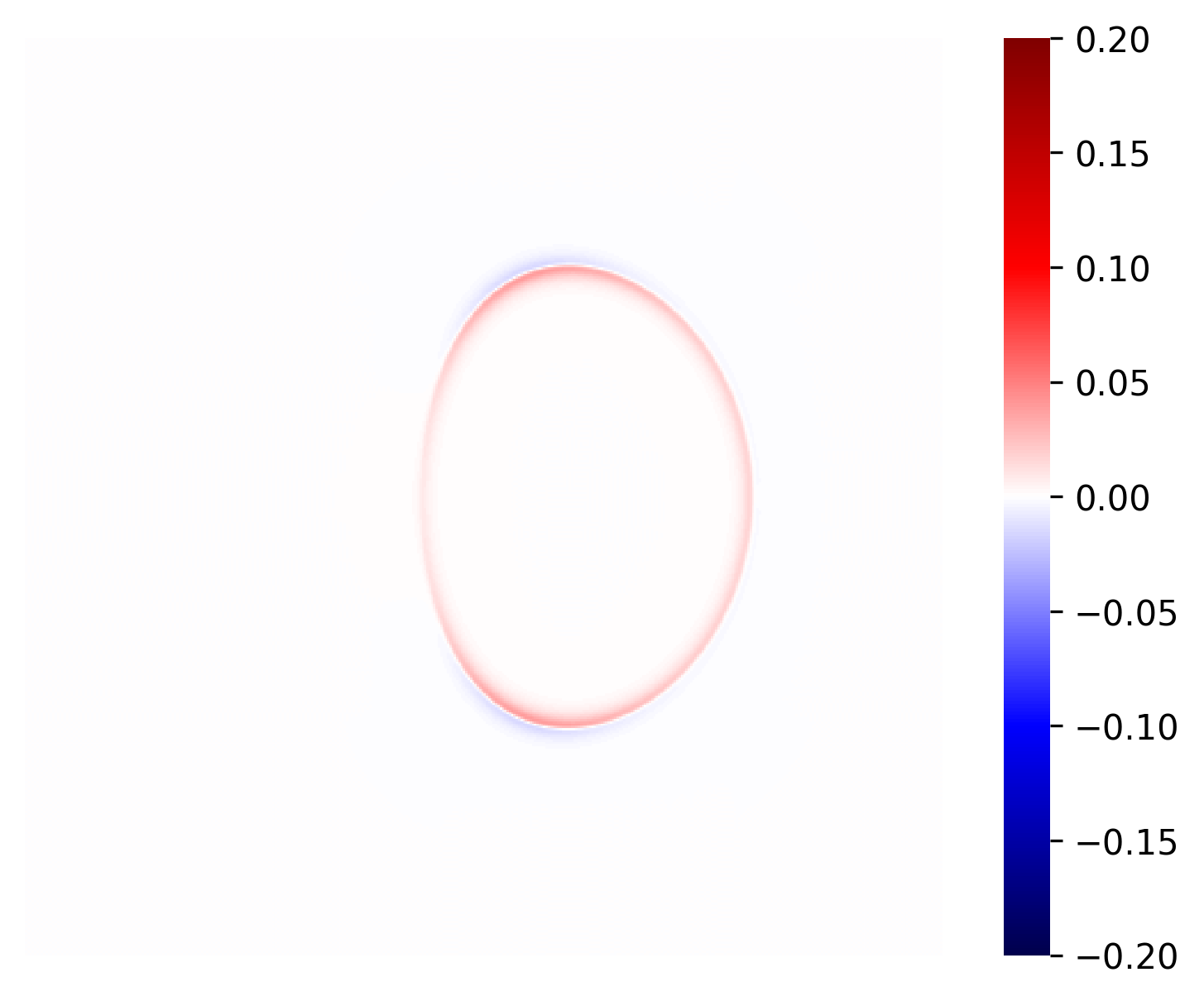}
    \includegraphics[width=0.24\linewidth]{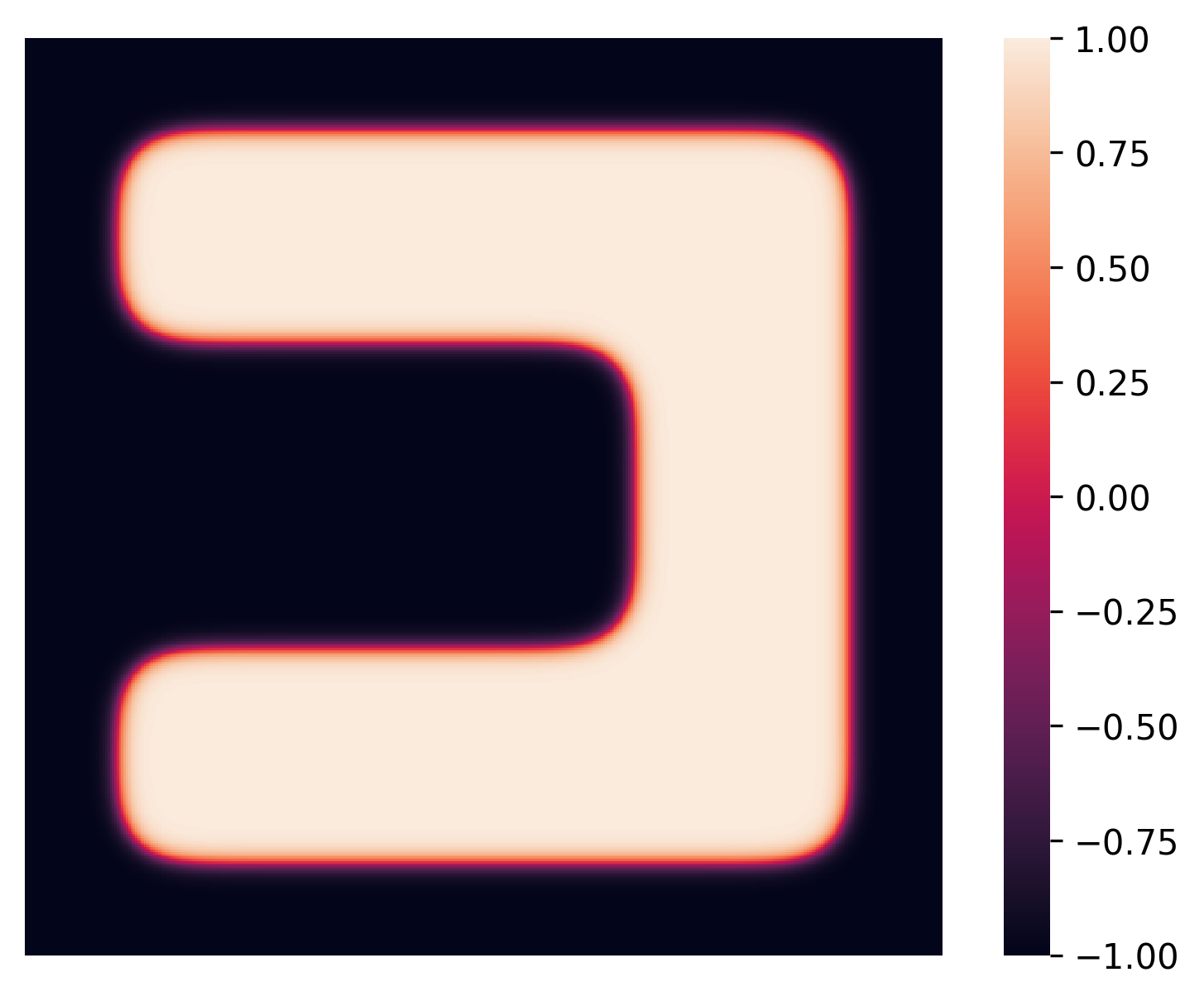}

    \includegraphics[width=0.24\linewidth]{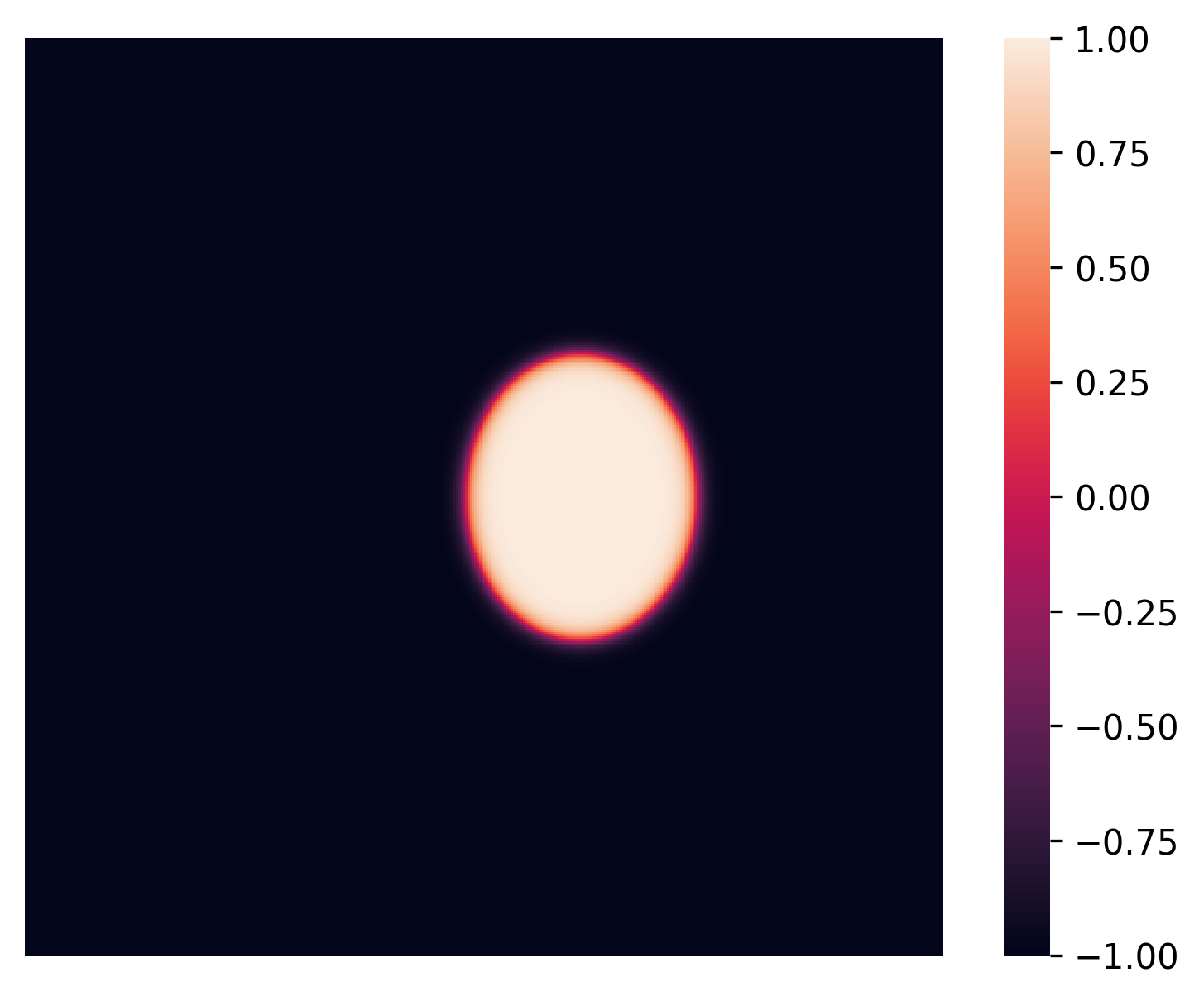}
    \includegraphics[width=0.24\linewidth]{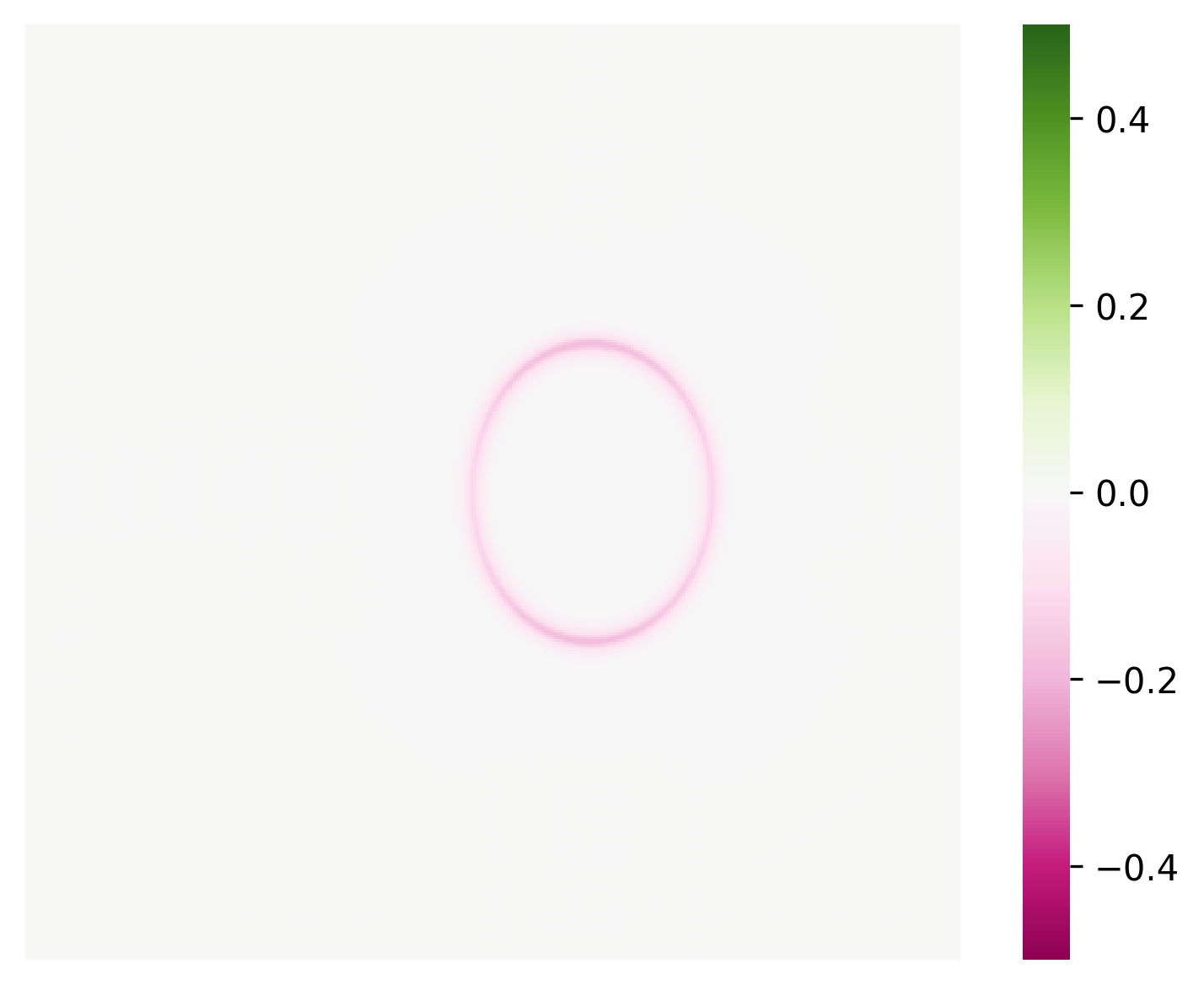}
    \includegraphics[width=0.24\linewidth]{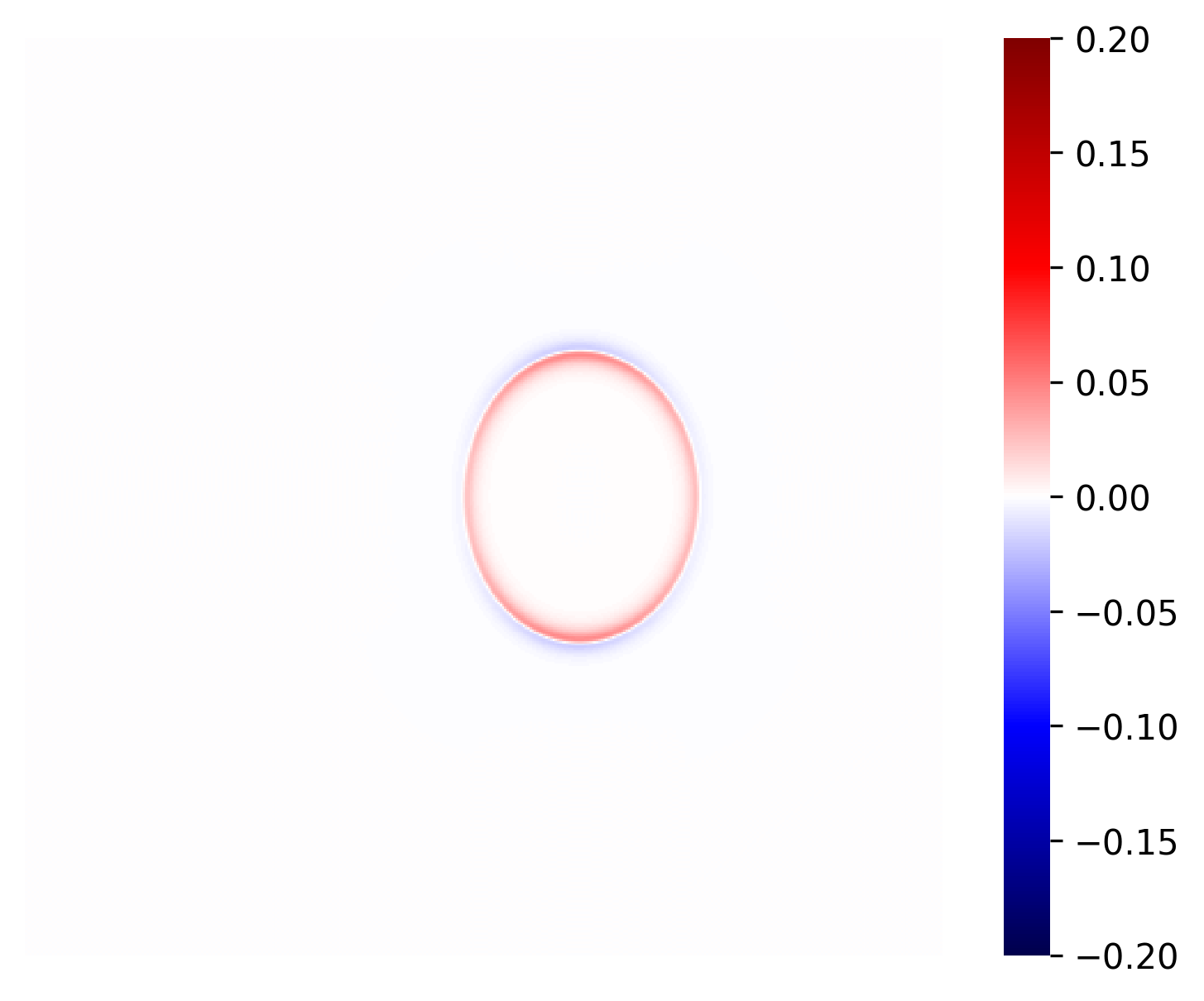}
    \includegraphics[width = .24\linewidth]{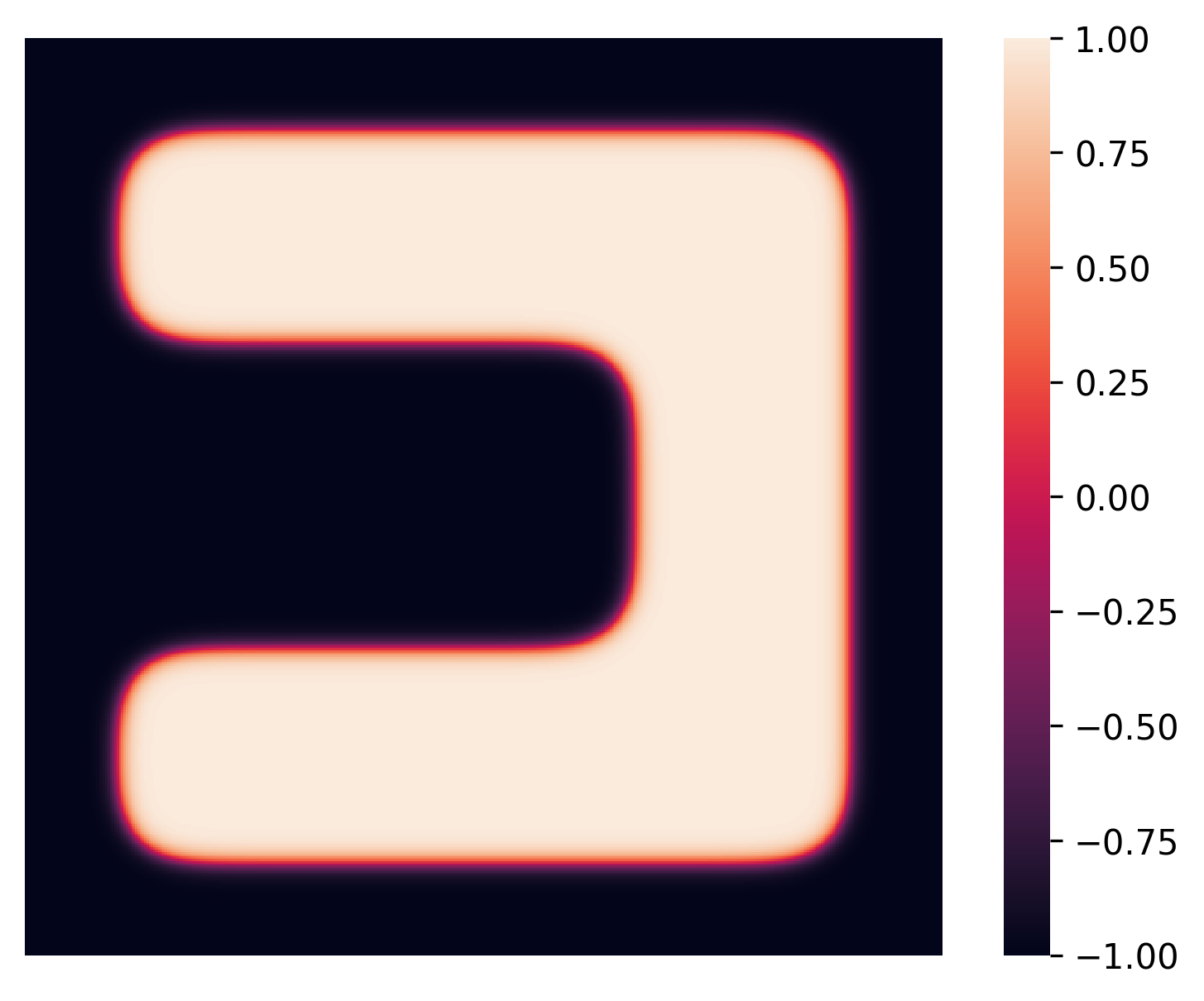}

    \includegraphics[width=0.24\linewidth]{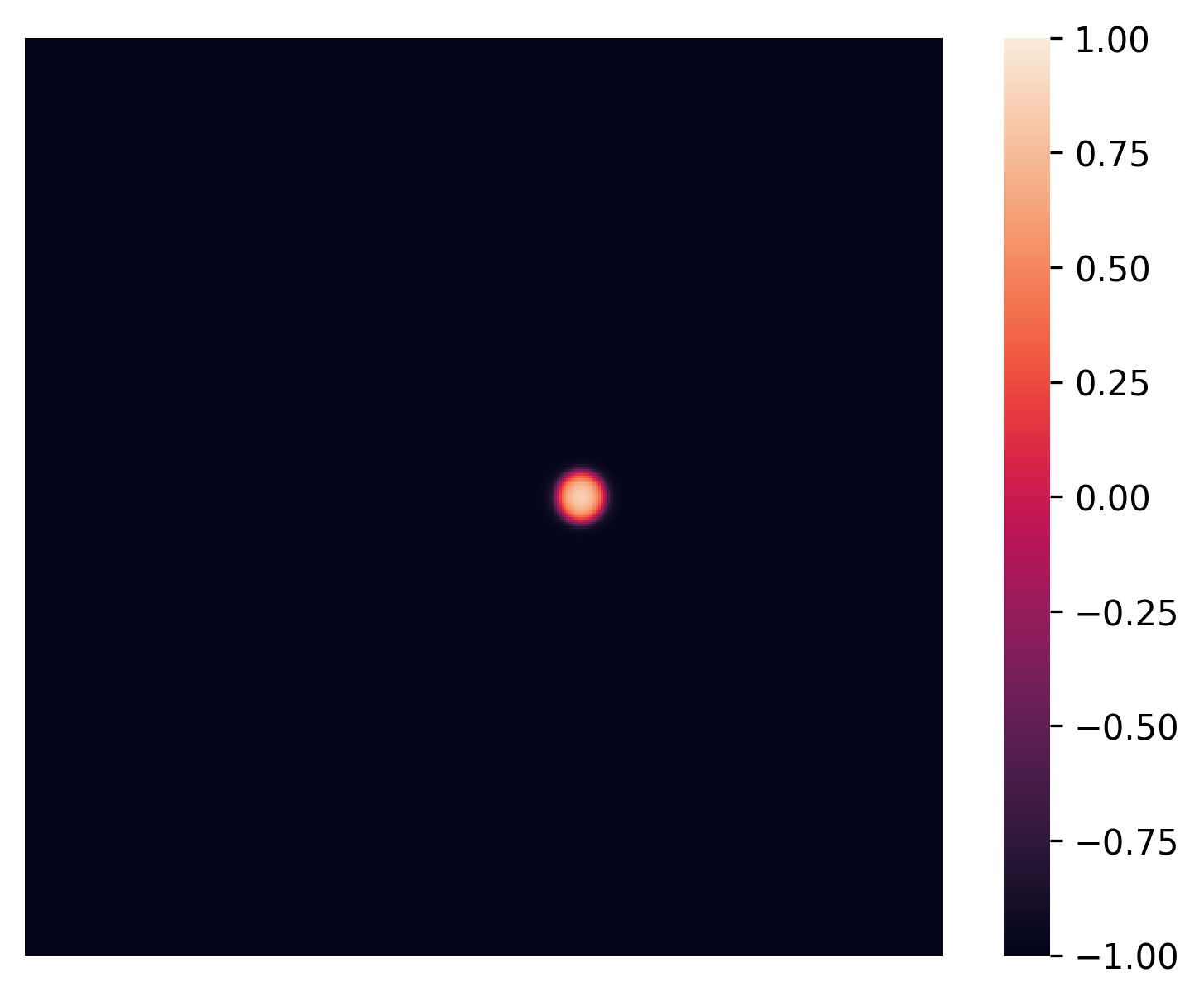}
    \includegraphics[width=0.24\linewidth]{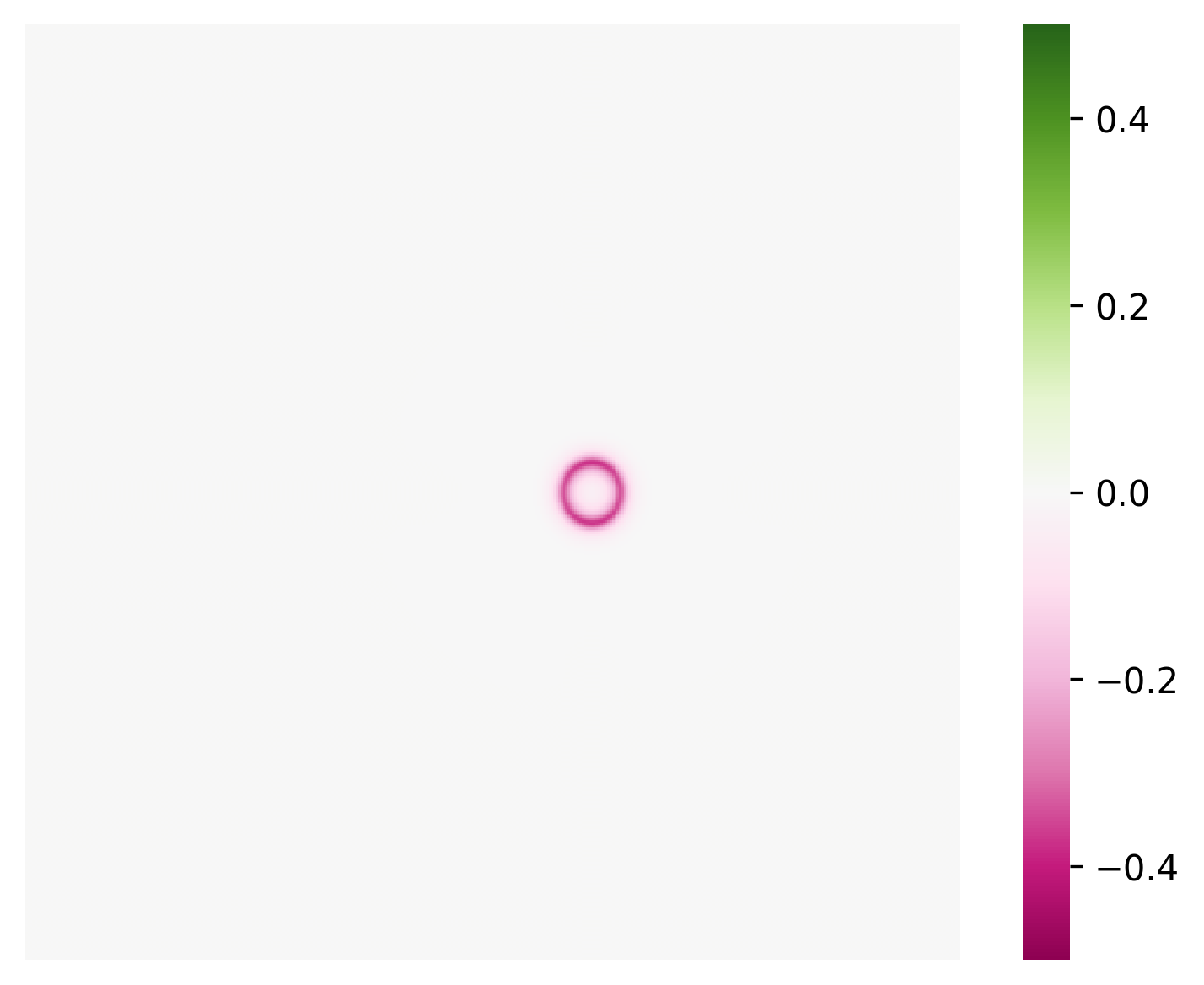}
    \includegraphics[width=0.24\linewidth]{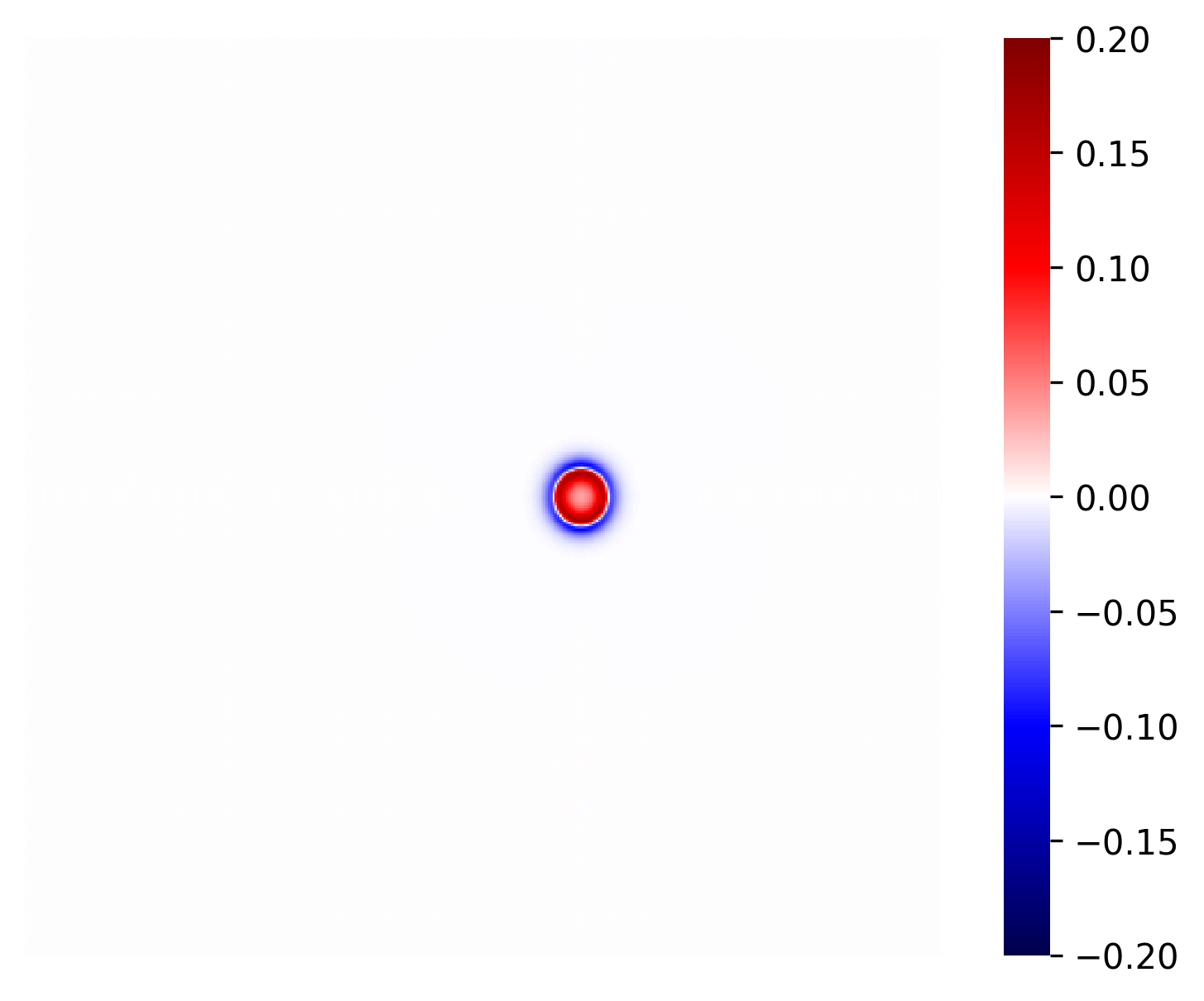}
    \includegraphics[width = .24\linewidth]{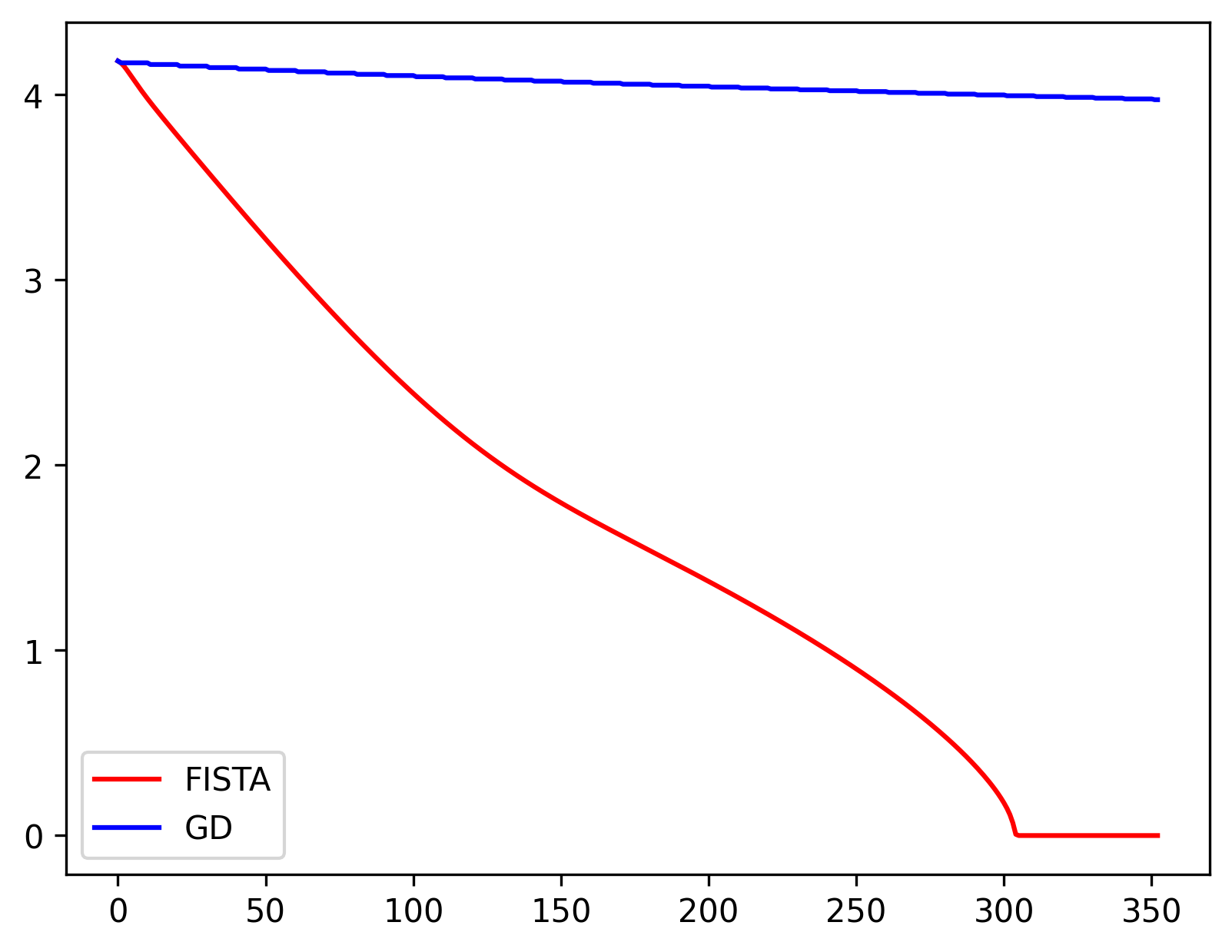}

    \caption{The FISTA-approximation to the accelerated Allen Cahn equation for $\alpha = 0.1$ with convex-concave splitting and large steps ($\tau=1$) after 50, 100, 150, 200, 250 and 300 steps (top to bottom). {Left column:} $u$, {second column:} $v= u_t$, third column: energy gradient density $\Delta u_{n+1}- \frac{2 u_{n+1} + W_{concave}'(u_n)}{\eps^2}$. For comparison, the solution $u$ for GD with step size $\tau = 1,000$ is plotted in the right column. The final image in the right column compares the energy decrease for FISTA and GD. The difference is more pronounced here than for larger $\eps = 0.03$ in Figure \ref{figure cinema vs fista pde}.\\ The energy gradient density (unsigned, sharpening the interface) accumulates in $v = u_t$ to a locally signed quantity driving the interface. Compared to the PDE, the large time-step scheme with momentum does not lead to singular interfaces with corners.}
    \label{figure large time step}
\end{figure}

\subsubsection{Surfaces in Three Dimensions}\label{section 3d}

To compare the performance of the classical convex-concave splitting gradient flow with our convex-concave splitting approach to FISTA, we consider an example with non-trivial stationary points. Since CINEMA does not seem to offer an additional speedup, we exclude it from this comparison. Thus we compute phase field approximations of triply-periodic minimal surfaces, in particular, we choose the well-known Schwarz P surface and Schoen's Gyroid. These minimal surfaces are stable under volume-preserving perturbations \cite{ROSS1992179}, they minimize area under such deformations \cite{MR1432843}. As they are not stable under general deformations, we impose volume preservation by adding a side condition in the gradient flow step (both in the classical gradient flow as well as in the gradient flow step of FISTA) so that the average of the phase field variable $u$ must vanish -- see also Appendix \ref{appendix volume-preserving}. As both Schwarz P and Gyroid split the period cell into equal volumes, this is will simply yield a volume preserving diffuse area flow.

For our 3d-implementation we use the ``Fourier Operators'' framework by L.\ Striet (Freiburg). Here, the equal volume fraction condition can be obtained by simply ensuring that the zeroth Fourier coefficient vanishes at each gradient flow step. All computations were performed on a regular grid of $N^3=256^3$ nodes covering the periodic unit cube and $\varepsilon=0.03 \approx \frac{7.5}{N}$. For the FISTA algorithm, we always first compute one regular gradient flow step with step-size $\tau^2$ to obtain a state with vanishing average.

Table \ref{tab:steps} shows a convergence analysis. For the initial condition being the nodal interpolant of 
\begin{equation} \label{eq:initSP}
u_0 = -1 + 2\cdot1_{\{(x-0.5)^2<0.3^2\}\cup\{(y-0.5)^2<0.3^2\}\cup\{(z-0.5)^2<0.3^2\}},
\end{equation}
i.e., three cylinders coarsely approximating the inside of a Schwarz P surface, we first compute a reference solution $u_\mathrm{ref}$ by executing a gradient flow with step size $\tau=0.1$ until the discrete $L^2$-distance between iterations $j$ and $k$ satisfies $||u_k-u_j||_{L^2} \le \delta_\mathrm{ref}$. We then compute the FISTA flow and classical gradient flow, starting from the same initial condition, until $||u^\mathrm{FISTA}_k-u_\mathrm{ref}||_{L^2} < \delta$ and $u^\mathrm{gf}_k < \delta$, respectively. For $\delta_\mathrm{ref} = 1.0\cdot 10^{-12}$, $\delta = 1.0\cdot 10^{-11}$, $\tau=0.4$, $\alpha=1.4$ in the FISTA algorithm, convergence was achieved after 93 steps (including the aforementioned initial step). We note that for these parameters, the $L^2$-difference to the reference solution does not increase again afterwards. Noting the values in Table \ref{tab:steps}, we see that FISTA provides a speedup by a factor larger than 5 compared to the classical convex-concave splitting gradient flow with the same stopping criterion -- independently of the step size used for the classical algorithm. We also remark that -- as expected from the analysis in \cite{related_splitting_paper} -- the convex concave splitting algorithm does not yield convergence to the minimizer faster when the step size is increased beyond a certain threshold.

In Figures \ref{fig:schp} and \ref{fig:gyr}, we illustrate the convergence of the FISTA flow to the Schwarz P surface (for the initial condition in equation \eqref{eq:initSP}) and to the Gyroid (for the initial condition $u_0 = -1 + 2\cdot 1_{\{\cos(2\pi(x-0.5)\sin(2\pi(y-0.5) + \cos(2\pi(y-0.5)\sin(2\pi(z-0.5) + \cos(2\pi(z-0.5)\sin(x-0.5) < 0\}}$), respectively. Noting that the standard approximation  of the Gryoid used as initial condition here is nearly indistinguishable from the ``real'' minimal surface, the FISTA algorithm essentially only introduces the optimal transition layer starting from a step function. The FISTA parameters in both these experiments were again $\tau=0.4$ and $\alpha=1.4$.

\begin{table}
\begin{center}
\begin{tabular}{r|c|c|c|c|c|c}
$\tau$ & $1.0 \cdot 10^{-4}$ & $2.0 \cdot 10^{-4}$ & $4.0 \cdot 10^{-4}$ & $8.0 \cdot 10^{-4}$ & $1.6 \cdot 10^{-3}$ & $3.2 \cdot 10^{-3}$ \\ 
\hline
Steps  & 2151 & 1310 & 891 & 681 & 576 & 524   \\ 
\arrayrulecolor{white}\hline\hline\hline\hline\hline \arrayrulecolor{black}
$\tau$ & $6.4 \cdot 10^{-3}$ & $1.28 \cdot 10^{-2}$ & $2.56 \cdot 10^{-2}$ & $5.12 \cdot 10^{-2}$ & $1.024 \cdot 10^{-1}$ & $2.048 \cdot 10^{-1}$ \\
\hline
Steps &   497 & 484 & 478 & 474 & 473 & 472 \\
\arrayrulecolor{white}\hline\hline\hline\hline\hline \arrayrulecolor{black}
$\tau$ & $4.096 \cdot 10^{-1}$ & $8.192 \cdot 10^{-1}$ & 1.6384 & 3.2768 & 6.5536 & 13.1072 \\ \hline
Steps &   472 & 471 & 471 & 471 & 471 & 471 \\
\arrayrulecolor{white}\hline\hline\hline\hline\hline \arrayrulecolor{black}
$\tau$ & $1.0\cdot 10^2$ & $1.0\cdot 10^3$ & $1.0\cdot 10^4$ & $1.0\cdot 10^5$  \\ \cline{1-5}
Steps &   471 & 471 & 471 & 471  \\
\end{tabular}
\vspace{4mm}

\caption{Number of time steps until convergence to the phase field approximation of a Schwarz P surface has been reached versus the step size for the classical gradient flow algorithm. For comparison, the FISTA algorithm requires only 93 steps for a step size $\tau=0.4$ and friction $\alpha=1.4$.} \label{tab:steps}
\end{center}
\end{table}

\begin{figure}
\begin{center}
\includegraphics[trim={4.5cm 4.5cm 4.5cm 4.5cm},clip,width=0.32\linewidth]{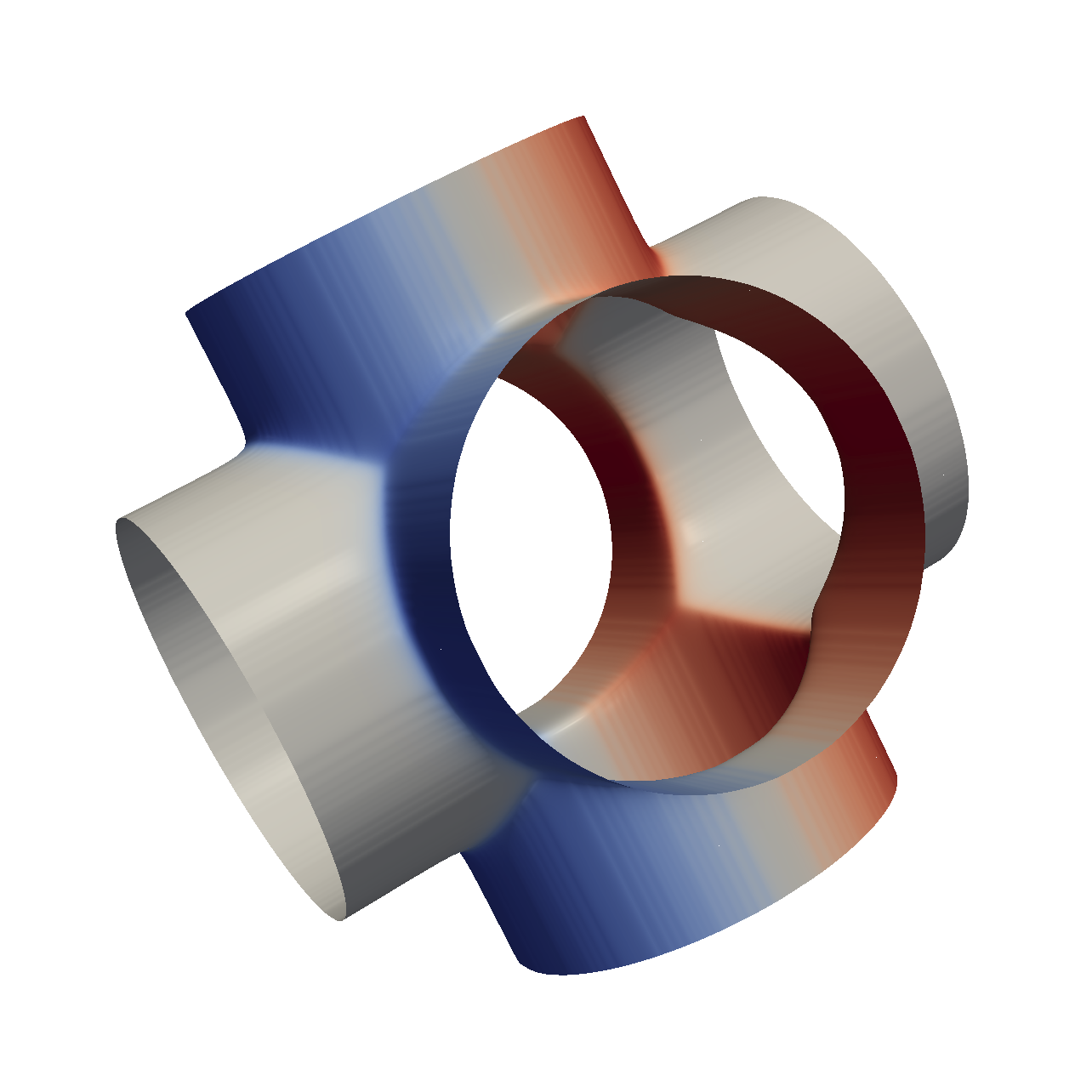}
\includegraphics[trim={4.5cm 4.5cm 4.5cm 4.5cm},clip,width=0.32\linewidth]{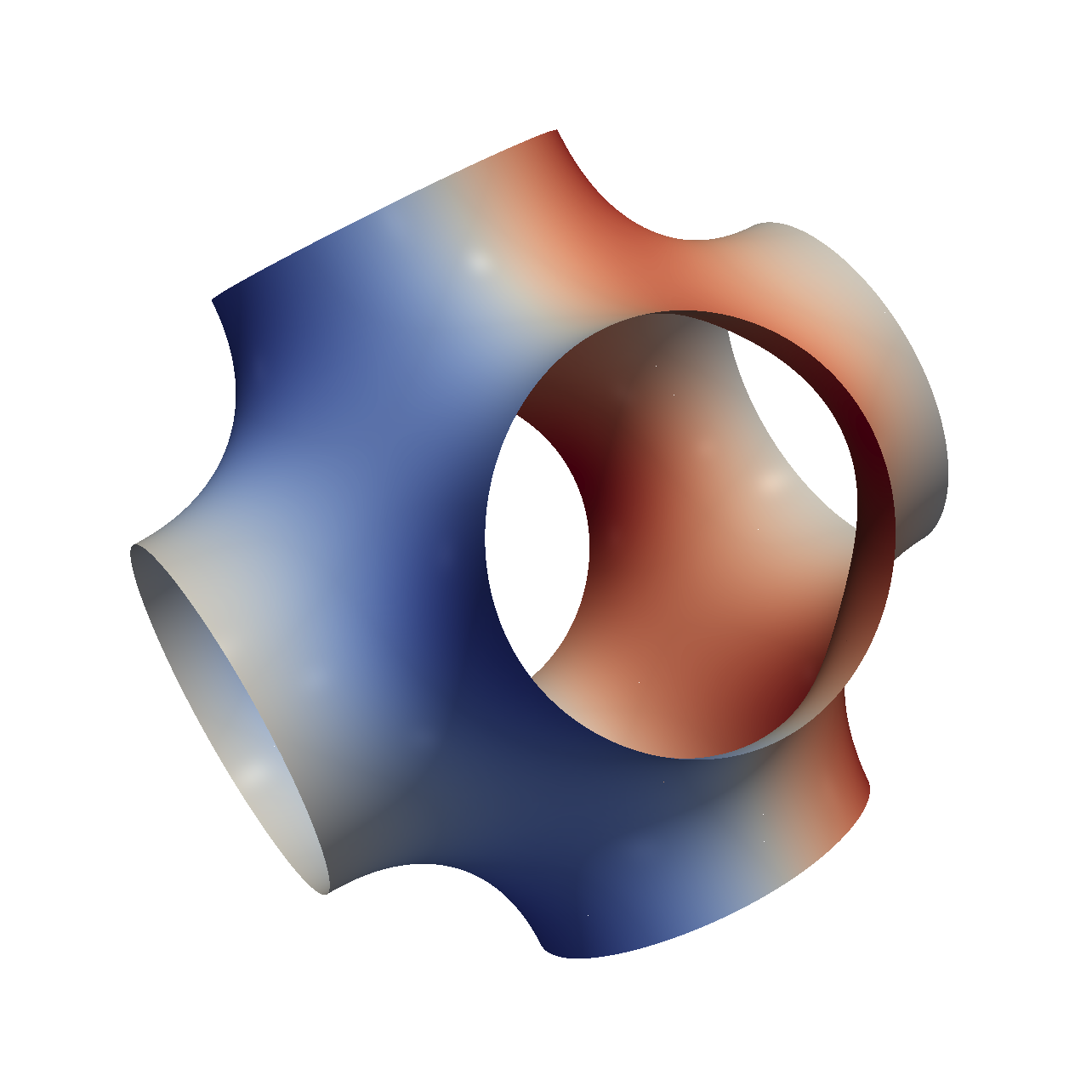}
\includegraphics[trim={4.5cm 4.5cm 4.5cm 4.5cm},clip,width=0.32\linewidth]{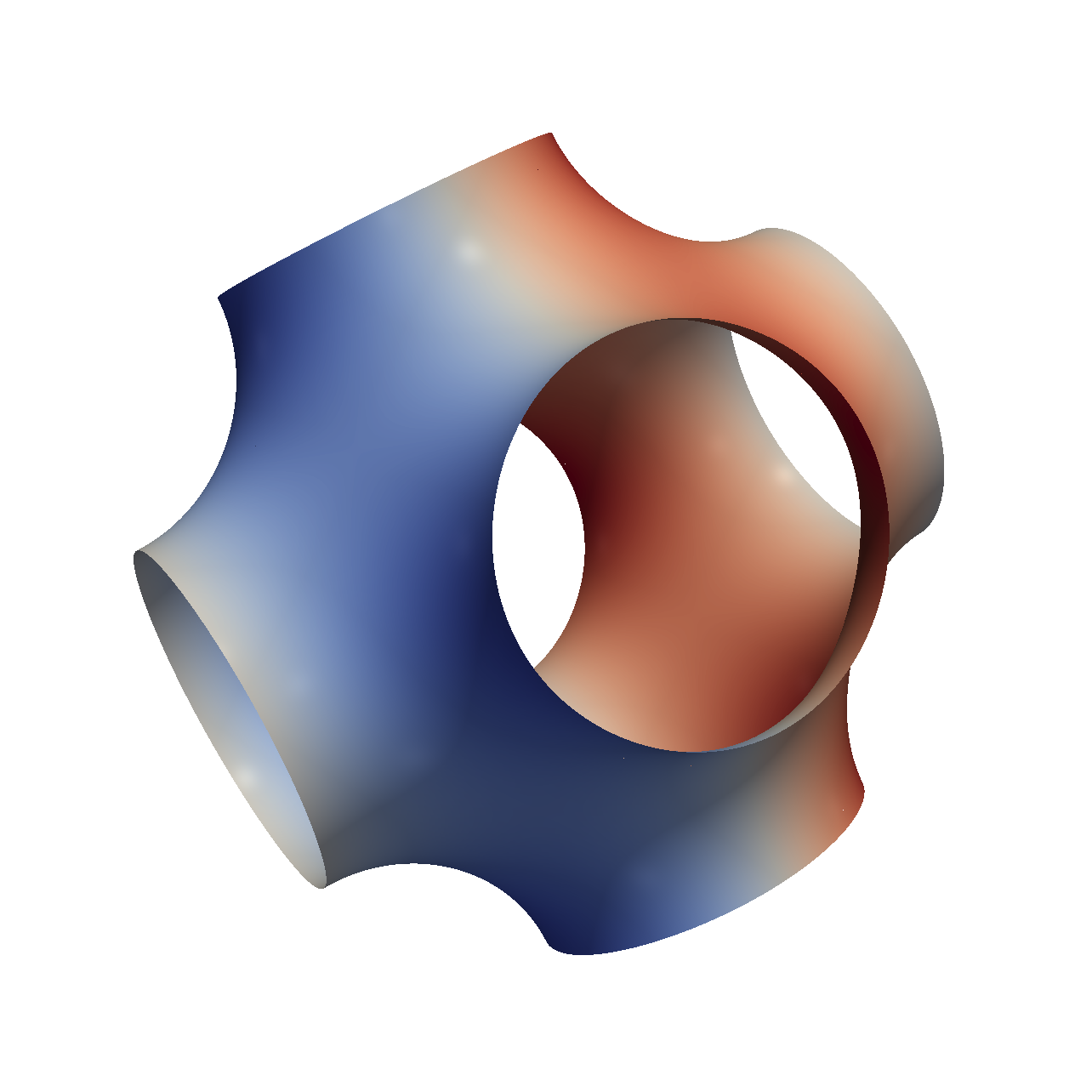}
\caption{Convergence of the FISTA flow to a Schwarz P minimal surface. Shown is the zero-level set of the phase field function, with color indicating the $x$-value of the surface normal. From left to right, time step 1 (immediately after the initial classical gradient flow step), 3, and 21 are shown. At time step 1, some artifacts from the discrete initialization are still visible. Between time step 4 and 21, there is no visually discernible difference. \label{fig:schp}}
\end{center}
\end{figure}
\begin{figure}
\begin{center}
\includegraphics[trim={4.5cm 4.5cm 4.5cm 4.5cm},clip,width=0.32\linewidth]{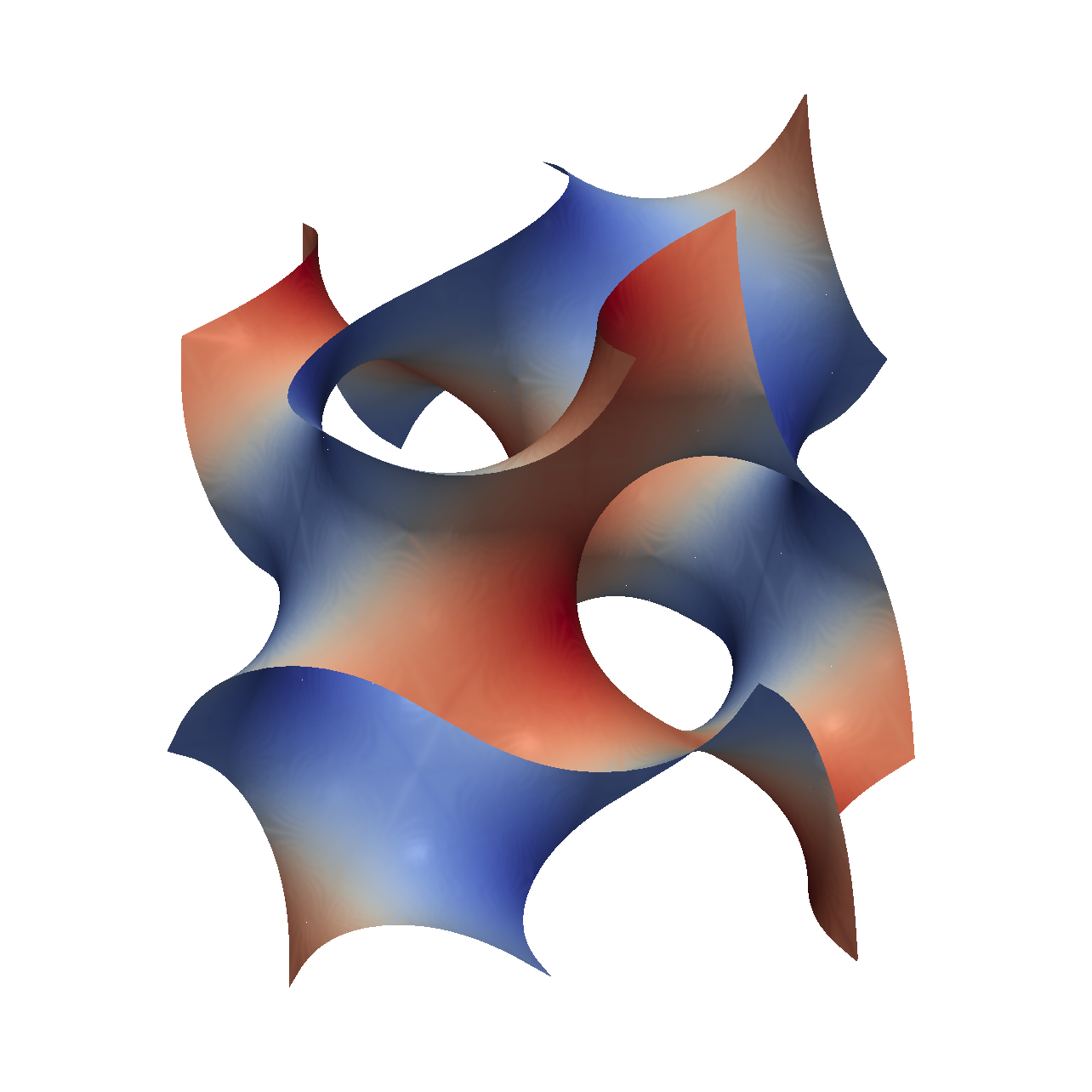}
\includegraphics[trim={4.5cm 4.5cm 4.5cm 4.5cm},clip,width=0.32\linewidth]{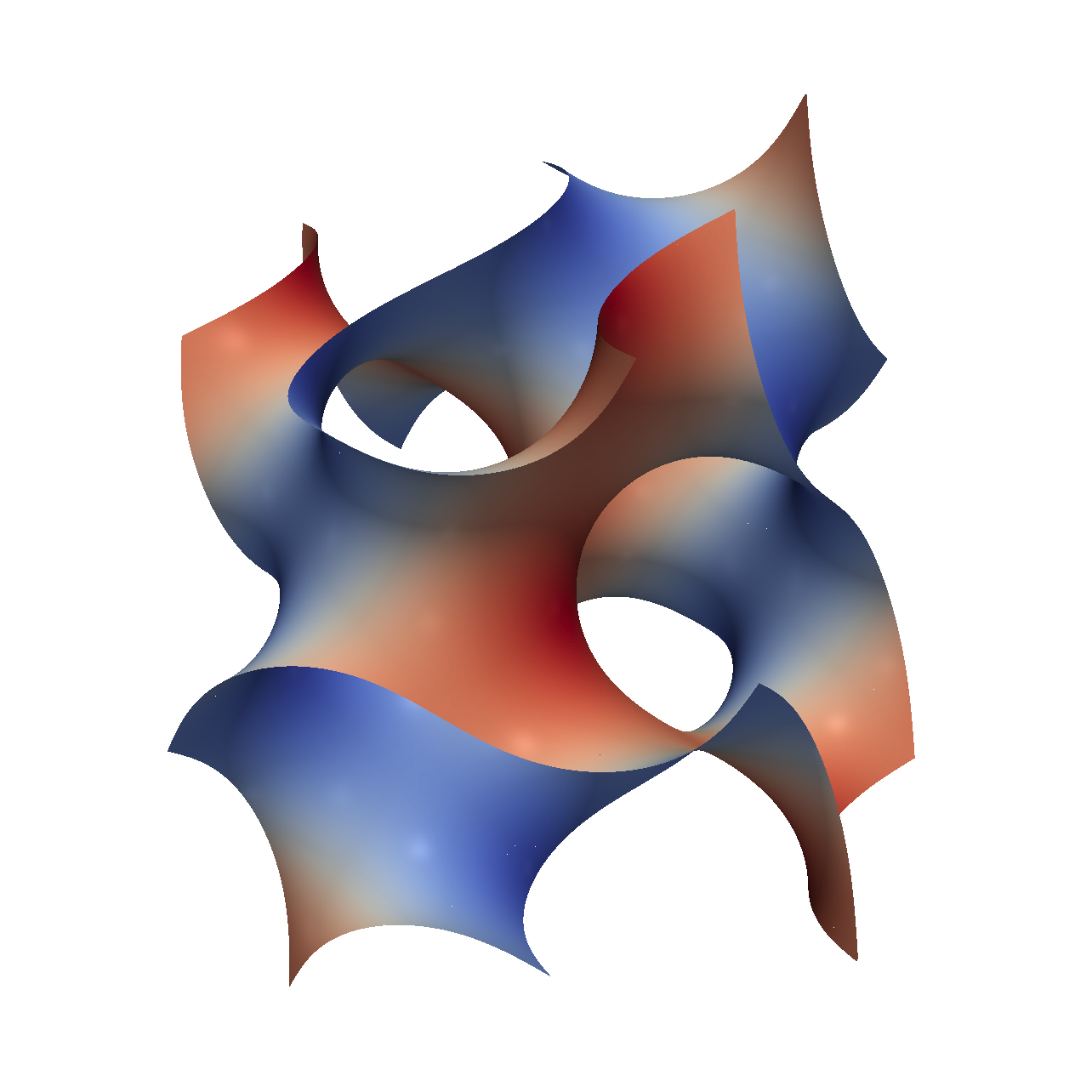}
\includegraphics[trim={4.5cm 4.5cm 4.5cm 4.5cm},clip,width=0.32\linewidth]{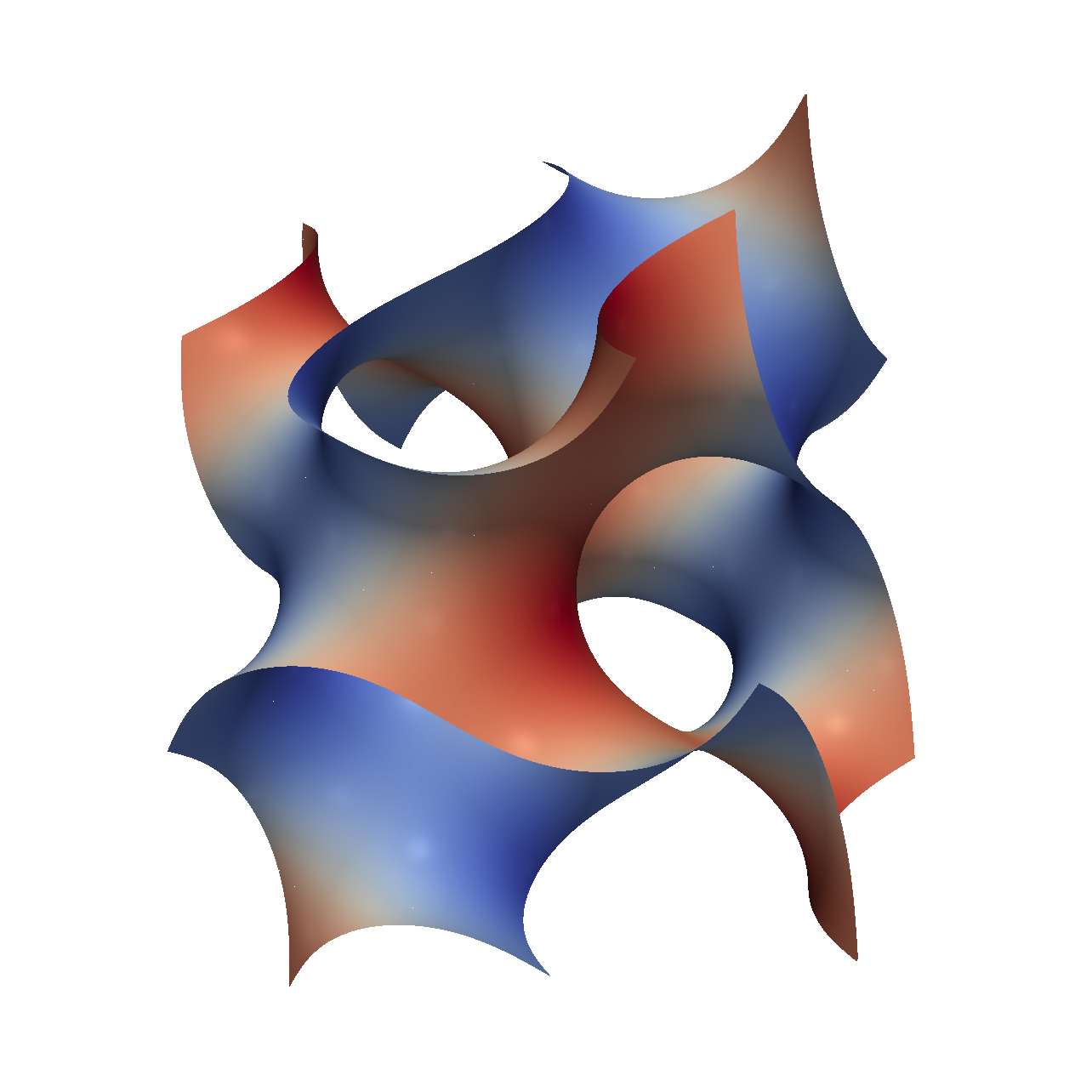}
\caption{Convergence of the FISTA flow to a Gyroid minimal surface. Shown is the zero-level set of the phase field function, with color indicating the $x$-value of the surface normal. From left to right, time step 1 (immediately after the initial classical gradient flow step), 3, and 21 are shown.  At time step 1, some artifacts from the discrete initialization are still visible. After step 2, there is no visually discernible difference. \label{fig:gyr}}
\end{center}
\end{figure}

\subsection{The Accelerated Allen-Cahn Equation on Graphs}\label{section graphs}

In this section, we present numerical experiments using the accelerated Allen-Cahn Equation on graphs with an eye towards applications in data science. We consider a synthetic two-dimensional dataset and the MNIST dataset. In both cases, we assign Gaussian edge weights $w_{ij} = \exp(-\|x_i-x_j\|^2/(2\sigma^2))$ for suitable $\sigma>0$ and use the unnormalized version of the graph Laplacian. As we consider multi-class classification, we require a vector-valued version of the Allen-Cahn and accelerated Allen-Cahn system. For details concerning the vector-valued implementation, see Appendix \ref{appendix vector-valued}.

\subsubsection{A synthetic dataset}
We generate $2,000$ data points in $\R^2$ belonging to five distinct classes using the \texttt{make\_blobs} function in \cite{scikit-learn} with within-cluster standard deviation $1.1$. In the experiment we report, we consider a fully connected graph with Gaussian weights for $\sigma = 0.2$ and we choose $\eps = 1$ for the graph-Ginzburg-Landau energy. We assume that $1\%$ of points are labeled. Edge weights below $10^{-3}$ are thresholded to zero. The linear system does not change from time step to time step and is solved by LU-factorization with pre-computed factors.

In Figures \ref{figure blobs data} and \ref{figure blobs performance}, we compare the gradient flow-based and momentum-based optimization methods for the Ginzburg-Landau energy. In all cases, the initial condition is $u_0\equiv \frac15\cdot (1,1,1,1,1)$ on the set of unlabeled data points.

\begin{figure}
    \centering
    \includegraphics[width=0.32\linewidth]{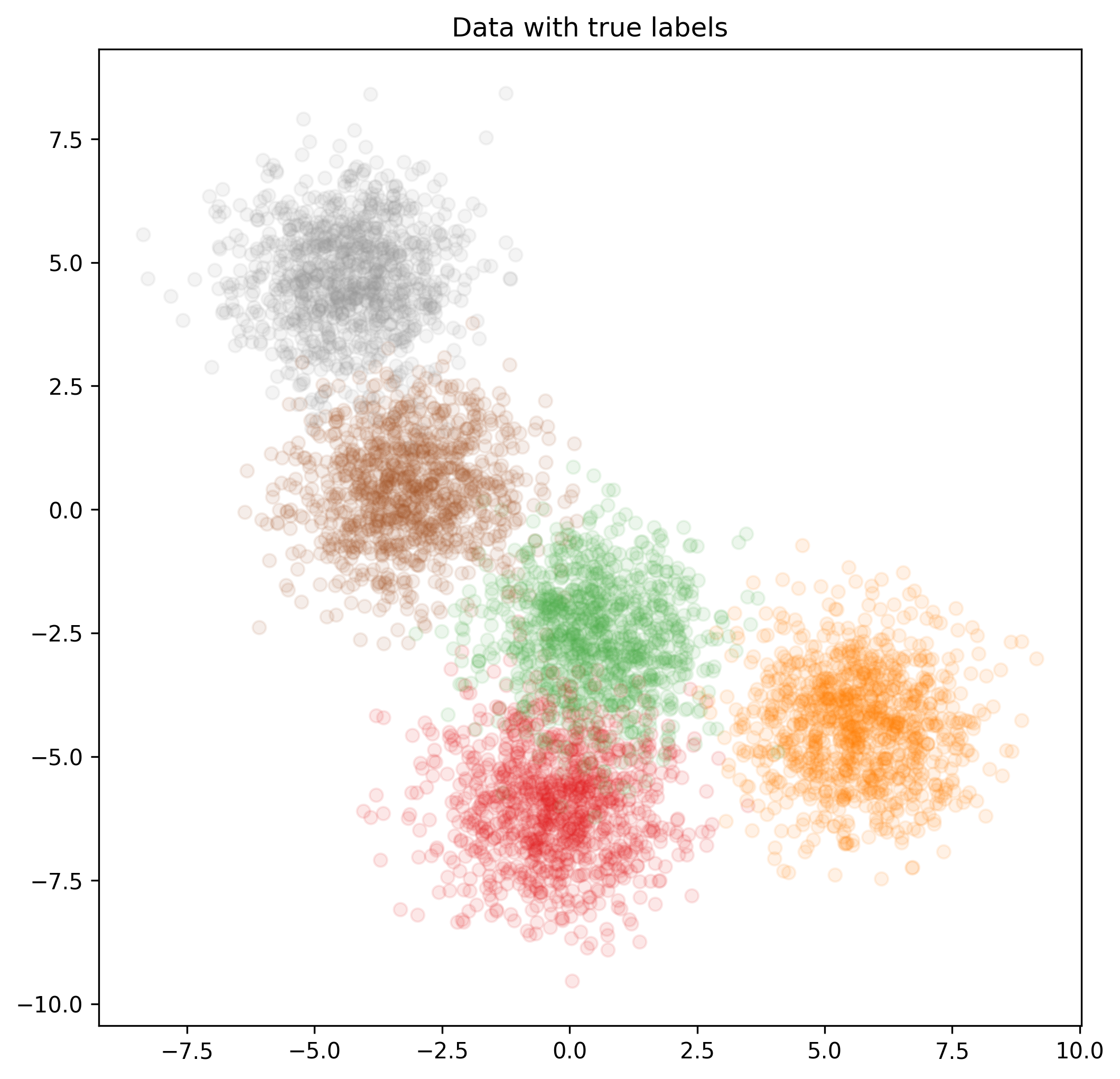}
    \includegraphics[width=0.32\linewidth]{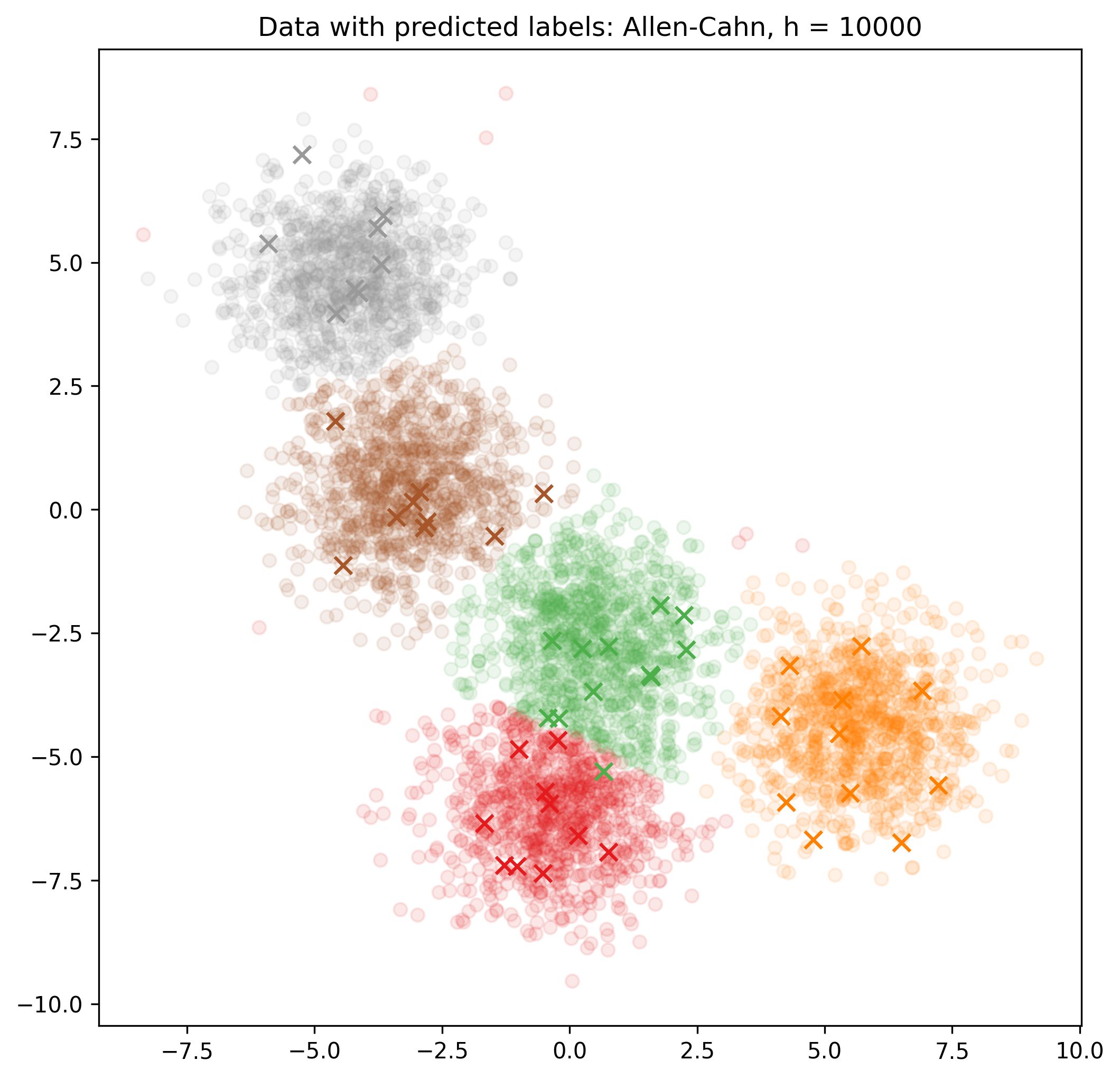}
    \includegraphics[width=0.32\linewidth]{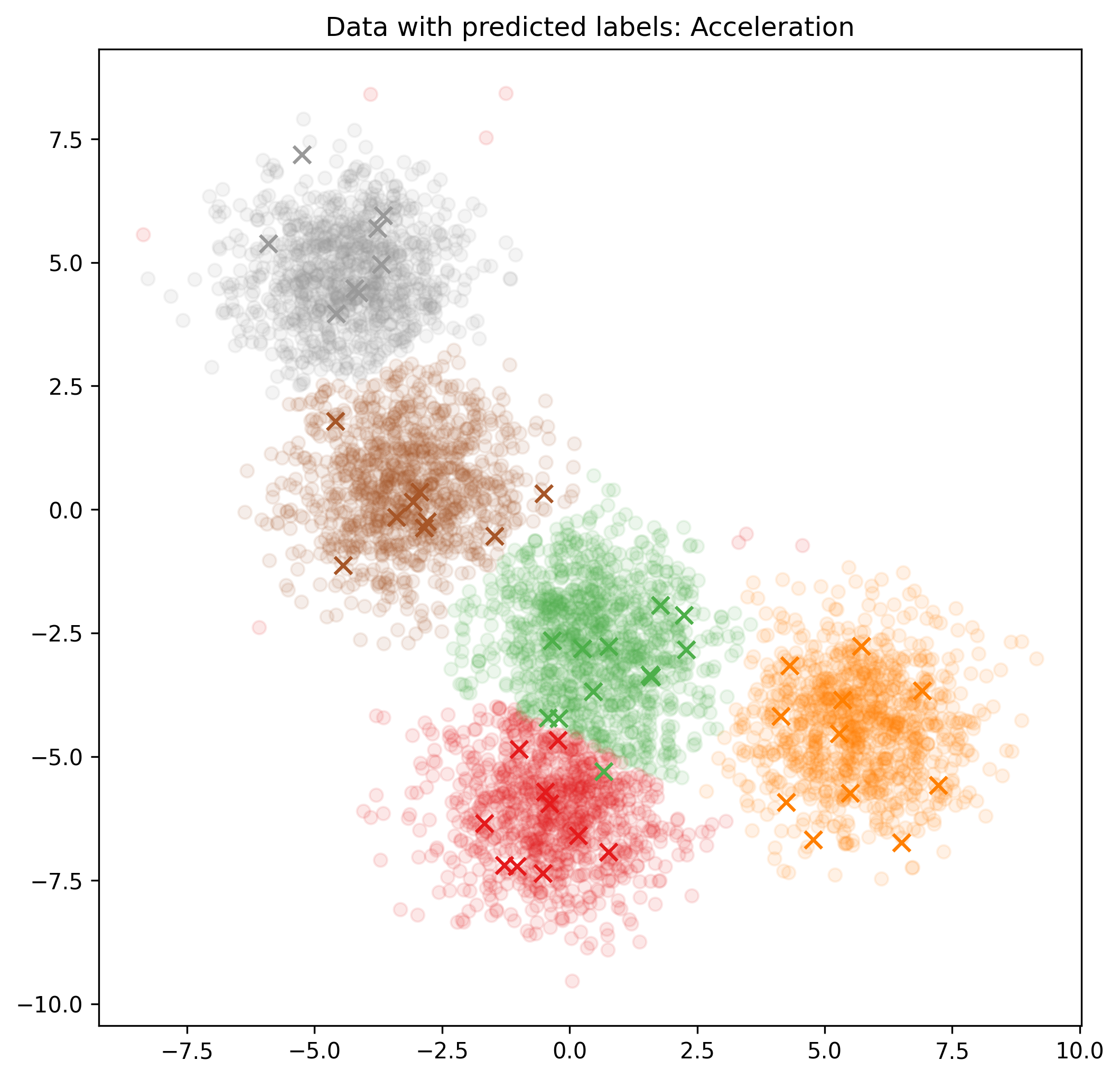}
    \caption{The synthetic blobs data colored by true label (left), the label suggested by gradient descent with convex-concave splitting (middle, $\tau$=10,000) and FISTA with convex-concave splitting (right, $\tau = 10, \rho = 0.4$). In the latter two plots, known labels are marked by a cross. Both methods find virtually identical label assignments. We note that (1) some outliers are disconnected from the main graph, so $u\equiv u_0$ by symmetry there (tie-broken to red label) and (2) the model prioritizes perimeter minimization and surrounds a `green' labeled point by red points in the bottom cluster to achieve a shorter boundary.}
    \label{figure blobs data}
\end{figure}

\begin{figure}
    \centering
    \includegraphics[width=0.48\linewidth]{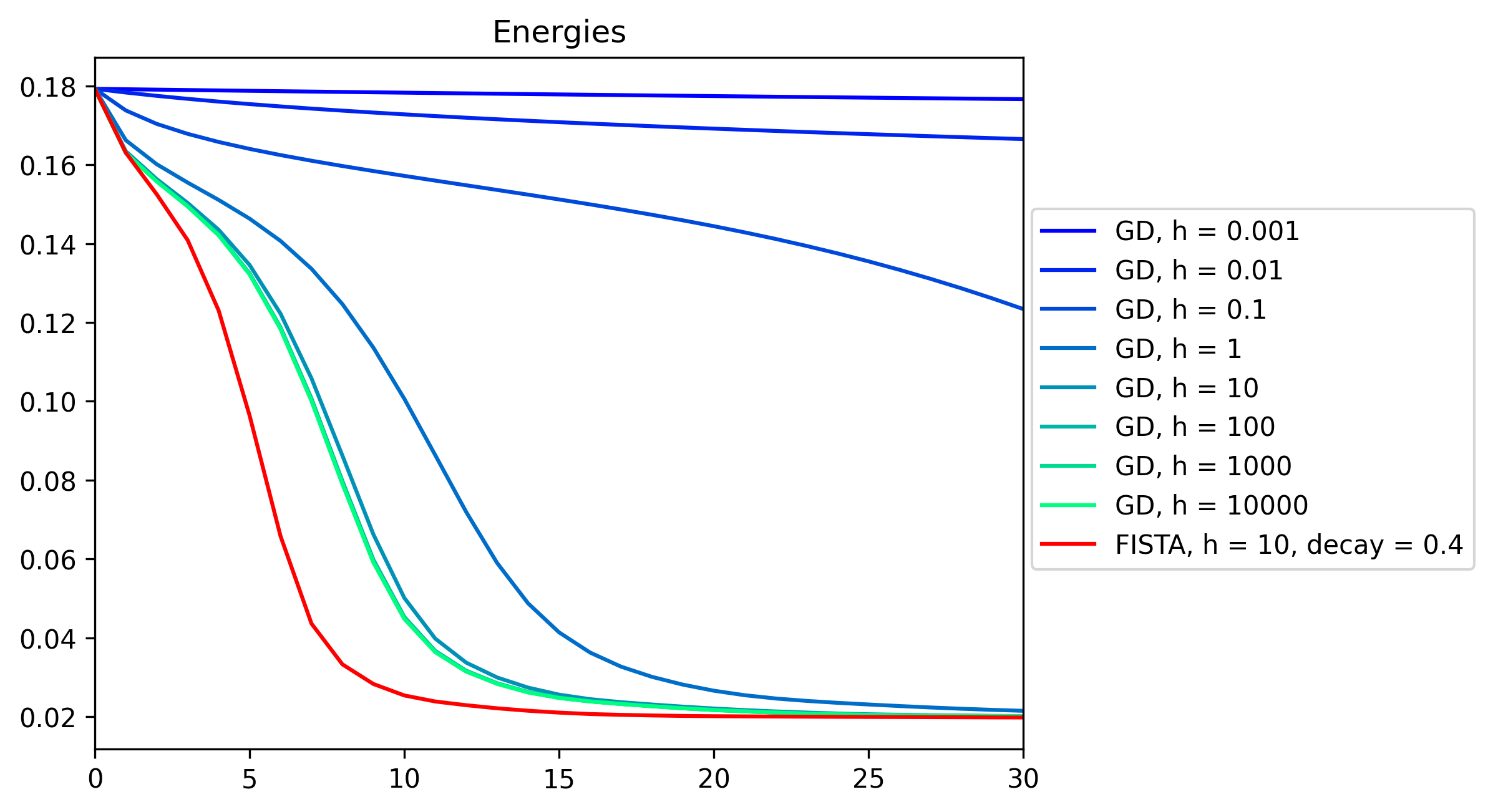}
    \includegraphics[width=0.48\linewidth]{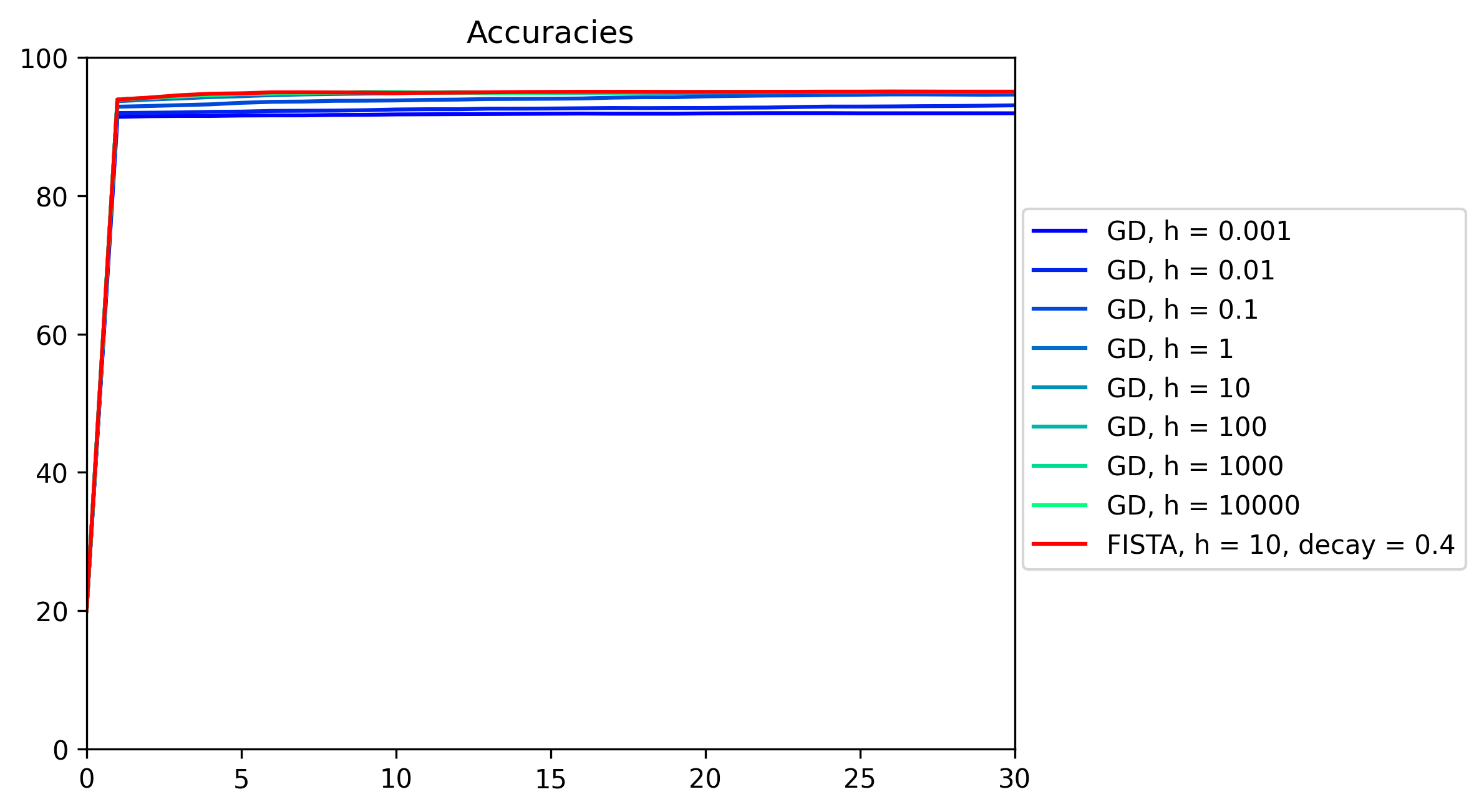}

    \caption{We compare the Allen-Cahn approach to classification with convex-concave splitting (blue and green lines for various time step sizes) with the FISTA-discretization with convex-concave splitting for the accelerated Allen-Cahn equation (red line). Notably, the Allen-Cahn scheme converges faster monotonically as the step size increases, but there is virtually no difference between the time step sizes $10^2, 10^3$ and $10^4$, suggesting that we have almost converged to the limiting model. On the other hand, FISTA reduces the Ginzburg-Landau energy faster than the best gradient-flow based minimizer. On the level of the accuracies, the convex-concave GD solver and FISTA achieve (almost) their final accuracy in a single time-step, with larger time steps favoring better local minimizers for GD.}
    \label{figure blobs performance}
\end{figure}

\subsubsection{MNIST}

We construct a $k$-nearest neighbor graph on the entire MNIST dataset of 70,000 images \cite{lecun2010mnist} for $k=5$ and assign Gaussian edge weights for $\sigma = 1.5$ (with image pixel values normalized to the range $[0,1]$). If $x_i$ is among the $k$ closest points to $x_j$, the opposite is generally not true. We therefore symmetrize the weight matrix as $W_{symm} = \frac12(W+W^T)$. Notably, weights for edges that were newly added in the symmetrization are only half as large as the Gaussian would suggest.

With one percent of labeled data and $\eps = 1$, we present classification accuracies and Ginzburg-Landau energy in Figure \ref{figure mnist performance} for both the Allen-Cahn approach and the accelerated Allen-Cahn approach to minimization. The arising linear system is sparse and we use a conjugate gradients solver. In all simulations, the initial condition was $u_0 \equiv \frac1{10}(1,\dots,1) \in\R^{10}$ at unlabeled points.

\begin{figure}
    \centering
    \includegraphics[width=0.48\linewidth]{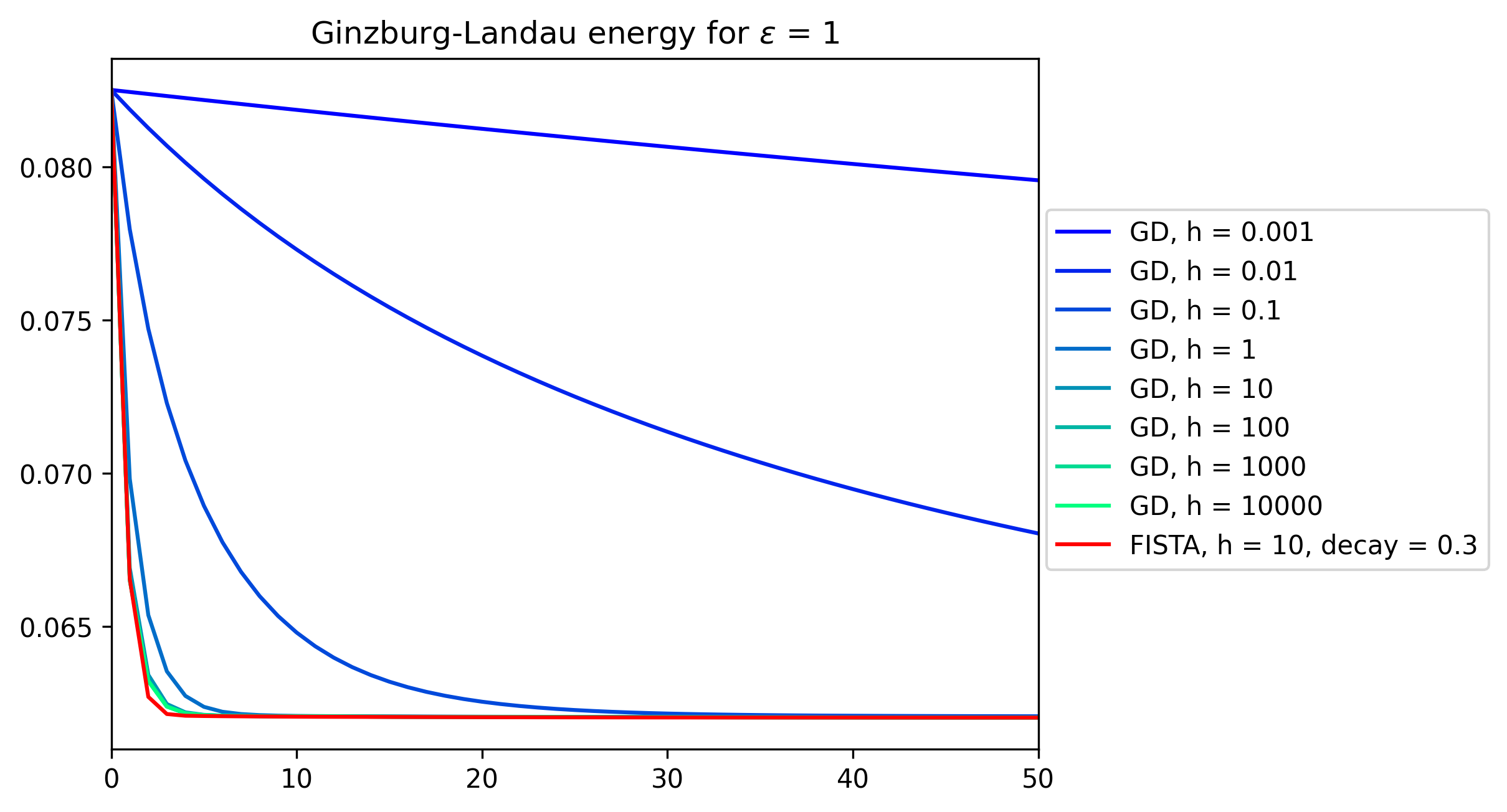}
    \includegraphics[width=0.48\linewidth]{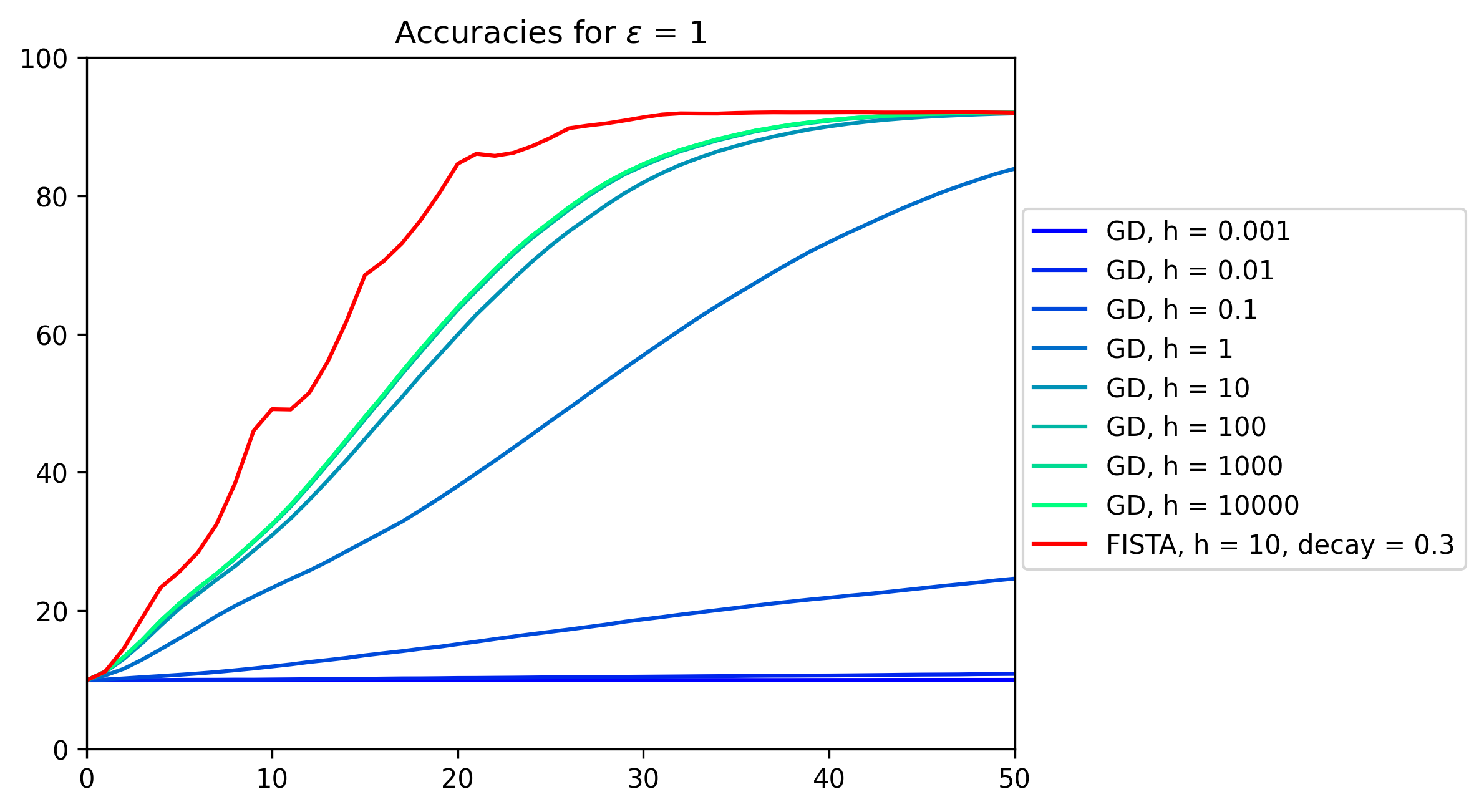}

    \caption{Also for MNIST data, we observe that GD with convex-concave splitting and step sizes $10^2, 10^3, 10^4$ performs essentially identically and that large time steps are desirable for the Allen-Cahn solver. In this experiment, FISTA converges faster than the best version of GD on the level of the accuracies, but there is no notable difference on the level of the energies.}
    \label{figure mnist performance}
\end{figure}

We note that the graphs are not constructed for performance. In these simple datasets, higher accuracy and faster convergence can be attained by simply choosing $\sigma, \eps$ (and $k$ for MNIST) advantageously. Often, it is possible to converge to a good local minimizer in a single time-step.

In general, constructing a `good' data graph is not an easy task. In the setting where a graph cannot easily be improved, our results show that it may be possible to accelerate convergence using momentum in the optimizer beyond what a mere convex-concave splitting can achieve.

\section{Conclusion} \label{section conclusion}

We have studied the `accelerated Allen-Cahn equation', a momentum based (`accelerated') method of optimization in a setting where momentum is not guaranteed to outperform gradient flow (i.e.\ the regular Allen-Cahn equation). We analyzed the basic geometric properties of solutions (Section \ref{section pde basics}) and obtained the singular limit of the equation as $\eps \to 0^+$ (Section \ref{section singular} and Figure \ref{figure verifying singular limit}). Numerical experiments illustrate the difference between the Allen-Cahn and accelerated Allen-Cahn equation (Figures \ref{figure mcf} and \ref{figure accelerated mcf}). Notably, in the small time step/PDE regime, the accelerated Allen-Cahn equation is in fact slower than the Allen-Cahn equation and exhibits the geometric features of a hyperbolic partial differential equation -- in particular, it is possible that interfaces may have sharp corners for a set of times of positive Lebesgue measure. Further research would be required to understand the deeper analytic properties of solutions to the limiting geometric motion.

The choice of a suitable large timestep discretization can have a large impact on the behavior of numerical solvers (Figures \ref{figure cinema vs fista pde} and \ref{figure large time step}), and it may alleviate some of the loss of regularity associated with second order hyperbolic equations. In simulations, a FISTA type discretization converges to a (local) minimizer significantly faster than a momentum-less method (Section \ref{section 3d}). A vector-valued graph-PDE version performs well for multi-class classification tasks on simple data sets (Figures \ref{figure blobs performance}, \ref{figure mnist performance}). Choosing a double-well potential with quadratic convex part allows for an implementation with low computational cost. 

Notably, our most encouraging results are purely experimental, and they concern the setting of large time steps where the PDE analysis does not apply. In fact, the lack of regularization and stability in FISTA appears to be driving its higher performance. In this setting, time slices do not resemble solutions to the accelerated Allen-Cahn equation in continuous time. Comparing to our insights on the accelerated Allen-Cahn equation, we conjecture that it is not the right object to study in order to better understand the performance of FISTA-based minimization of the Ginzburg-Landau functional.

\section*{Acknowledgements}

KG and SW gratefully acknowledge support through NSF grant 2424801 `Neural Networks for Stationary and Evolutionary Variational Problems'. PWD gratefully acknowledges discussions and help from L.\ Striet (Freiburg) with the implementation of Allen-Cahn equation for surfaces in three dimensions. These simulations use the \texttt{C++} ``Fourier Operators''-framework by L.\ Striet.

\appendix

\section{Optimal Profile and Corrector}\label{appendix corrector}

We first briefly review some useful facts about the optimal profile, which we expect to be well-known to the experts. We provide a full proof for completeness. For simplicity, we assume that the double-well potential $W:\R\to [0,\infty)$ is $C^\infty$-smooth and $W(u) =0$ if and only if $u\in\{-1,1\}$.

\begin{lemma}
There exists a unique function $\phi:\R\to\R$ such that
\begin{itemize}
\item $\phi'' = W'(\phi)$,
\item $\lim_{x\to \pm \infty} \phi(x) = \pm 1$, and
\item $\phi(0) = 0$.
\end{itemize}
Furthermore, $\phi$ is $C^\infty$-smooth and
\begin{enumerate}
\item $\phi$ is strictly monotone increasing and $\phi' = \sqrt{2\,W(\phi)}$,
\item $\int_\R \frac12 (\phi')^2 + W(\phi)\dx = \int_{-1}^1 \sqrt{2\,W(z)}\dz < + \infty$,
\item $\phi$ is odd if $W$ is even, and
\item If $W''(1)>0$, then for every $\alpha < \sqrt{W''(1)}$ there exist $C, R>0$ such that 
\[
|\phi(x) - 1|  + |\phi'(x)| + |\phi''(x)| \leq Ce^{-\alpha x}\qquad \forall\ x>R.
\]
\end{enumerate}
\end{lemma}

\begin{proof}
We note that $\phi'' = W'(\phi) $ implies that
\[
 \frac{d}{dx} \frac12\,(\phi')^2 = \phi' \phi'' = W'(\phi)\phi' = \frac{d}{dx} W(\phi)\qquad \Ra\quad (\phi')^2 \equiv 2\,W(\phi) + c
\]
for some $c\in\R$. We have $W(\phi(x))\to 0$ as $x\to \pm \infty$ since $\phi(x)\to \pm 1$, so $(\phi')^2 \to c$. Since $\phi'$ is continuous, it cannot change sign on $[R, \infty)$ or $(-\infty, -R]$ for some large $R$ if $c\neq 0$, meaning that $\phi$ could not be bounded. We conclude that $c=0$ for the solution we are interested in.\footnote{\ The constant $c$ governs a one-parameter group of qualitatively different solutions of the stationary Allen-Cahn equation $\phi''= W'(\phi)$: If $c<0$, the solutions are periodic with values in $(-1,1)$ and if $c>0$, they grow and may only exist on a finite interval, depending on the growth of $W$ at $\pm \infty$. The normalization $\phi(0)=0$ just centers $\phi$ since any translation $\phi(x+\tilde c)$ is also a solution.}\ We can therefore construct $\phi$ as the unique solution of the ODE 
\[
\begin{pde}
\phi' &= \sqrt{2\,W(\phi)} &x\neq 0\\
\phi &= 0 & x= 0.
\end{pde}
\]
Following the same argument backwards, we note that $\phi$ solves the original equation $\phi'' = W'(\phi)$ wherever $\phi'>0$, i.e.\ wherever $\phi \in (-1,1)$. A solution $\phi$ exists as long as $W$ is continuous, and we can always select a bounded solution
\[
\tilde \phi = \min\big\{1, \,\max\{-1, \phi\}\big\}
\]
even if $\sqrt{W}$ is not smooth at $\pm 1$. However, if $W\in C^2$, solutions are unique. We easily prove by contradiction that $\phi$ cannot be bounded away from $-1$ or $1$, i.e.\ $\lim_{x\to\pm \infty}\phi(x) = \pm 1$. Note that 
\begin{align*}
(\phi-1)' &= \phi' =  \sqrt{2\,W(\phi)}\\
	&= \sqrt{2\,W(1) + 2\,W'(1)\,(\phi-1) + W''(1)\,(1-\phi)^2 + O\big((1-\phi)^3\big)}\\
	&= \sqrt{W''(1)} (1-\phi) + O\big((1-\phi)^{3/2}\big)
\end{align*}
if $\phi(x)$ is close to $1$ and thus for all $\eps>0$ there exists $R>0$ such that 
\[
-(1+\eps)\sqrt{W''(1)}\, (1-\phi) \leq (1-\phi)' \leq -(1-\eps)\sqrt{W''(1)}\,(1-\phi) \qquad \forall\ x>R.
\]
By ODE comparison, we find that
\[
(1-\phi)(x) \leq (1-\phi)(R) \exp\left((1-\eps)\sqrt{-W''(1)}(R-x)\right) \leq \left(e^R(1-\phi)(R)\,\right)\,e^{-(1-\eps)\sqrt{W''(1)}\,x}
\]
and similarly for the lower bound. 
Notably, we have $-1<\phi(x) < 1$ for all $x\in\R$, so $\sqrt{2W}$ is locally Lipschitz continuous along the trajectory and the solution is unique by the Picard-Lindel\"off theorem.

The estimates for $\phi'$ and $\phi''$ follow from $\phi'(x) = \sqrt{2\,W(\phi(x))} \approx \sqrt{W''(1)}(1-\phi(x))$ and $\phi'' = W'(\phi(x)) \approx W''(1)(\phi(x)-1)$.

We use the identity $\phi' = \sqrt{2\,W(\phi)}$ to compute
\begin{align*}
\int_\R \frac12 (\phi')^2 + W(\phi)\dx &= \int_\R \frac12 \,\sqrt{2\,W(\phi)}\phi' + \sqrt{W(\phi)}\,\frac{\phi'}{\sqrt 2}\dx
	= \int_\R \sqrt{2\,W(\phi)}\phi'\dx = \int_{-1}^1 \sqrt{2\,W(z)}\dz.
\end{align*}
If $W$ is even, then $W'$ is odd, and the function $\tilde\phi(x) = - \phi(-x)$ satisfies the equation
\[
\tilde\phi'' (x) = -\phi''(-x) = - W'(\phi(-x)) = W'(-\phi(-x)) = W' (\tilde\phi(x))
\]
and $\lim_{x\to\pm \infty}\tilde\phi(x) = \pm 1$ as well as $\tilde\phi(0) = 0$. Since we showed that the solution of the equation is unique, this means that $\tilde\phi = \phi$, i.e.\ $\phi$ is odd.
\end{proof}

We can now prove that a corrector exists. For this, we assume that $W''(\pm 1)>0$ and that $W$ is even. While these assumptions can likely be weakened, they are satisfied by most standard double-well potentials. The condition that $W''(\pm 1)>0$ implies that $\phi$ approaches $\pm 1$ {\em exponentially} fast as $x\to\pm\infty$. 

\begin{lemma}
    There exists a unique solution $\psi \in H^1(\R)$ of the equation
    \begin{equation}\label{eq euler-lagrange corrector}
    \psi'' = W''(\phi)\psi  + \phi' + 2x\,\phi''
    \end{equation}
    such that
    \[
    \int_\R \phi'' \psi' \dx =0.
    \]
    Furthermore, $\psi$ is even and $\psi \in C^\infty(\R)$.
\end{lemma}

The proof is standard in Allen-Cahn-like settings where a corrector is needed -- see e.g.\ \cite[Theorem 3.2]{gonzalez2010slow} for a proof involving fractional differential operators. 

\begin{proof}
We aim to construct $\psi$ as a minimizer of the quadratic form
\[
F : H^1(\R) \to \R, \qquad F(\psi) = \frac12\int_\R |\psi'|^2 + W''(\phi)\,\psi^2\dx - \int_\R \psi\big(\phi' + 2x\phi''\big)\dx. 
\]
Since $W''$ is non-positive close to the origin, it is not immediately clear that $F$ is bounded from below or that minimization makes sense. However, if $\psi$ is a minimizer of $F$, it is clear that $\psi$ satisfies the Euler-Lagrange equation \eqref{eq euler-lagrange corrector}. The regularity of $\psi$ follows from standard ODE theory by bootstrapping. 

{\bf Step 1. Convexity of the quadratic part.} We first study the quadratic part 
\[
F^q(\psi) = \frac12\int_\R |\psi'|^2 + W''(\phi)\,\psi^2\dx
\]
only. The optimal interface equation $\phi'' = W'(\phi)$ implies that $\phi''' = W''(\phi)\phi'$ by differentiation. We can divide the identity by $\phi'= \sqrt{2\,W(\phi)}>0$ and thus obtain
\begin{align*}
\int_\R W''(\phi)\,\psi^2\dx &= \int_\R \phi'''\,\frac{\psi^2}{\phi'}\dx\\
	&= - \int_\R\phi'' \left(\frac{2\psi\psi'}{\phi'} - \psi^2\frac{\phi''}{(\phi')^2}\right)\dx\\
	&= \int_\R\phi' \left(\frac{2\psi\psi'}{\phi'}\right)'  +  \psi^2\frac{(\phi'')^2}{(\phi')^2}\dx\\
	&= \int_\R 2\,(\psi')^2 + 2 \psi \psi'' - 2\psi\psi'\,\frac{\phi''}{\phi'} + \psi^2\frac{(\phi'')^2}{(\phi')^2}\dx
\end{align*}
since we can integrate by parts without boundary terms as $H^1(\R) = H^1_0(\R)$. We note that
\[
\int_\R (\psi')^2 +\psi\psi'' \dx = 0
\]
by integration by parts, and after adding $\int_\R (\psi')^2\dx$ on both sides, we finally find that
\[
\int_\R (\psi')^2 + W''(\phi)\,\psi^2\dx = \int_\R (\psi')^2 - 2\psi\psi'\,\frac{\phi''}{\phi'} + \psi^2\frac{(\phi'')^2}{(\phi')^2}\dx = \int_\R\left|\psi' - \psi\,\frac{\phi''}{\phi'}\right|^2\dx \geq 0,
\]
so the quadratic form is bounded from below and in particular convex.\footnote{\ This statement is often formulated as: Layer solutions to the Allen-Cahn equation (in particular, the optimal interface) are stable.} We can the re-write the expression for a more informative bound:
\[
\int_\R\left|\psi' - \psi\,\frac{\phi''}{\phi'}\right|^2\dx = \int_\R\left|\frac{\psi'}\psi - \frac{\phi''}{\phi'}\right|^2\psi^2\dx = \int_\R\big|\log(|\psi|)' - \log(\phi')'\big|^2\psi^2\dx = \int_\R\left|\frac{d}{dx} \log\left(\frac{|\psi|}{\phi'}\right)\right|^2\psi^2\dx
\]
The expression is zero and only if $|\psi|/\phi'$ is constant (potentially zero), at least locally. Since $H^1$-functions are continuous in one dimension, we find that $F^q\geq 0$ and $F^q(\psi) = 0$ if and only if $\psi = a\phi'$ for some $a\in \R$.\footnote{\ The fact that $F^q(\phi') =0$ stems from the invariance of the Ginzburg-Landau energy under the spatial translation $x\mapsto x+ c$, as can be seen by taking the inner variation $\frac{d}{dt^2}\big|_{t=0} F_1^{GL}(\phi(\cdot +t))=0$.}

{\bf Step 2. Strong convexity/coercivity.} We consider the closed hyperplane
\[
\hat H^1 = \left\{\psi \in H^1(\R) : \int_\R \phi''\,\psi'\dx = 0\right\}.
\]
The space $\hat H^1$ is closed since the linear functional $\psi \mapsto \int_\R \phi''\,\psi'\dx $ is continuous on $H^1(\R)$ since $\phi'' = W''(\phi)\phi'\in L^2(\R)$. Note that $\phi'\notin \hat H^1$ since
\[
\int_\R  \phi'' (\phi')'\dx = \int_\R (\phi'')^2 \dx >0.
\]
In particular, any $\psi\in H^1(\R)$ can be uniquely decomposed as $\psi = \hat \psi + a\phi'$ for some $\hat\psi\in \hat H^1$ and $a\in\R$.
Let $\psi_n$ be a sequence in $\hat H^1$ such that
\begin{enumerate}
\item $\|\psi_n\|_{L^2(\R)}=1$ and
\item $F^q(\psi_n) \to \inf_{\psi \in \hat H^1, \: \|\psi\|_{L^2} = 1} F^q(\psi)$.
\end{enumerate}
Since
\[
C \geq F^q(\psi_n) \geq \frac{\|\psi'_n\|_{L^2} - \min_{z\in\R}W''(z)\,\|\psi_n\|_{L^2}^2}2 = \frac12 \,\|\psi_n'\|_{L^2(\R)}^2 - c 
\]
for all $n\in \mathbb N$, we see that $\psi_n$ is, in fact, bounded in $H^1(\R)$. By the weak sequential compactness of bounded subsets reflexive separable spaces \cite[Theorem 3.18]{brezis2011functional}, we find that $\psi_n\to \hat\psi$ weakly in $H^1(\R)$ for some $\hat\psi\in H^1(\R)$. Since the subspace $\hat H^1$ is convex, we may conclude that $\hat\psi \in \hat H^1$. 

We note that $H^1(\R)$ embeds continuously into $H^1(-R,R)$ and therefore compactly into $L^2(-R, R)$. We note that
\[
\liminf_{n\to \infty} F^q(\psi_n) \geq \frac12 \,\liminf_{n\to\infty} \left(\int_\R (\psi'_n)^2 + \max\{W''(\phi), 0\} \psi_n^2\dx + \int_\R\min\{W''(\phi),0\}\,\psi_n^2\dx\right) \geq F^q(\hat\psi) 
\]
since the first term is lower semi-continuous by convexity and the second is continuous due to strong convergence on bounded subsets. Since $\hat\psi \in \hat H^1$ and $\phi\notin \hat H^1$, we conclude that $F^q(\hat\psi)>0$ unless $\psi \equiv 0$. Then 
\[
\inf_{\psi\in \hat H^1, \:\|\psi\|_{L^2}=1} F(\psi) = \liminf_{n\to\infty} F(\psi_n) \geq F^q(\hat\psi)>0
\]
and thus 
\[
 F^q(\psi) = F^q\left(\frac{\psi}{\|\psi\|_{L^2}}\right)\,\|\psi \|_{L^2}^2 \geq \left(\inf_{\tilde \psi \in \hat H^1, \: \|\tilde\psi\|_{L^2}=1}F^q(\tilde\psi)\right)\, \|\psi\|_{L^2}^2 \geq F^q(\hat\psi)\,\|\psi\|_{L^2}^2
\]
for all $\psi\in \hat H^1$. On the other hand, if $\psi_n\wto 0$ in $H^1(\R)$, then $\psi_n\to 0$ in $L^2(-R,R)$ for all $R>0$. We therefore have
\[
\liminf_{n\to \infty} F^q(\psi_n) \geq \liminf_{n\to\infty} \frac12\int_\R (\psi_n')^2 + W''(\pm 1) \, \psi_n^2\dx \geq W''(\pm 1) >0
\]
since $W''(\phi)\approx W''(\pm 1)$ for $\phi$ close to $\pm 1$, i.e.\ for $x$ outside a compact set. Also in this case, we find a positive lower bound. Overall, we find that 
\[
F^q(\psi)\geq \eps_0 \|\psi\|_{L^2(\R)}^2, \qquad F^q(\psi) \geq \frac12\,\|\psi'\|_{L^2(\R)}^2 - c\,\|\psi\|_{L^2(\R)}^2
\]
for some constants $c, \eps_0>0$.

{\bf Step 3. Existence of a corrector $\psi$.} The functional $F$ is bounded from below on $\hat H^1$ since
\begin{align*}
F(\psi) &\geq \eps_0\,\|\psi\|_{L^2(\R)}^2 - \|\phi' + 2x\phi''\|_{L^2(\R)} \|\psi\|_{L^2(\R)} \\
	&\geq \min_{\alpha\in\R} \big(\eps_0 \alpha - \|\phi' + 2x\phi''\|_{L^2(\R)}\big) \alpha \hspace{1.5cm}= - \frac{\|\phi'+2x\phi''\|_{L^2(\R)}^2}{4\eps_0}.
\end{align*}
We claim that the functional $F$ has a minimizer $\psi^*\in \hat H^1$. First, we construct an energy competitor to obtain an a priori bound. Consider $\tilde\psi = \phi' + 2x\phi''$. Then
\[
\int_\R \tilde\psi \,\phi'' \dx = \int_\R \frac d{dx} \frac{(\phi')^2}2 +2x\,(\phi'')^2\dx = 0
\]
since the first term vanishes by the fundamental theorem of calculus and the second is the integral of an odd function.
Now, since $\phi' + 2x\phi'' \neq 0$, we find that the linear term dominates when considering
\[
F\big(-t(\phi' + 2x\phi'' )\big) = t^2 F^q(\phi'+2x\phi'') - t\,\|\phi'+2x\phi''\|_{L^2(\R)}^2 < 0
\]
for sufficiently small $t>0$ and thus $-\infty < \inf_{\psi\in \hat H^1} F(\psi) < 0$. We can therefore consider a sequence $\psi_n \in \hat H^1$ such that $F(\psi_n)\to \inf_{\psi \in \hat H^1}F(\psi)$ and $F(\psi_n) <0$ for all $n\in \mathbb N$. As above, we estimate
\[
0 \geq F(\psi_n) \geq \big(\eps_0\,\|\psi_n\|_{L^2(\R)} - \|\phi' + 2x\phi''\|_{L^2(\R)} \big)\|\psi_n\|_{L^2(\R)} \qquad\Ra\quad \|\psi_n\|\leq \frac{\|\phi' + 2x\phi''\|_{L^2(\R)}}{\eps_0}
\]
for all $n\in\mathbb N$. By the alternative lower bound  
\[
0 > F^q(\psi_n) \geq \frac12\,\|\psi_n'\|_{L^2(\R)}^2 - c\,\|\psi_n\|_{L^2(\R)}^2 \geq \frac12\,\|\psi_n'\|_{L^2(\R)}^2 - c\,\frac{\|\phi' + 2x\phi''\|_{L^2(\R)}^2}{\eps_0^2}
\]
on $F^q$, we find that also $\psi_n'$ is uniformly bounded in $L^2$. Convergence and lower semi-continuity of $F$ follow as in Step 2, so a minimizer $\psi^*\in \hat H^1$ of $F$ exists by the direct method of the calculus of variations since the linear functional is continuous in the weak topology.

We claim that $\psi^*$ is a minimizer of $F$ not just in the hyperplane $\hat H^1$. To see this, note that every $\xi \in H^1(\R)$ can be written as $\xi = \hat\psi + b\phi'$ for unique $\hat \psi \in \hat H^1$ and $b\in\R$ since $\phi'\notin \hat H^1$. We now have
\begin{align*}
F(\xi) &= \int_\R \frac12 \,(\hat \psi')^2 + b\,\hat\psi' \phi'' + \frac{b^2}2\,(\phi'')^2 + W''(\phi)\,\big(\hat\psi^2 + 2b\,\hat\psi\phi' + b^2(\phi')^2\big)\dx\\
	&\qquad - \int_R \big(\phi' + 2x\phi'')\big(\hat \psi + b\phi'\big)\dx\\
	&= \frac12 \int_\R (\hat \psi')^2 + W''(\phi)\hat\psi^2\dx - \int_R \big(\phi' + 2x\phi'')\hat \psi\dx + \frac{b^2}2 \int_\R(\phi'')^2 + W''(\phi)(\phi')^2 \dx\\
	&\qquad - b \int_R \big(\phi' + 2x\phi'')\phi'\dx + b \int_\R \hat \psi' \phi'' + W''(\phi)\,\phi'\hat\psi\dx.
\end{align*}
We note that 
\[
\int_\R(\phi'')^2 + W''(\phi)(\phi')^2 \dx = \int_\R \left(-\phi''' + W''(\phi)\phi'\right)\phi'\dx = 0
\]
since the Ginzburg-Landau energy of the optimal profile does not change under coordinate shifts. Integrating by parts, we have 
\[
\int_R \big(\phi' + 2x\phi'')\phi'\dx = \int_\R (\phi')^2 + 2x \,\phi'\phi'' \dx = \int_\R (\phi')^2 + x\,\frac{d}{dx}(\phi')^2 \dx = \int_\R (\phi')^2 - (\phi')^2\dx =0.
\]
Finally, we note that
\[
\int_\R \hat \psi' \phi'' + W''(\phi)\,\phi'\hat\psi\dx= \int_\R \phi'' \hat\psi' + (W'(\phi))' \hat\psi\dx = \int_\R\phi'' \hat\psi' + \phi''' \hat\psi\dx = 2 \int_\R \phi'' \hat\psi' \dx = 0
\]
by the definition of $\hat H^1$. We conclude that
\[
F(\xi) = F(\hat\psi + b\phi') = F(\hat\psi) \geq F(\psi^*),
\]
i.e.\ $\psi^*$ minimizes $F$ over the entire space $H^1(\R)$. We conclude that $\psi^*$ solves the Euler-Lagrange equation
\[
-(\psi^*)'' + W''(\phi)\,\psi^* = \phi' + 2x\phi''.
\]

{\bf Step 4. Uniqueness.} We have shown that $F$ is strictly convex on $\hat H^1$, i.e.\ the minimizer $\psi^*$ is unique. Since $\phi''$ is an odd function, we can substitute $z=-x$ to deduce that 
\[
\int_\R \phi''(x)\frac{d}{dx}\psi(-x)\dx = -\int_\R \phi''(x) \psi'(-x)\dx = \int_\R\phi''(-z)\psi'(z)\dz = - \int_\R \phi''(z) \,\psi'(z)\dz
\]
meaning that the space $\hat H^1$ is invariant under the change of coordinates $z=-x$. 
 We further note that $W''(\phi), \phi'$ and $x\phi''$ are even functions, meaning that $F$ is invariant under the replacement of $x\mapsto \psi(x)$ by $x\mapsto \psi(-x)$. Thus, if $\psi$ minimizes $F$, then so does $x\mapsto\psi(-x)$. From the uniqueness of minimizers, we deduce that $\psi^*(x) = \psi^*(-x)$ for all $x\in\R$, i.e.\ $\psi^*$ is even.
 
 All minimizers of the functional $F$ in $H^1(\R)$ have the form $\psi = \psi^* + a \phi'$ for some $a\in\R$. Since $\phi'(0) = \sqrt{2\,W(\phi(0))}>0$, there exists an affine linear map from $\psi(0)$ to $ \int_\R \phi'' \psi'\dx$. We may therefore solve
 \[
 \begin{pde}
 \psi_c'' &= W''(\phi)\,\psi_c & x>0\\
 \psi_c &= c &x=0\\
 \psi_c' &= 0 &x=0
 \end{pde}
 \]
 and extend $\psi_c$ by even reflection to the real line in order to find $\psi^*$ numerically. This is implemented numerically in Figure \ref{figure corrector equation}. To find the correct value of $c= \psi(0)$, we only need to find the root of the linear function $c\mapsto \int_\R \phi'' \psi_c'\dx$. 
 
All even solutions of the corrector ODE have the form $\psi^*+c\phi'$ for some $c\in\R$. The odd solutions do not correspond to minimizers of the energy $F$, but they would be critical points of the energy if they were in the domain $H^1(\R)$. Since $F$ is convex, this would make them minimizers, which means that the odd solutions cannot lie in $H^1(\R)$. 
 \end{proof}
 
  \begin{figure}
  \includegraphics[width = .48\textwidth]{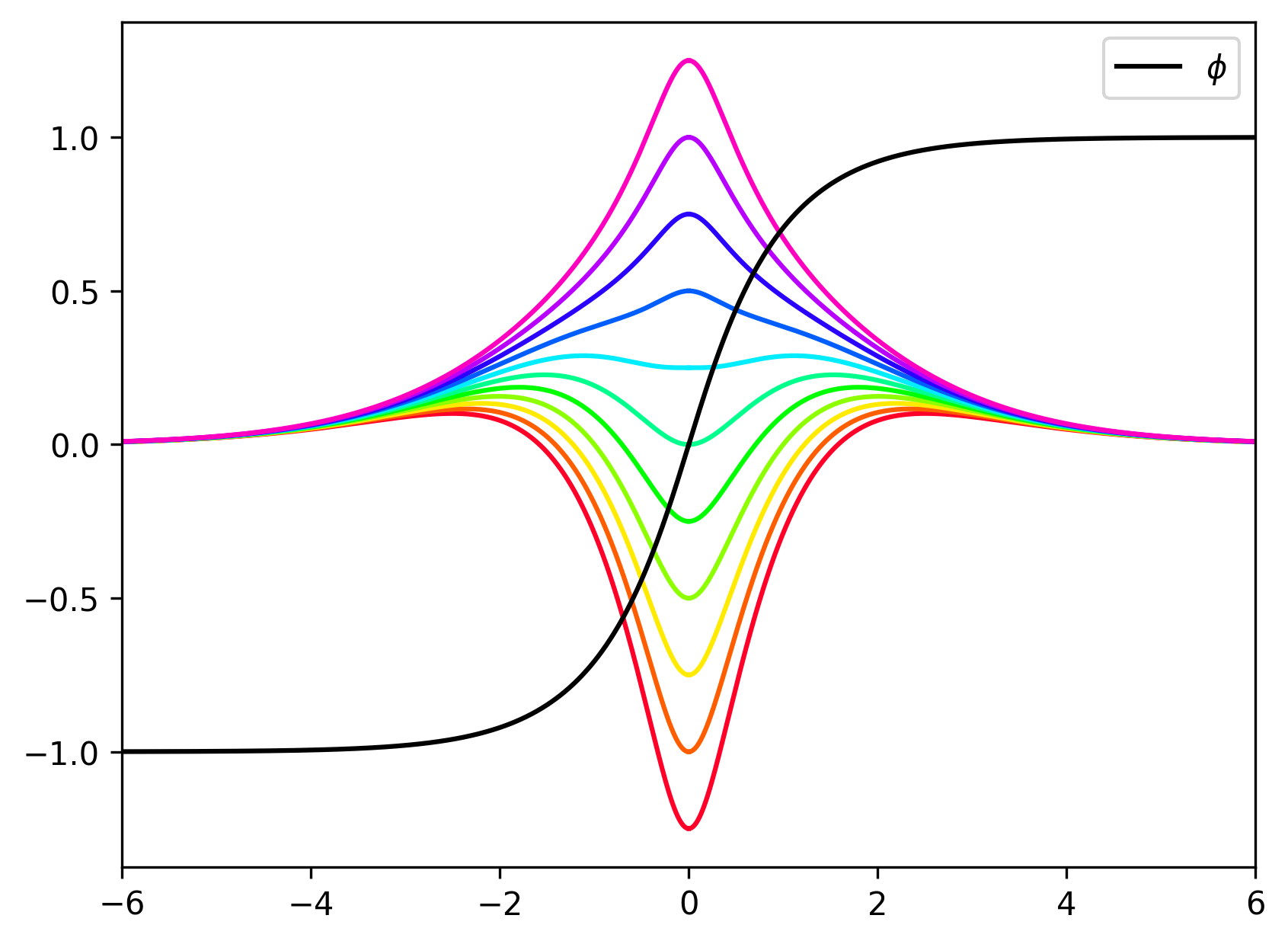}\hfill
 \includegraphics[width = .48\textwidth]{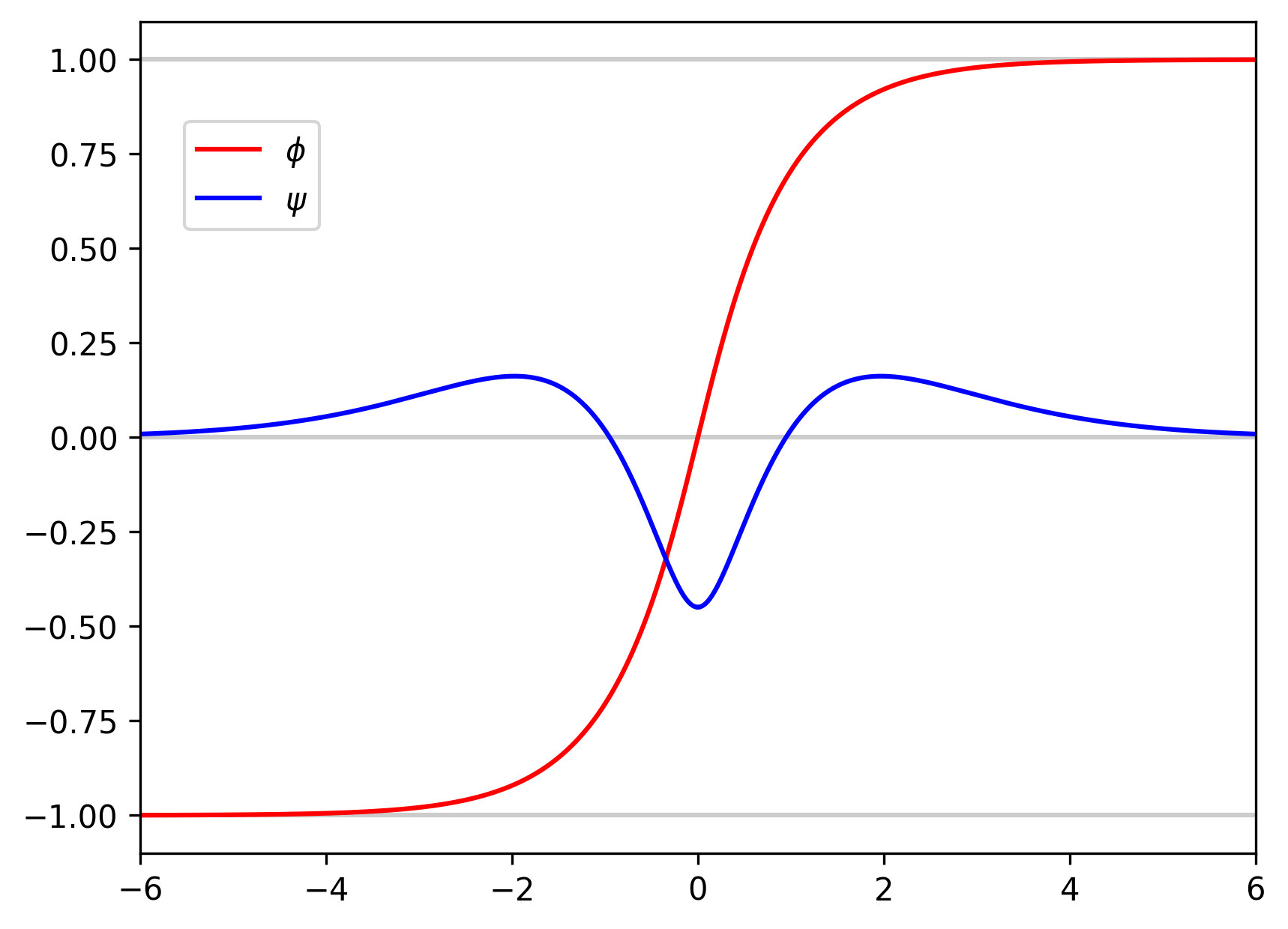}

 \caption{
 {\bf Left:} Even solutions of the corrector equation with different values at zero. {\bf Right:} The true corrector.
 \label{figure corrector equation}
 }
 \end{figure}

\begin{remark}
If we tried to solve
\[
\psi'' - W''(\phi)\psi = \phi' - cx\phi''
\]
for $c\neq 2$, we could still carry out the minimization in $\hat H^1$, but the minimizer would no longer be global and only solve an Euler-Lagrange equation with a Lagrange multiplier.
\end{remark}

\section{Proofs for CINEMA}\label{appendix cinema}

\begin{proof}[Proof of Theorem \ref{theorem descent existence}]
For $x,v\in H$, we consider the function 
\[
M:H\to\R, \qquad M(z) = \frac12\|z-x\|^2 + \langle \eta \nabla G(x) - \tau v, z \rangle + \eta F(z).
\]
 We observe that for for fixed $x$ and $v$, $M$ is bounded below since
\begin{align*}
M(z) &= \frac12\|z\|^2 + \frac12 \|x\|^2  + \langle \eta \nabla G(x) - \tau v - x, z \rangle + \eta F(z)  \\
&\geq \frac12\|z\|^2  - \|\tau v + x\|\,\|z\| + \| \eta \nabla G(x)\| \| z\| + \eta \big(F(\xi) - \|\partial F(\xi)\|\,\|z-\xi\|\big)
\end{align*}
for some $\xi$ in the domain of $\partial F$. The expression on the right is bounded from below even if $F$ is not. Suppose that $z_n$ is a minimizing sequence of $M$, i.e.\ $\lim_n M(z_n) = \inf_{z\in H} M(z)$. Then the inequality above implies that $z_n$ is a bounded sequence. By the Banach-Alaoglu theorem and the fact that $H$ is a a separable Hilbert space, $z_n$ has a weakly convergent subsequence. For notational convenience, we can assume that $z_n$ itself is weakly convergent and $z^*$ is its weak limit. Then $\lim_n \langle \eta \nabla G(x) - \tau v, z_n \rangle = \langle \eta \nabla G(x) - \tau v, z^* \rangle$. Further, since $F$ and the norm $\|\cdot\|_H$ are both convex and lower semi-continuous, they are also weakly lower semi-continuous, i.e.\ $F(z^*) \leq \liminf_n F(z_n)$ and $\|z^*-x\|^2 \leq \liminf_n \|z_n-x\|^2$. Hence, we have shown that 
\[
M(z^*) \leq \liminf_{n\to\infty} M(z_n) = \inf_z M(z) \leq M(z^*).
\]
We can conclude that $z^*$ is the minimizer of $M$ and thus $0\in \partial M(z^*)$, which is equivalent to
\[
  \frac{x + \tau v - z^*}\eta - \nabla G(x) \in \partial F (z^*). \qedhere
\]
\end{proof}

\begin{proof}[Proof of Theorem \ref{theorem descent scheme}]
 Existence of the scheme follows from Theorem \ref{theorem descent existence}. For convenience, we denote $s_{n+1} \in \partial F(x_{n+1})$  the element selected by \eqref{eq cinema time step} and $g_n = s_{n+1} + \nabla G(x_n)$. Using the first order convexity condition for $F$, we find that
\[ 
F(x_{n+1}) - F(x_n) \leq \langle s_{n+1}, x_{n+1} - x_n \rangle = \langle s_{n+1}, \tau v_{n}-\eta g_{n} \rangle.
\]
On the other hand, due to the concavity of $G$, we have
\[
G(x_{n+1}) - G(x_{n}) \leq \langle \nabla G(x_{n}), x_{n+1}-x_{n} \rangle = \langle\nabla G(x_{n}), \tau v_n -\eta g_{n}\rangle.
\]
Adding the last two inequalities,
\begin{align*}
F(x_{n+1}) + G(x_{n+1})
    &\leq F(x_n) + G(x_{n}) - \eta \|g_n\|^2 
    + \langle g_n , \tau v_{n}\rangle.
\end{align*}
From this inequality and the definition of $v_{n+1}$ it follows that
\begin{align*}
   e_{n+1} &= F(x_{n+1}) + G(x_{n+1}) + \frac1{2\rho^2}\|v_{n+1}\|^2\\
   &\leq F(x_n) + G(x_{n}) - \eta \|g_n\|^2 
    + \langle g_n, \tau v_{n}\rangle + \frac12 \|v_n\|^2 + \frac{\tau^2}2\|g_n\|^2 - \langle v_n, \tau g_n\rangle \\
   &= F(x_n) + G(x_{n}) + \frac12 \|v_n\|^2 + \left(\frac{\tau^2}2 - \eta\right)\|g_n\|^2\\
    &= e_n - \frac12(\rho^{-2}-1)\|v_n\|^2 + \left(\frac{\tau^2}2 - \eta\right)\|g_n\|^2.\qedhere
\end{align*}
\end{proof}

\section{Computational details}

\subsection{From problem to implementation}

We briefly describe implementations in the PDE setting and the graph setting. For now, we focus on scalar-valued PDEs -- generalizations are discussed below.

{\bf Problem data.} The domain $\Omega$ (or graph $\Gamma$) is assumed to be given. 
\begin{enumerate}
\item In semi-supervised learning, we model known labels as a time-independent Dirichlet boundary condition $u(t, x_i) = y_i$ for all $i$ where $y_i$ is known.

The initial condition can be freely chosen by the practitioner. In this article, we opted for a constant `average' initial condition, which we found to perform better than randomly selected initial values. In principle, initial values could be selected based on a nearest neighbor scheme or similar.

\item In PDE applications, the initial condition $u_0$ is generally part of the given problem, along with the boundary data. In our numerical experiments, we considered periodic boundary conditions, i.e.\ there is no boundary on which to assign specific values.
\end{enumerate}

For simplicity, we describe an implementation for periodic or time-independent Dirichlet boundary conditions, but other choices are possible with minor modifications.

{\bf Choice of model and discretization for the energy.} In the graph setting, the practitioner needs to select a version of the graph Laplacian which gives rise to a non-negative Dirichlet energy. We find the unnormalized graph Laplacian most natural since it gives zero Dirichlet energy to constant functions. However, we note that it may be advantageous to alter the graph once the set of labeled vertices is known in such fashion as \cite{calder2020properly}.

In the PDE setting, the Laplacian is given as an operator between function spaces, but a choice of discretization needs to be made. We opted for a Fourier approach, but finite elements are a valid alternative choice (and in fact needed for more complicated domains or boundary conditions). For either method, a specific discretization must be chosen (often described by a characteristic length scale $h>0$).

Additionally, the double-well potential $W$ and parameter $\eps>0$ have to be specified in order to characterize a specific discrete Ginzburg-Landau energy. In the continuous setting, we typically opt for $\eps= Ch$ with a sufficiently large constant $C>0$ in order to resolve transitions between phases sufficiently accurately (and with momentum, $C$ may have to be higher). In order to establish $\Gamma$-convergence, the stronger asymptotic condition $h/\eps_h \to 0$ is needed.

In the setting of graphs, the weights $w_{ij}$ carry the meaning of $h^{-2}$, i.e.\ we can choose $\eps$ proportional to $1/\sqrt{h_\Gamma}$ where $h_\Gamma$ is e.g.\ the average of the distance of a point to its nearest neighbor or $k'$ nearest neighbors.

We always assume that $W$ is an even non-negative $C^2$-function which only vanishes at $-1$ and $1$ and satisfies $W(u) = u^2 + W_{conc}(u)$ for a concave function $W_{conc}$. 

{\bf Choice of minimization algorithm and hyperparameters.} Given the energy characterized by $\eps$, notion of graph-Laplacian and $W$, we may freely select a gradient flow-based (Allen-Cahn) optimization algorithm or a momentum-based (accelerated Allen-Cahn) algorithm.

Within these two options, further choices of time discretization are available. We describe the simple explicit Euler and the FISTA discretization, both utilizing a convex-concave splitting. 

The gradient descent algorithm requires a choice of time step size $\tau>0$. FISTA requires a choice of time-step size $\tau>0$ and decay factor $\rho$, or a friction parameter $\alpha$ from which $\rho$ can be computed e.g.\ as $\rho = 1/(1+\alpha\tau)$.

{\bf Implementation.} All scenarios described above fall into the framework of minimizing
\[
F_\eps^{GL} :\R^N \to [0,\infty), \qquad F_\eps^{GL}(u) = \frac\eps2\,u^TLu + \frac1\eps \,\sum_i W(u_i).
\]
We assume that there exists a set of indices $I_{bd}\subseteq \{1,\dots, N\}$ of cardinality $0\leq N_{bd} < N$ such that $u_i$ is prescribed for $i \in I_{bd}$. Without loss of generality, we may assume that $I_{bd} = \{N- N_{known}+1, \dots, N\}$. We denote $I_{int} = \{1,\dots, N\}\setminus I_{bd}$. Accordingly, we segment
\[
u = \begin{pmatrix} u_{int}\\ u_{bd}\end{pmatrix}, \qquad L = \begin{pmatrix} L_{int} &L_\times\\ L_\times^T &L_{bd}\end{pmatrix}
\]
and note that
\[
u^TLu = u_{int}^T \,L_{int} u_{int} + 2\,u_{int}^T \underbrace{L_\times u_{bd}}_{=:f_{bd}} + \underbrace{u_{bd}^TL_{bd}u_{bd}}_{=: 2\,c_{bd}}.
\]
We denote $N_{int} = N - N_{bd}$. The algorithms can be implemented as follows. 

\begin{minipage}[t]{.48\textwidth}
\vspace{0pt}
\begin{algorithm}[H]
\caption{GD}\label{algorithm gd}
\Input{$I_{bd}, u_{bd}$, $u_0$, $L\in\R^{N\times N}_{spd}$, $\tau>0$, $T>0$.}
$S \gets (1+ 2\tau/\eps^2) I_{N_{int}} + L_{int}$\\
$f_{bd} \gets L_\times u_{bd}$\\
set up linear solver\\
$u\gets (u_0)_{int}$\\
${\color{white}{\phi}}$\\
\While{$t<T$}{
   ${\color{white}{\phi}}$\\
    $rhs \gets u - (\tau/\eps^2)\,W_{conc}'(u) - \tau\,f_{bd}$\\
    $u\gets S^{-1}\,rhs$\\
    ${\color{white}{\phi}}$\\
    ${\color{white}{\phi}}$\\
    ${\color{white}{\phi}}$\\
}
\Return{$u$}
\end{algorithm}
\end{minipage}
\hfill
\begin{minipage}[t]{.48\textwidth}
\vspace{0pt}
\begin{algorithm}[H]
\caption{FISTA}\label{algorithm fista}
\Input{$I_{bd}, u_{bd}$, $u_0$, $L\in\R^{N\times N}_{spd}$, $\tau>0$, $\rho\in(0,1)$, $T>0$.}
$S \gets (1+ 2\tau^2/\eps^2) I_{N_{int}} + L_{int}$\\
$f_{bd} \gets L_\times u_{bd}$\\
set up linear solver\\
$u\gets (u_0)_{int}$\\
$v\gets 0\in \R^{N_{int}}$\\
\While{$t<T$}{
    $u \gets u+ \tau v$\\
    $rhs \gets u - (\tau^2/\eps^2)\,W_{conc}'(u) - \tau^2\,f_{bd}$\\
    $u_{new}\gets S^{-1}\,rhs$\\
    $g\gets (u-u_{new})/\tau^2$\\
    $v\gets \rho(v- \tau\,g)$\\
    $u\gets u_{new}$
    }
\Return{$u$}
\end{algorithm}
\end{minipage}

\vspace{2mm}

In particular, FISTA includes the same convex-concave split gradient descent step (with step size $\tau^2$ rather than $\tau$), and only requires additional vector additions (and storage for two additional vectors $u_{new}, v$, if implemented without saving $g$).

The constant $c_{bd}$ does not need to be computed for the algorithm, but should be computed along with $f_{bd}$ in order to track the energy throughout the evolution. As an additive constant, it could be neglected, but it ensures positivity of $F_\eps^{GL}$ and thus $E_\eps$.

The linear solver may not need to be set up, but for LU solvers for instance, the system matrix $S$ should be factorized only once before the first time step. We use an LU solver for small densely connected graphs, a CG solver for sparsely connected large graphs, and a direct coordinatewise solver in the Fourier domain in the PDE setting.

\subsection{An enclosed volume-preserving version}\label{appendix volume-preserving}

In Section \ref{section 3d}, we consider the minimization of the Ginzburg-Landau energy under a fixed enclosed volume constraint $\frac1{|\Omega|}\int_\Omega u\dx = p\in(-1,1)$. Algorithms \ref{algorithm gd} and \ref{algorithm fista} easily extend to this situation. We present the modification in the setting of partial differential equations.

Abstractly, in Algorithm \ref{algorithm gd}, rather than selecting $u_{n+1}$ as the unique minimizer of
\begin{align*}
\frac1{2\tau} \,\|u- u_n\|^2_{L^2(\Omega)} + \int_\Omega \frac 12\,\|\nabla u\|^2 + \frac{u^2 + W_{conc}(u_n) + W_{conc}'(u_n)\,(u-u_n)}{\eps^2} \dx 
\end{align*}
or equivalently of 
\begin{align*}
\frac1{2\tau} \,\|u- u_n\|^2_{L^2(\Omega)} + \int_\Omega \frac 12\,\|\nabla u\|^2 + \frac{u^2 + W_{conc}'(u_n)\,u}{\eps^2} \dx
\end{align*}
in the affine space $u_{bd} + H_0^1(\Omega)$, we select the unique minimizer of the same functional in the smaller affine space 
\[
H = \left\{ h\in H^1(\Omega) : Tr_{\partial\Omega}(u) = u_{bd}, \:\:\frac1{|\Omega|}\int_\Omega u \dx = p\right\}.
\]
By the Lagrange-multiplier theorem, the Euler-Lagrange equation of the modified minimization problem is
\begin{equation}\label{eq lagrange multiplier volume}
\frac{u-u_n}\tau - \Delta u + \frac2{\eps^2}u + \frac{W_{conc}'(u_n)}{\eps^2} \equiv \lambda
\end{equation}
for some $\lambda\in\R$. Crucially, since the convex part of $W$ is quadratic, its derivative is linear, yielding the linear PDE \eqref{eq lagrange multiplier volume}. In particular, if $\Omega$ is a closed manifold (in particular, a periodic domain/torus) we may solve the Euler-Lagrange equation
\[
-\tau\,\Delta u + \left(1+\frac{2\tau}{\eps^2}\right) u = u_n - \frac{W_{conc}'(u_n)}{\eps^2}
\]
of the unconstrained problem for $\tilde u_{n+1}$ and the replace $\tilde u_{n+1}$ by $\tilde u_{n+1} + \frac1{|\Omega|} u_n - \tilde u_{n+1}\dx$. The Lagrange multiplier is the unique $\lambda\in\R$ leading to integral preservation along the scheme.
If $\partial\Omega \neq \emptyset$, we need to replace the constant function $1$ by the solution to the PDE
\[
\begin{pde}
-\tau\,\Delta u_{av} + \left(1+\frac{2\tau}{\eps^2}\right) u_{av} &= 1 &\text{in }\Omega\\ 
u_{av} &= 0 &\text{on }\partial\Omega.
\end{pde}
\]
The corrector $u_{av}$ for preserving the average only needs to computed when the algorithm is set up since \eqref{eq lagrange multiplier volume} is linear in $u$ -- a major advantage of the potentials whose convex part is quadratic. 

An even easier implementation is available in a Fourier discretization. On the periodic unit hypercube, the Fourier coefficients of $u_{n+1}$ can be found as
\[
a_{k}(u_{n+1}) = \frac{1}{1 + \frac{2\tau}{\eps^2} + \tau (2\pi)^2\|k\|_2^2} \, a_{k}\left(u_n - \frac{\tau}{\eps^2}\,W_{conc}'(u_n)\right)
\]
for $k = (k_1,\dots, k_d) \neq 0$ and $a_0 = p$ in order to prescribe the average. In the FISTA algorithm, if the gradient descent step is implemented in this manner such that
\[
\frac1{|\Omega|}\int_\Omega u_{new}\dx = \frac1{|\Omega|}\int_\Omega u\dx,
\]
then $g = \frac{u-u_{new}}{\tau^2}$ has vanishing mean, and by induction also $v$ has mean zero in all time steps, leading to enclosed volume conservation.

\subsection{Vector-valued extension}\label{appendix vector-valued}

In Section \ref{section graphs}, we consider the minimization of a Ginzburg-Landau energy for vector-valued inputs $u$. Algorithms \ref{algorithm gd} and \ref{algorithm fista} easily extend to this situation. We present the modification in the setting of graph-PDEs.

The Dirichlet energy for vector-valued functions $u:\Gamma\to\R^k$ is
\[
\sum_{j=1}^k E_{DR}(u_j) = \sum_{j=1}^k u_j^TLu_j = tr\big(U^TLU\big)
\]
where $U \in \R^{N\times k}$ is the matrix with columns $u_i$. For any $R>0$ and $W_R$ as in Section \ref{section ws}, the potential
\[
W(u) = \sum_{j=1}^k W_R(u_j)
\]
has the convex part $\sum_{j=1}^k u_j^2 = \|u\|_2^2$. However, the potential has $2^k$ wells in $\{-1,1\}^k$, most of which do not support unique classification. To address these issues, we consider
\[
W(u) = \begin{cases} \frac14 \sum_{j=1}^k W_R(2u_j +1) &\text{if }\sum_{j=1}^k u_j = 1\\ +\infty &\text{otherwise}\end{cases}
\]
instead. While slightly more challenging than a potential with purely quadratic convex part, $W$ admits a very similar computational approach. Namely, for a gradient descent step, we select
\[
U_{n+1} = \argmin_U \left(\frac1{2\tau}\|U-U_{n+1}\|^2 + \frac12\,tr(U^TLU) + \frac1{\eps^2}\,\sum_{i=1}^N\sum_{j=1}^k U_{ij}^2 + W_{R,conc}'(U_{n, ij})\,U_{ij}\right)
\]
under the constraints that
\begin{enumerate}
    \item $U_i = y_i$ for $i \in I_{bd}$ where $U_i = (U_{i1},\dots, U_{ik})$ and $y_i$ is the one-hot encoding of the known label of the data point $x_i$ and
    \item $\sum_{j=1}^k U_{ij} = 1$ for all $i=1,\dots, N$.
\end{enumerate}
Again, the Lagrange multiplier theorem tells us that
\[
\frac{U_{int}- U_{n,int}}\tau + L_{int}U_{int} + F_{bd} + \frac2{\eps^2}\,U + W'_{R,conc}(U_n) = \lambda \otimes \vec 1
\]
for $\vec 1 = (1,\dots, 1)$ in $\R^k$, $F_{bd} = L_\times U_{bd}$ and some vector of Lagrange multipliers $\lambda \in \R^{N-N_{bd}}$. We can solve the linear system 
\[
\left(1 + \frac{2\tau}{\eps^2}\right) U + \tau\,L_{int}U = U_{n, int} - \frac{\tau}{\eps^2}\,W_{R,conc}'(U_{n,int}) - 2\tau\,F_{bd} + \tau \,\lambda \otimes \vec 1
\]
by solving the equation without Lagrange multiplier
\[
\left(1 + \frac{2\tau}{\eps^2} + \tau\,L_{int}\right) \tilde U_{n+1} = U_{n, int} - \frac{\tau}{\eps^2}\,W_{R,conc}'(U_{n,int}) - 2\tau\,F_{bd}
\]
for $\tilde U_{n+1, int}$ per time-step and modify the $i$-th row as
\[
U_{n+1, i} = \tilde U_{n+1, i} + \frac{1-\sum_{j=1}^k \tilde U_{n+1,ij}}k \,\vec 1_k =: \tilde \lambda \otimes \vec 1_k
\]
since then 
\[
\sum_{j=1}^k U_{n+1,ij} = \sum_{j=1}^k\left(\tilde U_{n+1, ij} + \frac{1-\sum_{j=1}^k \tilde U_{n+1,ij}}k\right) = 1
\]
for all $i$ and
\begin{align*}
\left(1 + \frac{2\tau}{\eps^2} + \tau\,L_{int}\right) U_{n+1} 
    &= \left(1 + \frac{2\tau}{\eps^2} + \tau\,L_{int}\right) \tilde U_{n+1} + \left(1 + \frac{2\tau}{\eps^2} + \tau\,L_{int}\right) \tilde\lambda \otimes \vec 1_k\\
    &=U_{n, int} - \frac{\tau}{\eps^2}\,W_{R,conc}'(U_{n,int}) - 2\tau\,F_{bd} + \left(\left(1 + \frac{2\tau}{\eps^2} + \tau\,L_{int}\right) \tilde\lambda \right) \otimes \vec 1_k,
\end{align*}
i.e.\ $U_{n+1}$ is the unique solution to the Lagrange-multiplier equation which satisfies $\sum_j U_{n+1,ij} = 1$ for all $i \in I_{int}$.
As before: If the gradient-descent step with magnitude $\tau^2$ in FISTA preserves $\sum_{j=1}^k U_{ij}$ for all $i \in I_{int}$, then both sequences
\[
G_n = \frac{U_{n+1} - U_n}{\tau^2} \qquad\text{and}\quad V_{n+1} = \rho(V_n - \tau G_n)
\]
satisfy $\sum_{j=1}^k G_{n,ij} = \sum_{j=1}^k V_{n,ij} = 0$ for all $i\in I_{int}$ and $n\in \mathbb N$ by induction. Thus, with minimal additional computational effort, the problem falls into the same framework.

\bibliographystyle{alpha}
\bibliography{./bibliography.bib}

\end{document}